\documentclass{amsart}

\usepackage[T1]{fontenc}
\usepackage[utf8]{inputenc}
\usepackage{framed}
\usepackage[english]{babel}
\usepackage{amssymb}
\usepackage{amsthm}
\usepackage{amsmath}
\usepackage{stmaryrd,mathtools}
\usepackage{bbm}
\usepackage{tikz}
\usepackage{tikz-cd}
\usetikzlibrary{matrix,arrows,decorations.pathmorphing}
\usepackage{pdfpages}
\usepackage{graphicx}

\newtheorem{pro}{Proposition}
\numberwithin{pro}{subsection}
\newtheorem{df}{Definition}
\numberwithin{df}{subsection}
\newtheorem{thm}{Theorem}
\numberwithin{thm}{subsection}
\newtheorem{lem}{Lemma}
\numberwithin{lem}{subsection}
\newtheorem{cor}{Corollary}
\numberwithin{cor}{subsection}
\newtheorem{rem}{Remark}
\numberwithin{rem}{subsection}

\newcommand{\V}{\mathcal{V}}

\newcommand{\R}{\mathbb{R}}
\newcommand{\Z}{\mathbb{Z}}

\newcommand{\C}{\mathbb{C}}

\newcommand{\G}{\mathcal{G}}

\newcommand{\Lie}[1]{\mbox{\sf #1}}

\newcommand\emptyarg{{}\cdot{}}

\makeatletter
\newcommand*{\rom}[1]{\expandafter\@slowromancap\romannumeral #1@}
\makeatother

\graphicspath{{./MFIllustrationer/}}

\begin{document}
\title[Construction of Modular Functors from Modular Categories]{Construction of Modular Functors\\from Modular Tensor Categories}
\date{}

\author{J{\o}rgen Ellegaard Andersen}
\address{Center for Quantum Geometry of Moduli Spaces\\
        University of Aarhus\\
        DK-8000, Denmark}
\email{andersen{\@@}qgm.au.dk}
\author{William Elbæk Petersen}
\address{Center for Quantum Geometry of Moduli Spaces\\
        University of Aarhus\\
        DK-8000, Denmark}
\email{william{\@@}qgm.au.dk}

\begin{abstract}  In this paper we follow the constructions of Turaev's book \cite{Tu} closely, but with small modifications,  to construct  a modular functor, in the sense of Kewin Walker, from any modular tensor category. We further show that this modular functor has duality and if the modular tensor category is unitary, then the resulting modular functor is also unitary.
\end{abstract}

\thanks{Supported in part by the center of excellence grant ``Center
  for quantum geometry of Moduli Spaces'' from the Danish National
  Research Foundation (DNRF95).}
  
\maketitle

\section{Introduction} Following Turaev's book \cite{Tu} closely, we provide in this paper, for any modular tensor category $\mathcal{V}$ a construction of a modular functor $Z_{\mathcal{V}}$ subject to the axioms formulated by Kevin Walker in \cite{W} and use in \cite{AU2,AU3,AU4}.

A labeled marked surface is a closed oriented surface $\mathbf{\Sigma}$ endowed with a finite set of distinguished points equipped with a direction as well as a label from a finite set $\Lambda.$ Moreover $\mathbf{\Sigma}$ is equipped with a Lagrangian subspace of its first homology group. A modular functor associates to any labeled marked surface $\Sigma$ a module called its module of states. 
See section \ref{AxiomsforMF} below where we spell out the axioms for a modular functor in all details.

Given a modular tensor category $(\mathcal{V},(V_i)_{i\in I})$, Turaev constructs a $2$-DMF in \cite{Tu}. Taking the label set to be $\Lambda=I,$ we simply use the modular functor provided by Turaev, and provide natural identifications between certain modules of states to make up for the difference between Turaev's axioms for a $2$-DMF and then Walker's axioms for a 2-DMF. To do this one needs to fix isomorphisms
\begin{equation}\label{qi}
 q_i : V_{i^*} \rightarrow \left( V_i \right)^*.
 \end{equation}

We first obtain the following result (Theorem \ref{MT} and \ref{QIT}).

\begin{thm}
For any choices of the isomorphisms (\ref{qi}) we get a modular functor $Z_\mathcal{V}$. For any two choices of the isomorphisms (\ref{qi}) we get quasi-isomorphic modular functors. 
\end{thm}

Here quasi-isomorphism refers to a  notation which is exactly like isomorphism of modular functor, except it allows for scalings of the glueing isomorphisms in a label dependent way, see Definition \ref{qid}. Hence we see that there is a unique quasi-isomorphism class of modular functors associated to every modular tensor category. Two sets of isomorphisms $q_i^{(j)} : V_{i^*} \rightarrow \left( V_i \right)^*$, $j=1,2$ give rise to two strictly isomorphic modular functors if the unique $ u_i \in K^*$ determined by $q^{(2)}_i =  u_ i q^{(1)}_{i}$ satisfies that $u_{i^*} = u_i$.

\begin{thm}
For any choices of the isomorphisms (\ref{qi}) we get a duality structure on the modular functor $Z_\mathcal{V}$. If the modular tensor category is unitary, then we also get a unitary structure compatible with the rest of the structure of the modular functor.
\end{thm}

This is the content of Theorem \ref{DT} and Theorem \ref{HT} below. We emphasise that we do not need to choose the same $q_i$ for the glueing maps and for the duality as discussed in section \ref{scaling section}.
Further, in the compatibility between glueing and duality, duality with it self and duality with the unitary structure, there are projective factors allowed, as detailed in the Definition \ref{duality def} and Definition \ref{unitarity}. First we establish that we can normalise the duality pairing and the unitary pairing, such that it is strictly compatible with glueing. This is done in section \ref{scaling section}.

From this scaling analysis, one sees that the scaling can be separated into a product of two factors, one which only depends on the genus of the surface (see Definition \ref{genus normalized pairing}) and one, which is simply a product of contributions from each of the labels (see equation (\ref{scalenormalization})).  This provides us with what we call the {\em canonical symplectic scaling}, where (\ref{canonical}) in Theorem \ref{Canonical solution} relate the two scalings of the isomorphisms (\ref{qi}), which has the effect that the quantum invariant of the flat unknot labeled by $i$ becomes $\dim(V_i)$ (see equation (\ref{3}), which is the corresponding normalization for the unknot with one negative twist).
The multiplicative factor in the compatibility of duality with duality and unitary pairing with duality becomes in this case negative one raised to the number of {\em symplectic self-dual labels} of a given labeled marked surface (see Definition \ref{symplectic label} and \ref{symplectic multiplicity} and Theorem \ref{Canonical solution} and \ref{Canonical solution w. unitarity}).

In order to analyse if we can find a normalization such that all projective factors in the compatibility between glueing and duality, duality with it self and duality with unitarity can be made unity, which we call {\em strict compatibility}, we introduce the {\em dual fundamental group}  $\Pi(\mathcal{V},I)^*$ of a modular tensor category.

\begin{df}[$\Pi(\mathcal{V},I)^*$]
Let $\Pi(\mathcal{V},I)^*$ consist of the set of functions 
\[
\tilde{\mu}:I \rightarrow K^*
\] that satisfies 
\[
\tilde{\mu}(i)\tilde{\mu}(i^*)=1,
\] and such that 
\[
\tilde{\mu}(i)\tilde{\mu}(j)\tilde{\mu}(k)\not=1,
\] implies 
\[
\text{Hom}(\mathbf{1},V_i\otimes V_j \otimes V_k)=\mathbf{0}.
\]
\end{df}

We call it the dual of the fundamental group due to its similarity with the dual of the fundamental group of a simple Lie algebra as spelled out in section \ref{g}. 

We make the following definition.

\begin{df}
An $\tilde \mu \in \Pi(\mathcal{V},I)^*$ with the property that $\tilde \mu$ takes on the values $\pm 1$ on the self-dual simple objects, in such way that $\tilde \mu$ is $-1$ on the symplectic simple objects and $1$ on the rest of the self-dual simple objects, is called a {\em fundamental symplectic character}. 
\end{df}

We observe that if ${\mathcal V}$ has no symplectic simple objects, then the identity in $\Pi(\mathcal{V},I)^*$ is a fundamental symplectic character. We ask the question if any modular tensor category has such a fundamental symplectic character.

\begin{thm}
If ${\mathcal V}$ has a fundamental symplectic character, then we can arrange that glueing and duality, duality with it self and duality with the unitary paring are strictly compatible.
\end{thm}

This is proven in section \ref{Thedualofthefundamentalgroup}. In section \ref{SU(N)} we provide a fundamental symplectic character for the quantum $\mbox{SU}(N)$ modular tensor category $H^{\rm{SU}(N)}_k$ at the root of unity $q= e^{2\pi i /(k+N)}$ first constructed by Reshetikhin and Turaev for $N=2$ \cite{RT1,RT2} and by Turaev and Wenzl for general $N$ \cite{TW1,TW2}. See also \cite{BHMV1,BHMV2} for a skein theory model of the $N=2$ case and \cite{B} for the general $N$. In section \ref{g}, we provide a fundamental symplectic character for any modular tensor category associated to the quantum group at a root of unity for any simple Lie algebra. Hence we have established

\begin{thm}
Any quantum group at a root of unity gives a modular functor such that glueing and duality, duality with it self and duality with the unitary paring are strictly compatible.
\end{thm}

We thank Henning Haahr Andersen, Christian Blanchet, Jens Carsten Jantzen, Nicolai Reshetikhin and Vladimir Turaev for valuable discussion regarding this paper.

\section{Axioms for a modular tensor category}

For the axioms of a modular tensor category $(\mathcal{V},(V_i)_{i \in I})$ we refer to chapter $\rom{2}$ in \cite{Tu}.  For any modular tensor category, we have an induced involution $^* : I \rightarrow I$, determined by
$$(V_i)^* \cong V_{i^*}.$$
Recall that the ground ring is $K=\text{End}(\mathbf{1})$ in the notation of \cite{Tu}. For an object $V$ we have the important $K$-linear trace operation $\text{tr}:\text{End}(V) \rightarrow K.$ We have the following definition
$\text{dim}(V):=\text{tr}(\text{id}_V)$ and one gets the following identities for all objects $V$
\begin{align*}
\text{dim}(V)&=\text{dim}(V^*),
\end{align*}
We simply write $\text{dim}(V_i)=\text{dim}(i)$ and so for all indices $i \in I$
$$\\ \text{dim}(i) = \text{dim}({i^*}).
$$

\section{Labeled marked surfaces, extended surfaces and marked surfaces}

\subsection{$\Lambda$-Labeled marked surfaces}

Let $\Lambda$ be a finite set equipped with an involution $^\dagger: \Lambda \rightarrow \Lambda$ and a preferred element $0\in \Lambda$ with $0^{\dagger}=0.$ 

We start by recalling that for a closed connected surface $\Sigma$, Poincare duality induce a non-degenerate skewsymmetric pairing
\[
(\emptyarg, \emptyarg): H_1(\Sigma,\Z) \times H_1(\Sigma,\Z) \longrightarrow \Z,
\]
called the intersection pairing. For the rest of this paper, $H_1(\Sigma)$ will mean the first integral homology group. We remark that we could just as well have considered $H_1(\Sigma,\R).$

For any real vector space $W,$ let $P(W):= (W\setminus \{0\})/ \R_{+}.$

We now define the objects of the category of $\Lambda$-labeled marked surfaces.

\begin{df}[$\Lambda$-marked surfaces]
A $\Lambda$-marked surface is given by the following data: $(\Sigma,P,\lambda,V,L).$

Here $\Sigma$ is a smooth oriented closed surface. $P$ is a finite subset of $\Sigma.$ We call elements of $P$ distinguished points of $\Sigma.$ $V$ assigns to any $p$ in $P$ an element $v(p) \in P(T_p \Sigma).$ We say that $v(p)$ is the direction at $p.$ $\lambda$ is an assignment of labels from $\Lambda$ to the points in $P$, e.g. it is a map $P \rightarrow \Lambda.$ We say that $\lambda(p)$ is the label of $p.$ 

Assume $\Sigma$ splits into connected components $\{\Sigma_{\alpha}\}.$ Then $L$ is a  Lagrangian subspace of $H_1(\Sigma)$ such that the natural splitting $H_1(\Sigma) \simeq \oplus_{\alpha} H_1(\Sigma_{\alpha})$ induce a splitting $L \simeq \oplus_{\alpha} L_{\alpha}$ where $L_{\alpha} \subset H_1(\Sigma_{\alpha})$ is a Lagrangian subspace for each $\alpha$.

By convention the empty set $\emptyset$ is regarded as a $\Lambda$-labeled marked surface.
\end{df}

For the sake of brevity, we will refer to a $\Lambda$-labeled marked surface as a labeled marked surface, whenever there is no risk of ambiguities. Now we describe the morphisms of this category.

\begin{df}[Morphisms]
\label{mor}
Let $\mathbf{\Sigma}_i,i=1,2$ be two (non-empty) $\Lambda$-labeled marked surfaces. For $i=1,2,$ write $\mathbf{\Sigma}_i=(\Sigma_i,P_i,V_i,\lambda_i,L_i).$

A morphism is a pair $\mathbf{f}=(f,s)$, where $s$ is an integer, and $f$ is an equivalence class of orientation preserving diffeomorphisms $\phi:\Sigma_1 \overset{\sim}{\longrightarrow} \Sigma_2$ that restricts to a bijection of distinguished points $P_1 \overset{\sim}{\longrightarrow} P_2$ that preserves directions and labels. Two such diffeomorphisms $\phi,\psi$ are said to be equivalent if they are related by an isotopy of such diffeomorphisms.
\end{df}

For a diffeomorphism such as $\phi,$ we will write $[\phi]$ for the equivalence class desribed above. Thus we will sometimes denote a morphism by $([f],s)$ if we want to stress that we are dealing with a pair where the isotopy class is the equivalence class of the diffeomorphism $f.$

Let $\sigma$ be Wall's signature cocycle for triples of Lagrangian subspaces. We now define composition.

\begin{df}[Composition]
\label{comp} Assume that we are given two composable morphisms $\mathbf{f}_1=(f_1,s_1):\mathbf{\Sigma}_1 \rightarrow \mathbf{\Sigma}_2$ and $\mathbf{f}_2=(f_2,s_2):\mathbf{\Sigma}_2 \rightarrow \mathbf{\Sigma}_3.$ We then define:
\[
\mathbf{f}_2 \circ \mathbf{f}_1:= \left(f_2\circ f_1,s_2+s_1-\sigma((f_2\circ f_1)_{\#}(L_1),(f_2)_{\#}(L_2),L_3) \right)
\]
\end{df}

Using properties of Wall's signature cocycle we obtain that the composition operation is associative and therefore we obtain the category of $\Lambda$-labbelled marked surfaces:

\begin{df}[The category of $\Lambda$-labeled marked surfaces]
The category $\mathbf{C}(\Lambda)$ of $\Lambda$-labeled marked surfaces has $\Lambda$-labeled marked surfaces as objects and morphisms as desribed in defintion \ref{mor} and composition as described in definition \ref{comp}. 
\end{df}

There is an easy way to make this category into a symmetric monoidal category.

\begin{df}[The operation of disjoint union]
Let $\mathbf{\Sigma}_1, \mathbf{\Sigma}_2$ be two $\Lambda$-labeled marked surfaces. For $i=1,2,$ write $\mathbf{\Sigma}_i=(\Sigma_i,P_i,V_i,\lambda_i,L_i).$ We define their disjoint union $
\mathbf{\Sigma}_1 \sqcup \mathbf{\Sigma}_2$ to be
\[
(\Sigma_1 \sqcup \Sigma_2, P_1\sqcup P_2, V_1 \sqcup V_2, \lambda_1 \sqcup \lambda_2, L_1 \oplus L_2).
\]
For morphisms $\mathbf{f}_i: \mathbf{\Sigma}_i \rightarrow \mathbf{\Sigma}_3$  we define $\mathbf{f_1\sqcup f_2}$ to be
\[
(f_1\sqcup f_2, s_1+s_2).
\]
We have an obvious natural transformation:
\[
\text{Perm}:\mathbf{\Sigma}_1 \sqcup \mathbf{\Sigma}_2 \rightarrow \mathbf{\Sigma}_2 \sqcup \mathbf{\Sigma}_1
\]
\end{df}

\begin{pro}[$\mathbf{C}(\Lambda)$ is a symmetric monoidal category]
The category of $\Lambda$-labeled marked surfaces is a symmetric monoidal category with disjoint union as product, the empty surface as unit, and Perm as the braiding.
\end{pro}

We now describe the operation of orientation reversal. For an oriented surface $\Sigma$ we let $-\Sigma$ be the oriented surface where we reverse the orientation on each component. For a map $g$ with values in $\Lambda$ we let $g^{\dagger}$ be the map with the same domain and codomain given by $g^{\dagger}(x)=g(x)^{\dagger}.$

\begin{df}[Orientation reversal] Let $\mathbf{\Sigma}=(\Sigma,P,V,\lambda,L)$ be a $\Lambda$-labeled marked surface. Then we define 
\[
-\mathbf{\Sigma}:=(-\Sigma,P,V,\lambda^{\dagger},L)
\]
We say that $-\mathbf{\Sigma}$ is obtained form $\mathbf{\Sigma}$ by reversal of orientation. For a morphism $\mathbf{f}=(f,s)$ we let
\[
-\mathbf{f}:=(f,-s).
\]
\end{df}

\noindent {\bf Remark.} 
{\it We note that we could also have defined the reversal of orientation to also involve changing the sign on the tangent vectors at the marked points. This gives complete equivalent theories, since there is a canonical morphism of labeled marked surfaces, which induces minus the identity at the marked points, and which is the identity on the complement of small disjoint neighbourhoods of the marked points and which locally around each marked point twist half a turn positively according to the surface orientation around the marked point, jet remains the identity near the boundary of the neighbourhood of the marked point.}

Finally we describe the factorization procedure, where we obtain a $\Lambda$-labeled marked surface by cutting along an oriented simple closed curve $\gamma$ whose homology class is in the distinguished Lagrangian subspace, collapsing the resulting two boundary components to points which get labeled by $(i,i^{\dagger})$ in the following way.

\begin{df}[Factorization data]
Factorization data is a triple $(\mathbf{\Sigma},\gamma,i).$ Here $\mathbf{\Sigma}$ is a $\Lambda$-labeled marked surface and $\gamma$ is a smooth, oriented, simple closed curve with a basepoint $x_0$, such that the homology class of $\gamma$ lies in $L.$ Further, $i$ is a element of the labelset  $\Lambda.$

We also say that the pair $(\gamma,i)$ is a choice of factorization data for $\mathbf{\Sigma}.$
\end{df}

\begin{df}[Factorization]
Let $\mathbf{\Sigma}=(\Sigma,P,V,\lambda,L)$ be a $\Lambda$-labeled marked surface with factorization data $(\gamma,i).$ We will define a $\Lambda$-labeled marked surface $\mathbf{\Sigma}_{\gamma}^i.$ We denote the underlying smooth surface by $\Sigma_{\gamma}.$

Cutting along $\gamma$ we get a smooth oriented surface with boundary $\tilde{\Sigma}_{\gamma}$ with two boundary components $\gamma_{-}$ and $\gamma_{+}.$ The orientation of $\gamma$ together with the orientation of $\Sigma$ allows us to define $\gamma_{+}$ to be the component whose induced Stokes orientation agrees with that of $\gamma.$

The underlying smooth surface is given by $ \Sigma_{\gamma}:=\tilde{\Sigma}_{\gamma}/\sim$ where we collapse $\gamma_{-}$ to a point $p_{-}$ and we collapse and we collapse $\gamma_{+}$ to a point $p_+.$ We orient this surface such that $\Sigma \setminus \gamma \hookrightarrow \tilde{\Sigma}_{\gamma}/\sim$ is orientation preserving. The set of distinguised points for $\Sigma_{\gamma}$ is $P\sqcup \{p_{-},p_{+}\}.$ Identifying $P(T_{p_{\pm}} (\tilde{\Sigma}_{\gamma}))$ with $\gamma$, we choose $v(p_{\pm})$ to be $x_0.$ We extend the labelling $\lambda$ by labelling $p_{+}$ by $i$ and $p_{-}$ by $i^\dagger.$

There is a topological space $X$ given by identifying $p_{-}$ and $p_{+}.$ Clearly this space is naturally homeomorphic to $\Sigma/ \sim,$ where we collapse $\gamma$ to a point Thus we have quotient maps $q:\Sigma \rightarrow X$ and $n: \Sigma_{\gamma} \rightarrow X.$ Define
$L_{\gamma}:= (n_{\#})^{-1}(q_{\#})(L).$ This yield a Lagrangian subspace of $H_1(\Sigma_{\gamma})$ that respect the splitting induced by decomposing $\Sigma_{\gamma}$ into connected components.

We say that $\mathbf{\Sigma}_{\gamma}^i$ is obtained by factorizing $\mathbf{\Sigma}$ along $(\gamma,i).$
\end{df}

There is an inverse procedure that we call gluing.

\begin{df}[Gluing data]
Gluing data consist of a triple $(\mathbf{\Sigma},(p_0,p_1),c ).$ Here $\mathbf{\Sigma}=(\Sigma,P,V,\lambda,L)$ is a $\Lambda$-labeled marked surface with $p_0,p_1 \in P$ such that $\lambda(p_{0})=\lambda(p_{1})^\dagger$ and $c:P(T_{p_0} \Sigma) \overset{\sim}{\longrightarrow} P(T_{p_+} \Sigma)$ is an orientation reversing projective linear isomorphism mapping $v(p_{0})$ to $v(p_{1}).$

We also say that $(p_0,p_1,c)$ determine gluing data for $\mathbf{\Sigma}$ and that $(p_0,p_1)$ is subject to gluing.
\end{df}

As we are dealing with ordered pairs $(p_0,p_1)$ we will sometimes speak of $p_0$ as the prefered point.

\begin{df}[Gluing] Assume we are given gluing data $(\mathbf{\Sigma},(p_0,p_1),c ).$ We will define a $\Lambda$-labeled marked surface $\mathbf{\Sigma}_{c}^{p_0,p_1}.$ We denote the underlying smooth surface by $\Sigma_{c}^{p_0,p_1}.$

Blow up $\Sigma$ at $p_0,p_1$ and glue in $P(T_{p_0} \Sigma)$ and $P(T_{p_1} \Sigma)$ to obtain a smooth oriented surface with boundary, that as a set can be canonically identified with
\[
(\Sigma \setminus \{p_0,p_1\} )\sqcup P(T_{p_0}\Sigma)\sqcup P(T_{p_1} \Sigma)
\]

Now identify the two boundary components through $x\sim c(x).$ This yield a smooth oriented surface, that will be the underlying surface of $\mathbf{\Sigma}_{c}^{p_0,p_1}.$ As distinguished points, directions and labels, we simply take those from $\mathbf{\Sigma}.$

Let $X$ be the topological space obtained from $\Sigma$ by identifying $p_0$ with $p_1.$ We have continuous maps $q: \Sigma \rightarrow X$ and $n: \Sigma_{c}^{p_1,p_2} \rightarrow X.$

Set $L_{c,p_0,p_1}:= (n_{\#})^{-1}(q_{\#})(L).$ This is a Lagrangian subspace of $H_1(\Sigma_{\gamma})$ that respect the splitting induced by decomposing $\Sigma_{\gamma}$ into connected components.

Observe that the homology class of $ P(T_{p_0}\Sigma)$ lies in $L_{c,p_0,p_1}.$
\end{df}

\begin{pro}[Consecutive gluing]
\label{cg}
Assume that two distinct pairs of points $(p_1,p_2,c)$ and $(q_1,q_2,d)$ are subject to gluing. Then there is a canonical diffeomorphism
\[
s^{p_1,p_2,q_1,q_2}:(\mathbf{\Sigma}_c^{p_1,p_2})_d^{q_1,q_2} \rightarrow (\mathbf{\Sigma}_d^{q_1,q_2})_c^{p_1,p_2}.
\]
In an abuse of notation we will also write $s^{p_1,p_2,q_1,q_2}$ for the induced morphism of labeled marked surfaces given by $([s^{p_1,p_2,q_1,q_2}],0).$
\end{pro}

We recall that any two orientation reversing self-diffeomorphisms of $S^1$ fixing a basepoint are isotopic among diffeomorphisms fixing this basepoint. Therefore we wish to detail the independence of the choice of $c$ in the gluing construction.

\begin{pro}[Gluing independent of $c.$] \label{glui} Assume we are given a $\Lambda$-labeled marked surface $\Sigma$ and two pairs of gluing data $(p_0,p_1,c_1)$ and $(p_0,p_1,c_2).$ Then there is an orientation preserving diffeomorphism $f: \Sigma \rightarrow \Sigma$ that induce the identity on $(P,V,\lambda,L)$ and such that $c_1 \circ df=df \circ c_2.$ Moreover $f$ can be choosen to induce the identity morphism $(\text{id},0)$ on $\mathbf{\Sigma}.$ Any two such $f$ induce the same morphism of $\Lambda$-labeled marked surfaces, and therefore we have a canonical identification morphism $\tilde{f}(c_1,c_2): \mathbf{\Sigma}_{c_1}^{p_0,p_1} \rightarrow \mathbf{\Sigma}_{c_2}^{p_0,p_1}$ given by the pair
$([f],0).$
\end{pro}

It follows from this that in order to specify gluing, it will suffice to specity an ordered pair $(p_0,p_1)$ with $\lambda(p_0)=\lambda(p_1)^\dagger.$ 

\begin{pro}[Functoriality of gluing]
\label{Fofg}
Let $\mathbf{\Sigma}_i$ for $i=1,2$ be $\Lambda$-labeled marked surfaces. Assume $(p_0^i,p_1^i)$ are subject to gluing for $i=1,2.$ Consider any morphism $\mathbf{f}=([f],s):\mathbf{\Sigma}_1 \rightarrow \mathbf{\Sigma}_2$ with $f(p_0^1)=p_0^2$ and $f(p_1^1)=p_1^2.$ Let $c: P(T_{p_0^1} \Sigma_1) \rightarrow P(T_{p_1^1} \Sigma_1)$ be orientation reversing. Let $c':=df \circ c \circ df^{-1}: P(T_{p_0^2} \Sigma_2) \rightarrow P(T_{p_0^2} \Sigma_2).$ This data induce a morphism

\[
\mathbf{f}'=([f'],s): \left(\mathbf{\Sigma}_1 \right)_{c}^{p_0^1,p_1^2} \longrightarrow \ \left( \mathbf{\Sigma}_2\ \right)_{c'}^{p_0^2,p_1^2}
\] compatible with $\mathbf{f}.$
\end{pro}

\subsection{Extended surfaces}\mbox{ }

We now describe the category of extended surfaces following Turaev \cite{Tu}. Observe that this is only defined relative to a modular tensor category $(\mathcal{V},(V_i)_{i\in I}).$ We recall that an orientation for a closed topological surface $\Sigma$ is a choice of fundamental class in $H^2(\Sigma_{\alpha},\Z)$ of each component $\Sigma_{\alpha}.$ A degree $1$-homeomorphism between oriented closed surfaces is a homeomorphism that respects this choice. We recall that an arc $\gamma \subset \Sigma$ is a topological embedding of $[0,1].$ 

\begin{df}[Extended surfaces]
An $e$-surface $\Sigma$ is given by the following data: $(\Sigma,(\alpha_i),(W_i,\mu_i), L).$

Here $\Sigma$ is an oriented closed surface, $(\alpha_i)$ is a finite collection of disjoint oriented arcs. To each arc $\alpha_i$ we have an object $W_i$ of $\mathcal{V}$ and a sign $\mu_i \in \{\pm 1\}.$ The pair $(W_i, \mu_i)$ is called the marking of $\alpha_i.$ Finally, $L$ is a Lagrangian subspace of $H_1(\Sigma,\R).$

By convention $\emptyset$ is an $e$-surface.
\end{df}

We now describe the arrows.

\begin{df}[Weak extended homeomorphisms and their composition]
Let $\Sigma_1, \Sigma_2$ be two $e$-surfaces. A weak $e$-homoemorphism $f: \Sigma_1 \rightarrow \Sigma$ is a degree $1$-homeomorphism between the underlying topological surfaces $\Sigma_1 \rightarrow \Sigma_2$ that induce an orientation and marking preserving bijection between their distinguished arcs. 

An $e$-homeomorphism $f: \Sigma_1 \rightarrow \Sigma$ is a weak $e$-homeomorphism that induce an isomorphism of distinguished Lagrangian subspaces: $f_{\#}:L_1 \rightarrow L_2.$ 

We obeserve that the class of weak $e$-homeomorphisms is closed under compostion, and that this is also the case for $e$-homeomorphisms.
\end{df}

Thus we have the category of extended surfaces based on $(\mathcal{V},(V_i)_{i \in I}).$ 
\begin{df}[The category of extended surfaces based on $\mathcal{V}$]
The category of extended surfaces based on $\mathcal{V}$ has $e$-surfaces as objects and weak $e$-homeomorphisms as morphisms. We denote it by $\mathbf{E}(\mathcal{V}).$
\end{df}

As above we wish to make this into a symmetric monoidal category with an orientation reversal.

\begin{df}[Disjoint union of $e$-surfaces]

Let $\Sigma_1=(\Sigma_1, (\alpha_i), (W_i,\mu_i),L)$ and $\Sigma_2=(\Sigma_2, (\beta_j), (Z_j, \eta_j),L')$  be two $e$-surfaces. We define $\Sigma_1 \sqcup \Sigma_2$ to be
\[
(\Sigma_1 \sqcup \Sigma_2, (\alpha_i \sqcup \beta_j), (W_i^, \mu_i)\sqcup (Z_j, \eta_j), L \oplus L').
\]
For a pair of (weak) morphism $f_i: \Sigma_i \rightarrow \Sigma_3$ we observe that $f_1\sqcup f_2$ is a (weak) morphism.

We have an obvious natural transformation:
\[
\text{Perm}:\Sigma_1 \sqcup \Sigma_2 \rightarrow \Sigma_2 \sqcup \Sigma_1
\]
\end{df}

\begin{pro}[$\mathbf{E}(\mathcal{V})$ is a symmetric monoidal category]
The category of extended surfaces is a symmetric monoidal category with disjoint union as product, the empty surface as unit, and Perm as the braiding.
\end{pro}

\begin{df}[Orientation reversal for $e$-surfaces]
\label{Orientation reversal for e}
Consider an extended surface $\Sigma=(\Sigma,(\alpha_i),(W_i,\mu_i), L).$ We define $-\Sigma$ to be
\[
(-\Sigma,(-\alpha_i),(W_i, -\mu_i),L).
\]
That is, we reverse the orientation on each component, reverse the orientation of arcs, keep the labels, multiply all signs by $-1,$ and keep the Lagrangian subspace. We observe that any (weak) $e$-homeomorphism $f: \Sigma_1 \rightarrow \Sigma_2$ yield a (weak) morphism $f: -\Sigma_1 \rightarrow -\Sigma_2.$
\end{df}

\subsection{Marked surfaces}

Finally we describe the category of marked surfaces\footnote{Not to be confused with $\Lambda$-labeled marked surfaces}. This is defined relative to a monoidal class. That is, a class $C$ together with a strictly associative operation $C\times C \rightarrow C$ and a unit $1$ for this operation. Again we here follow Turaev \cite{Tu}.

\begin{df}[Marked surface over $C$]
A marked surface (over $C$) is a compact oriented surface $\Sigma$ endowed with a Lagrangian subspace of $H_1(\Sigma,\R)$ and such that each connected component $X$ of $\partial \Sigma$ is equipped with a basepoint, a sign $\delta$, and an element $V$ of $C$ called the label. The pair $(V,\delta)$ is called the marking of $X.$

By convention $\emptyset$ is an $m$-surface.
\end{df}

Next we desribe the morphisms

\begin{df}[Weak m-homeomorphisms] Let $\Sigma_1, \Sigma_2$ be two marked surfaces. A weak $m$-homeomorphism $f: \Sigma_1 \rightarrow \Sigma_2$ is an orientation preserving homeomorphism $f$ that that respect the marks of boundary components. An $m$-homeomorphism is a weak $m$-homeomorphism that also preserve the Lagrangian subspaces.

We observe that the class of weak $m$-homoemorphisms is closed under composition.
\end{df}

\begin{df}[The category of marked surfaces over $C$]
The category of marked surfaces over $C$ has $m$-surfaces as objects and weak $m$-homeomorphisms as morphisms. We denote it $\mathbf{M}(C).$
\end{df}

As above this naturally constitute a symmetric monoidal category with disjoint union as the product:

\begin{df}[Disjoint union of marked surfaces]
Let $\Sigma_1, \Sigma_2$ be two $m$-surfaces. Then we define the marked surface $\Sigma_1 \sqcup \Sigma_2$ by declaring that the boundary components naturally inherit basepoints and markings, and equipping it with a Lagrangian subspace of $H_1(\Sigma_1 \sqcup \Sigma_2, {\R})$, by taking the direct sum of Lagrangian subspaces of $\Sigma_1,\Sigma_2.$ If $f_1,f_2$ are (weak) $m$-homeomorphisms, then so is $f_1 \sqcup f_2$ is a (weak) $m$-homeomorphism. We have an natural transformaion 
\[
\text{Perm}: \Sigma_1 \sqcup \Sigma_2 \rightarrow \Sigma_2 \sqcup \Sigma_1.
\]
\end{df}

\begin{pro}[$\mathbf{M}(C)$ is a symmetric monoidal category]
The category $\mathbf{M}(C)$ of marked surfaces (over C) is a symmetric monoidal category with disjoint union as product, the empty surface as unit, and Perm as the braiding.
\end{pro}

\begin{df}[Gluing] Let $\Sigma$ be an $m$-surface. Assume that there are two components $X,Y$ with the same label, but with opposite sign. Then there is a (unique up to isotopy) basepoint preserving orientation-reversing homeomorphism $c:X \rightarrow Y.$ Then the quotient $\Sigma'=\Sigma/ \sim$ where $ x \sim c(x)$ is naturally an oriented compact surface. The quotient map $q: \Sigma \rightarrow \Sigma'$ yields a bijection $\partial \Sigma'  \sim \partial \Sigma \setminus X \cup Y.$ Using this, we equip each component of $\partial \Sigma'$ with a basepoint and a marking. Finally, equip $\Sigma'$ with the Lagrangian subspace that is the image of the Lagrangian subspace of $\Sigma$ under $q_{\#}.$ Denote the resulting $m$-surface by
\[
\Sigma/ [X=Y]_c.
\]
\end{df}

\begin{pro}[Functorial property of gluing of $m$-surfaces]
Let $\Sigma$ be an $m$-surface. Assume $X,Y \subset \partial \Sigma$ are two boundary components subject to gluing. Let $x: X \rightarrow Y$ be basepoint preserving and orientation reversing. Let $f: \Sigma \rightarrow \Sigma'$ be a (weak) $m$-homeomorphism.

Then $X'=f(X),Y'=f(Y)\subset \partial \Sigma'$ are subject to gluing and the map $c'$ given by $f\circ c \circ f^{-1}: X' \rightarrow Y'$ is orientation reversing and basepoint preserving. There is a unique (weak) homeomorphism $f_c: \Sigma /[X=Y]_c \rightarrow \Sigma' / [X'=Y']_{c'}$ inducing a commutative diagram:

\begin{center}
\begin{tikzcd}[swap]
\Sigma \arrow{r}[swap]{f}
\arrow{d}{q}
& \Sigma' \arrow{d}[swap]{q}
\\ \Sigma/ [X=Y]_c \arrow{r}[swap]{f_c} & \Sigma'/[X'=Y']_{c'}
\end{tikzcd}
\end{center}
Here the vertical maps are the quotient maps.
\end{pro} 

\begin{rem}\label{desc} Let $\mathcal{M}'(C)$ be the category with the same objects as $\mathbf{M}(C),$ but where morphisms are equivalence classes of weak $m$-homeomorphisms, where two parallel weak $m$-homeomorphisms are equivalent if they are isotopic through weak $m$-homeomorphisms. We recall that the $2$-DMF $\mathcal{H}_{\mathcal{V}}$ defined in chapter $\rom{5}$ of \cite{Tu} descends to $\mathcal{M}'(C)$ in the sense that if $f,g$ are two equivalent weak $m$-homeomorphisms, then we have the identity $\mathcal{H}(f)=\mathcal{H}(g).$
\end{rem}

\section{Axioms for a modular functor}

\label{AxiomsforMF}

We now recall Kevin Walker's axioms for a modular functor as they are given and used in \cite{AU2,AU3,AU4}. For the Turaev axioms of a modular functor, we refer to chapter $\rom{5}$ in \cite{Tu}. We assume familiarity with the notion of symmetric monoidal functors. Roughly speaking, a symmetric monoidal functor between symmetric monoidal categories $(C,\otimes , e)\rightarrow (D,\otimes' , e')$ is a triple $(F,F_2,f)$ where $F: C \rightarrow D$ is a functor, $F_2$ is a family of morphisms $F_2:F(a)\otimes' F(b) \rightarrow F(a\otimes b)$ and $f$ is a morphism $f:e' \rightarrow F(e).$ For the precise formulation of the axioms we refer to \cite{M}. For brevity we will write $F=(F,F_2,f).$ If $F_2,f$ are allways isomorphisms, we say that $F$ is a strong monoidal functor.

\subsection{The Walker axioms for a modular functor}

Let $\Lambda=(\Lambda,^\dagger,0)$ be a label set. Let $K$ be a commutative ring (with unit). Let $\mathbf{P}(K)=\text{Proj}(K)$ be the category of finitely generated projective $K$-modules. We recall that this is a symmetric monoidal category with the tensor product over $K$ as product, and $K$ as unit.

\begin{df}[Modular functor $V$ based on $\Lambda$ and $K$]
\label{defofmf}
A modular functor based on a label set $\Lambda$ and a commutative ring $K$ is a pair $(V,g)$ consisting of a strong monoidal functor $V$
\[
V: \mathbf{C}(\Lambda) \rightarrow \mathbf{P}(K),
\]
and a gluing isomorphism $g$ with the properties described below.

\paragraph{Gluing axiom.}\label{gaxiom} Assume $(p_1,p_2,c)$ is gluing data for a labeled marked surface $\mathbf{\Sigma}.$ For any $\lambda \in \Lambda,$ Let $\mathbf{\Sigma}(\lambda)$ be the labeled marked surface identical to $\mathbf{\Sigma}$ except for the fact that $p_1$ is labeled with $\lambda$ and $p_2$ is labeled with $\lambda^{\dagger}.$ Then $(p_1,p_2,c)$ is gluing data for $\mathbf{\Sigma}(\lambda).$ We demand that there is a specified isomorphism
\begin{equation}
g: \bigoplus_{\lambda \in \Lambda} V(\mathbf{\Sigma}(\lambda)) \overset{\sim}{\longrightarrow} V(\mathbf{\Sigma}_c^{p_1,p_2}).
\end{equation}
Let $g_\lambda$ be the restriction of $g$ to $V(\mathbf{\Sigma}(\lambda)).$ If the context is clear, we will simply write $g$ for this restriction, and suppress $\lambda$ from the notation. If we wish to stress the gluing map $c$, we will write $g^c.$ The gluing isomorphism is subject to the four axioms below.

\paragraph*{(i).} The isomorphism should be associative in the following sense. Assume that $(q_1,q_2,d)$ is another pair subject to gluing. For any pair $(\lambda,\mu) \in \Lambda^2$ let $\mathbf{\Sigma}(\lambda,\mu)$ be the labeled marked surface identical to $\mathbf{\Sigma}$ except that $p_1$ is labeled with $\lambda,$ $p_2$ is labeled with $\lambda^{\dagger}$, $q_1$ is labeled with $\mu$ and $q_2$ is labeled with $\mu^{\dagger}.$ Then the following diagram is commutative:

\begin{equation}
\begin{tikzcd}[swap]
V(\mathbf{\Sigma}(\lambda,\mu))\arrow{r}[swap]{g_\mu}
\arrow{d}{g_\lambda}
& V((\mathbf{\Sigma}_d^{q_1,q_2})(\lambda)) \arrow{d}[swap]{g_\lambda}
\\ V((\mathbf{\Sigma}_c^{p_1,p_2})(\mu)) \arrow{r}[swap]{s'\circ g_\mu} & V( (\mathbf{\Sigma}_d^{q_1,q_2})_c^{p_1,p_2})
\end{tikzcd}
\end{equation}
Here $s'=V(s^{p_1,p_2,q_1,q_2}),$ where $s^{p_1,p_2,q_1,q_2}$ is as defined in Prop. \ref{cg}.

\paragraph*{(ii).} The isomorphism should be compatible with gluing of morphisms in the following sense. Assume that $\mathbf{f}: \mathbf{\Sigma}_1 \rightarrow \mathbf{\Sigma}_2$ is a morphism such that a pair $(p_0,p_1)$ subject to gluing is taken to the pair $(q_0,q_1).$ Choosing $c$ will induce a morphism $\mathbf{f}': \left(\mathbf{\Sigma}_1\right)_c^{p_0,p_1} \longrightarrow \left(\mathbf{\Sigma}_2\right)_{c'}^{q_0,q_1}$ as in Prop \ref{Fofg}. This should induce a commutative diagram:

\begin{equation}
\begin{tikzcd}[swap]
V(\mathbf{\Sigma}_1) \arrow{r}[swap]{g}
\arrow{d}{V(\mathbf{f})}
& V((\mathbf{\Sigma}_1)_c^{p_0,p_1}) \arrow{d}[swap]{V(\mathbf{f}')}
\\ V(\mathbf{\Sigma}_2) \arrow{r}[swap]{g} & V((\mathbf{\Sigma}_2)_{c'}^{q_0,q_1})
\end{tikzcd}
\end{equation}

\paragraph*{(iii).} The isomorphism should be compatible with disjoint union in the following way. Assume that $(p_0,p_1,c)$ is gluing data for $\mathbf{\Sigma}_1.$ For any $\mathbf{\Sigma}_2,$ we see that $(p_0,p_1,c)$ is also a choice of gluing data for $\mathbf{\Sigma}_1 \sqcup \mathbf{\Sigma}_2,$ and that there is a canonical morphism $\mathbf{\iota}=(\iota,0): (\mathbf{\Sigma}_1)_{c}^{p_1,p_2} \sqcup \mathbf{\Sigma}_2 \longrightarrow \left(\mathbf{\Sigma}_1 \sqcup \mathbf{\Sigma}_2\right)_{c}^{p_0,p_1}.$ This should induce a commutative diagram

\begin{equation}\label{multofg}
\begin{tikzcd}[swap]
V(\mathbf{\Sigma}_1 \sqcup \mathbf{\Sigma}_2) \arrow{r}[swap]{g}
& V((\mathbf{\Sigma}_1\sqcup \mathbf{\Sigma}_2)_c^{p_0,p_1})
\\ V(\mathbf{\Sigma}_1) \otimes V(\mathbf{\Sigma}_2) \arrow{r}[swap]{g\otimes 1} \arrow{u}[swap]{V_2} & V((\mathbf{\Sigma}_1)_{c}^{p_0,p_1})\otimes V(\mathbf{\Sigma}_2) \arrow{u}{V(\iota)\circ V_2}
\end{tikzcd}
\end{equation}

\paragraph*{(iv).} The isomorphism should be independent of the the gluing map $c$ in the following way. Assume a pair of points $(p_0,p_1)$ in $\mathbf{\Sigma}$ is subject to gluing. Assume that $c_1,c_2: P(T_{p_0} \Sigma
) \rightarrow P(T_{p_1}\Sigma)$ are two gluing maps. Consider the identification morphism $\tilde{f}(c_1,c_2): \mathbf{\Sigma}_{c_1}^{p_0,p_1} \rightarrow \mathbf{\Sigma}_{c_2}^{p_0,p_2}$ as in Prop \ref{glui}. This should induce a commutative diagram:
\begin{equation}
\begin{tikzcd}[swap]
  V(\mathbf{\Sigma}) \arrow{r}[swap]{g^{c_1}}
       \arrow{rd}{g^{c_2}}   
  & V(\mathbf{\Sigma}_{c_1}^{p_0,p_1}) \arrow{d}[swap]{V(\tilde{f}(c_1,c_2))}
  \\ &  V (\mathbf{\Sigma}_{c_2}^{p_0,p_1})
\end{tikzcd}
\end{equation}

\paragraph{Once punctured sphere axiom.} For any $\lambda \in \Lambda$ consider a sphere with one distinghuished point: $\mathbf{\Sigma}_{\lambda}=(S^2,\{p\},\{v\},\{\lambda\},0).$ We demand that:
\begin{equation} \label{oncepunctured}
V(\mathbf{\Sigma}_0) \simeq \begin{cases} & K \ \text{if $\lambda=0$}
\\& 0 \ \text{if $\lambda\not=0$} \end{cases}  .
\end{equation}

\paragraph{Twice punctured sphere axiom.} For any ordered pair $(\lambda,\mu)$ in $\Lambda,$ consider a sphere with two distinguished points $\mathbf{\Sigma}_{\lambda,\mu}= (S^2,\{p_1,p_2\},\{v_1,v_2\},\{\lambda,\mu\},0).$ We demand that:
\begin{equation} \label{twicepunctured}
V(\mathbf{\Sigma}_{\lambda,\mu}) \simeq \begin{cases} & K \ \text{if $\mu=\lambda^{\dagger}$}
\\& 0 \ \text{if $\mu\not=\lambda^{\dagger}$} \end{cases}  .
\end{equation}
\end{df} 
 
 We stress that the isomorphisms given in (\ref{oncepunctured}) and (\ref{twicepunctured}) are not part of the data of a modular functor. Only the existence of such isomorphisms are required.
 
\section{Construction of a modular functor $Z_\mathcal{V}.$}

\subsection{The symmetric monoidal functor}

From now on, we consider a modular tensor category $(\mathcal{V},(V_i)_{i\in I})$ and take $\Lambda=I$ and $^\dagger = ^*$. We let $K$ be the commutative ring $\text{End}(\mathbf{1}),$ where $\mathbf{1}$ is the unit for the tensorproduct in $\mathcal{V}.$

\begin{pro}[Existence of a strong monoidal functor $\mathbf{C}(I) \rightarrow \mathbf{M}(\mathcal{V})$]
\label{C(I) to M(V)}
Consider a modular tensor category $(\mathcal{V},(V_i)_{i\in I}).$ Let $\Lambda =I,$ $^\dagger = ^*$  and let $C= \mathcal{V}$ considered as a monoidal class. There is a strong monoidal functor from the category of $\Lambda$-labeled marked surface into the category $\mathbf{M}'(C).$
\[
\G:\mathbf{C}(I) \rightarrow \mathbf{M}'(C). 
\]
 For a $\Lambda$-labeled marked surface $\mathbf{\Sigma}=(\Sigma,P,V,\lambda,L)$ the marked surface $\Phi(\mathbf{\Sigma)}$ is given as follows. For any distinguished point $p,$ blow up $\Sigma$ at $p.$ That is, the underlying topological surface of $\G(\mathbf{\Sigma})$ is denoted by $\G(\Sigma)$ and is given as follows:
\[
\left( \Sigma \setminus P \bigsqcup_{p \in P} S^1_p \right) / \sim.
\]
Here we glue in the circle $S^1_p$ using smooth coordinates in a neigbourhood of $p.$ The orientation agrees with that on $\mathbf{\Sigma}.$ The direction $v_p$ yields a basepoint on $S^1,$ the label $i \in I$ yields a marking $(V_i,1).$ Collapsing $S^1_p$ to a point at at all $p$ yields a surface $\Sigma'$ that is canonically homeomorphic to $\Sigma.$ Let $\eta$ denote the natural homeomorphism $\Sigma' \rightarrow \Sigma.$ Let $q$ denote the quotient map that collapses any component to a point. The composition $g:=\eta \circ q: \G(\Sigma) \rightarrow \Sigma$ will be an isomorphism on homology, and this provide us with a Lagrangian subspace $L':= {g_{\#}}^{-1}(L).$ 
Given a morphism of labeled marked surfaces $(f,s): \mathbf{\Sigma}_1 \rightarrow \mathbf{\Sigma}_2$ any representative of $f$ naturally induce a weak $m$-homeomorphism $ \mathcal{G}(\mathbf{\Sigma_1}) \rightarrow \mathcal{G}(\mathbf{\Sigma}_2)$ and we let $\mathcal{G}(f,s)$ be the corresponding equivalence class.
\end{pro}

We are now finally ready to define our modular functor. We recall that even though Turaevs axioms for a $2-DMF$ as given in chapter $\rom{5}$ only requires functoriality with respect to $m$-homeomorphisms, it is also defined on weak $m$-homeomorphisms. See section $4.3$ in chapter $\rom{5}.$

\begin{df}[The definition of $Z_{\mathcal{V}}$] 
\label{The definition of Z}
Let $\mathcal{H}_{\mathcal{V}}$ be the $2$-DMF as defined in chapter $\rom{5$} of \cite{Tu} relative to $(\mathcal{V},(V_i)_{i\in I}).$ On the level of objects we define $Z_{\mathcal{V}}$ to be
\[
Z_{\mathcal{V}}:=\mathcal{H}_{\mathcal{V}}\circ \G: \mathbf{C}(\Lambda) \longrightarrow \mathbf{P}(K).
\]
For a morhpism of labeled marked surfaces $(f,s): \mathbf{\Sigma} \rightarrow \mathbf{\Sigma}'$ we define
\[
Z_{\mathcal{V}}(f,s):=(\Delta^{-1} D)^s \mathcal{H}_{\mathcal{V}}(\mathcal{G}(f,s)).
\]
\end{df}

Here $D,\Delta$ are invertible scalars in $K$ to be introduced in section \ref{quantuminvariantsof3mfds} below. We write $Z=Z_{\mathcal{V}}$ and $\mathcal{H}=\mathcal{H}_{\mathcal{V}}.$ We need to adress the issue of functoriality. That is we must verify that $Z(\mathbf{f})\circ Z(\mathbf{g})=Z(\mathbf{f}\circ \mathbf{g})$ for composable morphisms of labeled marked surfaces. Let $V,V'$ be symplectic vector spaces. Recall that Walkers signature cocycle for an ordered triple $(L_1,L_2,L_3)$ of Lagrangian subspaces $L_i \subset V$ coincide with the Maslov index $\mu(L_1,L_2,L_3).$ Recall also that $\mu(L_1,L_2,L_3)=\mu(f(L_1),f(L_2),f(L_3))$ for any symplectomorphism $f:V \rightarrow V'.$ These facts together with remark \ref{eqidefofZ} and lemma $6.3.2$ in chapter $\rom{4}$ of \cite{Tu} easily imply functoriality.

We need to define a gluing isomorphism. We start by observing the following proposition:
\begin{pro}[$\G$ is compatible with gluing]
Assume $\mathbf{\Sigma}$ is a labeled marked surface. Assume we are given gluing data $(p,q,c).$ Assume $p$ is labeled with $i.$ Consider $\Sigma'=\G(\mathbf{\Sigma}).$ If we replace the marking of $X_q$ with $(V_i,-1)$ to obtain a new marked surface $\Sigma''$ then $X_p\subset \partial \Sigma''$ and $X_q\subset \partial \Sigma''$ are subject to gluing. We observe 
\[
\G(\mathbf{\Sigma}_c^{p,q})=\Sigma'' /[X_p \approx X_q].
\]
\end{pro}

We now compare the gluing isomorphism axiom of Walker and the splitting axiom of Turaev more closely.

Turaevs modular functor is subject to the splitting axiom, which means that the gluing homomorphisms provide an isomorphism:
\[
g:\bigoplus_{i \in I} \mathcal{H}\left(\G(\Sigma), \lambda, (V_i,1), (V_i,-1))\right)  \overset{\sim}{\longrightarrow}  Z(\mathbf{\Sigma}_c).
\]
See chapter $\rom{5}$, the splitting axiom on page $246.$ in \cite{Tu}. Comparing with Walker's glueing axiom, we see that the summands are not the same, since there we need an isomorphism
\[
g:\bigoplus_{i \in I} \mathcal{H}\left(\G(\Sigma), \lambda, (V_i,1), (V_{i^*},1)\right)  \overset{\sim}{\longrightarrow}  Z(\mathbf{\Sigma}_c).
\]
Hence we just need to provide isomorphisms between modules of states, where we exchange a marking $(V_{i^*},1)$ with $(V_i,-1).$

To provide these identifications, we first recall that $\mathcal{H}$ is compatible with the operator invariant $\tau^e.$ See the following remark. 

\begin{rem}\label{eqidefofZ} Let $\Sigma$ be a marked surface over $\mathcal{V}.$ We recall that $\mathcal{H}(\Sigma)$ is naturally isomorphic to $\mathcal{T}^e(\overline{\Sigma}),$ where $\overline{\Sigma}$ is an extended surface obtained from $\Sigma$ by gluing in discs with preferred diamemeters, that are taken to be marked arcs. We can therefore use the operator invariant $\tau^{e}$ to obtain morphisms between modules of states. The natural isomorphism $\mathcal{H}(\Sigma) \simeq \mathcal{T}^e(\overline{\Sigma})$ also implies, that for a labeled marked surface $\mathbf{\Sigma}$ we could just as well defined $Z(\mathbf{\Sigma})$ as $\mathcal{T}^e(\widetilde{\mathbf{\Sigma}})$, where $\widetilde{\mathbf{\Sigma}}$ is an extended surface naturally obtained from $\mathbf{\Sigma}.$ Similarly we observe that a morphism $(f,s)$ of $I$-labeled marked surfaces induce an isotopy class $f'$ of weak $e$-homeomorphisms, and that $\mathcal{H}(\mathcal{G}(f,s))\sim \mathcal{T}^e(f').$ 
\end{rem}

Now we provide the needed identifications.

\begin{lem}[The natural transformation $\dot{f}$]
\label{f}
Let $\Sigma$ be an $m$-surface with a boundary component $X_\alpha$ marked with $(V,1).$ Assume that $\Sigma'$ is obtained from $\Sigma$ by replacing the marking $(V,1)$ with $(W,1).$ Assume that $f:V \rightarrow W$ is a morphism.

There is a $K$-linear morphism 
\[
\dot{f}: \mathcal{H}(\Sigma) \rightarrow \mathcal{H}(\Sigma').
\] 
where $\dot{f}$ is induced from the extended three manifold $M=\overline{\Sigma}\times I$ where we think of the bottom as $\overline{\Sigma},$ the top as $\overline{\Sigma'}$, and we provide $M$ with following ribbon graph. For each arc $\beta$ different from the arc $\alpha$ corresponding to $X_{\alpha} \subset \partial \Sigma,$ we put in the identity strand $\beta \times I.$ For $\alpha,$ we put in a coupon colored with $f.$

\end{lem}

\begin{lem}[The natural transformation $h_{\alpha}$]
\label{h}
Let $\Sigma$ be an $m$-surface with a boundary component $X_\alpha$ marked with $(V^*,1).$ Assume that $\Sigma'$ is obtained from $\Sigma$ by replacing the marking $(V^*,1)$ with the marking $(V,-1).$

There is a $K$-linear morphism 
\[
h_{\alpha}: \mathcal{H}(\Sigma) \rightarrow \mathcal{H}(\Sigma').
\] 
The morphism $h_{\alpha}$ is induced from the extended three manifold $M=\overline{\Sigma}\times I$ where we think of the bottom as $\overline{\Sigma},$ the top as $\overline{\Sigma'}$, and we provide $M$ with following ribbon graph. For each arc $\beta$ different from the arc $\alpha$ corresponding to $X_{\alpha} \subset \partial \Sigma,$ we put in the identity strand $\beta \times I.$ For $\alpha,$ we put in a coupon colored with $\text{id}_{V^*}.$

\end{lem}

 If the relevant boundary component is understood, we will simply write $h_{\alpha}=h.$ In section \ref{Proofs of lemmas} below we will give all details of how these two lemmas follow directly from similar statements in \cite{Tu}.

\subsection{The gluing isomorphism}

Let $\mathbf{\Sigma_c}$ be a $\Lambda$-labeled marked surface obtained from $\mathbf{\Sigma}$ by gluing. We must provide an isomorphism
\[
\bigoplus_{i \in I} Z(\mathbf{\Sigma},\lambda, i,i^{\dagger} ) \overset{\sim}{\longrightarrow} Z(\mathbf{\Sigma_c}, \lambda).
\]

For each $i \in I$ fix an isomorphism $q_i:V_{i^*} \rightarrow V_{i}^*.$ We define the gluing isomorphism as follows. For each $i,$ consider the composition
\[
Z(\mathbf{\Sigma},\lambda, i,i^* ) \overset{\dot{q_i}}{\longrightarrow} \mathcal{H}\left(\G(\Sigma), \lambda, (V_i,1), (V_i^*,1))\right) \overset{h}{\longrightarrow} \mathcal{H}\left(\G(\Sigma), \lambda, (V_i,1), (V_i,-1))\right).
\]

Using that $\mathcal{H}$ satisfies the splitting axiom as defined in chapter $\rom{5}$ we see that we have an isomorphism
\[
g:\bigoplus_{i \in I} \mathcal{H}\left(\G(\Sigma), \lambda, (V_i,1), (V_i,-1))\right)  \overset{\sim}{\longrightarrow}  Z(\mathbf{\Sigma}_c).
\]
Thus we can define our gluing isomorphism as follows.
\begin{df}[The gluing isomorphism]
\label{the gluing isomorphism}
We define
\[
  \tilde{g}(q):=g\circ(\oplus_{i\in I}h \circ \dot{q_i}):\bigoplus_{i \in I} Z(\mathbf{\Sigma},\lambda,i,i^*) \overset{\sim}{\longrightarrow}  Z(\mathbf{\Sigma}_c).
  \]
\end{df}
We write $\tilde{g}(q)$ to stress that this depends on the choices of isomorphisms $q_i.$

  \subsection{Main theorem}

We are now ready to state our main theorem. 

\begin{thm}[Main Theorem]
\label{MT}
For any modular tensor category $(\mathcal{V},(V_i)_{i\in I})$ the symmetric modular functor $Z_{\mathcal{V}}$ as given in definition \ref{The definition of Z} together with the gluing isomorphism $\tilde{g}(q)$ as given in definition \ref{the gluing isomorphism} satisfies Walker's axioms of a modular functor based on $I$ and $K$ as given in section \ref{AxiomsforMF}.
\end{thm}

We will sometimes write $Z(q)$ for the modular functor $Z_{\mathcal{V}}$ equipped with the gluing $\tilde{g}(q).$

\section{Proof of the main theorem}

We first state more or less trivial statements about the K-linear morphisms coming from Lemma \ref{f} and \ref{h}.

Recalling the setting and notation of Lemma \ref{f}, it is clear that if $g: \Sigma \rightarrow \tilde{\Sigma}$ is a weak $m$-homeomorphism, then so is $g: \Sigma' \rightarrow \tilde{\Sigma}',$ where $\tilde{\Sigma}'$ is obtained from $\tilde{\Sigma}$ by replacing the marking of $g(X_{\alpha})$ with $(W,1).$

\begin{lem}[The natural transformation $\dot{f}$]
\label{fd}
For each such $g$, $\dot{f}$ induce a commutative diagram.

\begin{center}
\begin{tikzcd}[swap]
\mathcal{H}(\Sigma) \arrow{r}[swap]{\mathcal{H}(g)}
\arrow{d}{\dot{f}}
& \mathcal{H}(\tilde{\Sigma}) \arrow{d}[swap]{\dot{f}}
\\ \mathcal{H}(\Sigma') \arrow{r}[swap]{\mathcal{H}(g)} & \mathcal{H}(\tilde{\Sigma}')
\end{tikzcd}
\end{center}
Moreover, $\dot{f}$ is compatible with disjoint union in the following sence. Assume $\Sigma=\Sigma_1 \sqcup \Sigma_2,$ and $X_{\alpha} \subset \Sigma_2.$ Then the following diagram commute:
\begin{center}
\begin{tikzcd}[swap]
\mathcal{H}(\Sigma) \arrow{r}[swap]{}
\arrow{d}{\dot{f}}
& \mathcal{H}(\Sigma_1)\otimes \mathcal{H}(\Sigma_2) \arrow{d}[swap]{\text{id}\otimes \dot{f}}
\\ \mathcal{H}(\Sigma') \arrow{r}[swap]{} & \mathcal{H}(\Sigma_1) \otimes  \mathcal{H}(\Sigma_2')
\end{tikzcd}
\end{center}
\end{lem}

Recalling the setting and notation of Lemma \ref{h}, we observe that if $g: \Sigma \rightarrow \tilde{\Sigma}$ is a weak $m$-morphism, then so is $g: \Sigma' \rightarrow \tilde{\Sigma}',$ where $\tilde{\Sigma}'$ is obtained from $\tilde{\Sigma}$ by replacing the marking of $g(X_{\alpha})$ with $(V,-1).$

\begin{lem}[The natural transformation $h_{\alpha}$]
\label{hd}

For each such $g$, $h_{\alpha}$ induce a commutative diagram:

\begin{center}
\begin{tikzcd}[swap]
\mathcal{H}(\Sigma) \arrow{r}[swap]{\mathcal{H}(g)}
\arrow{d}{h_{\alpha}}
& \mathcal{H}(\tilde{\Sigma}) \arrow{d}[swap]{h_{\alpha}}
\\ \mathcal{H}(\Sigma') \arrow{r}[swap]{\mathcal{H}(g)} & \mathcal{H}(\tilde{\Sigma}')
\end{tikzcd}
\end{center}
Moreover, $h_{\alpha}$ is compatible with disjoint union in the following sense. Assume $\Sigma=\Sigma_1 \sqcup \Sigma_2,$ and $X_{\alpha} \subset \Sigma_2.$ Then the following diagram commute:
\begin{center}
\begin{tikzcd}[swap]
\mathcal{H}(\Sigma) \arrow{r}[swap]{}
\arrow{d}{h_{\alpha}}
& \mathcal{H}(\Sigma_1)\otimes \mathcal{H}(\Sigma_2) \arrow{d}[swap]{\text{id}\otimes h_{\alpha}}
\\ \mathcal{H}(\Sigma') \arrow{r}[swap]{} & \mathcal{H}(\Sigma_1) \otimes  \mathcal{H}(\Sigma_2')
\end{tikzcd}
\end{center}
\end{lem}

We need to know how the morphisms $\dot{f},h$ relate to the gluing homomorphism provided by Turaev, and we need to know what happens if we apply the $\dot{f},h$ operations consecutively to distinct boundary components. We call a morphism of type $\dot{f}$ or $h$ a coupon morphism.

\begin{lem}[Far commutativity]
\label{farc} Assume that an $m$-surface $\Sigma_2$ is obtained from an $m$-surface $\Sigma_1$ by altering the markings of two distinct boundaries $X_\alpha,X_\beta$ components in one of the two ways described above. Let $q$ be the $K$-morphism $\mathcal{H}(\Sigma_1) \rightarrow \mathcal{H}(\Sigma_2)$ that is obtained from composing the coupon morphism that alters $X_{\alpha}$ with the coupon morphism that alters $X_{\beta}.$ Let $p$ be the $K$-morphism $\mathcal{H}(\Sigma_1) \rightarrow \mathcal{H}(\Sigma_2)$ that is obtained from composing the coupon morphism that alters the labelling of $X_{\beta}$ with the coupon morphism that alters the labelling of $X_{\alpha}.$ Then we have $p=q.$
\end{lem}

\begin{lem}[Compatibility of $\dot{f},h$ with gluing.]
\label{Cfhg}
Assume that $\mathbf{\Sigma_c}$ is obtained from $\mathbf{\Sigma}$ by gluing. Consider a component $X_\beta$ of $\mathbf{\Sigma},$ that is not part of the gluing data. Assume the marking of $\beta$ is altered either by using a morphism $f$ of objects of $\mathcal{V},$ or by replacing $(V^*,1)$ with $(V,-1).$ Then this operation applies to $\mathbf{\Sigma}_c$ as well. Let $r$ denote the resulting isomorphisms of modules. Let $g$ denote the gluing homomorphism. Then
\begin{gather*}
r\circ g=g\circ r
\end{gather*}
\end{lem}

These lemmas will be proven in section \ref{Proofs of lemmas} below.

\begin{proof}[Proof of the Main Theorem]
Since $\mathcal{H}$ is a strong monoidal functor, it is immediate that $Z_\mathcal{V}$ is a symmetric monoidal functor, since it is a composition of strong monoidal functors. Thus it remains to verify the once punctured sphere axiom, the twice punctured sphere axiom, and the gluing axiom.

The once punctured sphere axiom follows directly from Turaev's disc axiom, which is axiom $1.5.5$ in chapter $\rom{5}.$ 

The twice punctured sphere axioms follows directly from the third normalization axiom $1.6.2$ in chapter $\rom{5}$ of \cite{Tu}. It remains to verify the gluing axiom. 

 If $\mathbf{f}=(f,n)$ is a morphism of labeled marked surfaces, we will abuse notation and write $f$ for $\G(\mathbf{f}).$

\noindent {\bf (i)} In the notation of definition \ref{gaxiom} we must prove that the following diagram commutes
\begin{center}
\begin{tikzcd}[swap]
Z(\mathbf{\Sigma}(i,j))\arrow{r}[swap]{\tilde{g}_j}
\arrow{d}{\tilde{g}_i}
& Z((\mathbf{\Sigma}_d^{q_1,q_2})(i)) \arrow{d}[swap]{\tilde{g}_i}
\\ Z((\mathbf{\Sigma}_c^{p_1,p_2})(j)) \arrow{r}[swap]{s'\circ \tilde{g}_j} & Z( (\mathbf{\Sigma}_d^{q_1,q_2})_c^{p_1,p_2})
\end{tikzcd}
\end{center}
Here $s'=Z(s^{p_1,p_2,q_1,q_2}),$ where $s^{p_1,p_2,q_1,q_2}$ is as defined in Prop. \ref{cg}, and $i=\lambda$ and $j=\mu.$ Let $\alpha$ be the relevant distinguished point labeled with $j^*.$ Let $\beta$ be the relevant distinguished point labeled with $i^*.$ As above, let $g$ be the gluing homomorphism provided by Turaev in chapter $\rom{5}$ of \cite{Tu}.  In the following calculation we use that the integer associated to the morphism $s^{p_1,p_2,q_1,q_2}$ is $0.$

Commutativity of the diagram above can be rewritten as the following equation
\begin{equation}
\label{*}
g_i\circ h_{\beta} \circ \dot{q_i}\circ g_j\circ h_{\alpha} \circ \dot{q_j} = s' \circ g_j\circ h_{\alpha} \circ \dot{q_j} \circ g_i\circ h_{\beta} \circ \dot{q_i}.
\end{equation}
Using lemma \ref{Cfhg} and lemma \ref{farc} we see that 
\[
g_i\circ h_{\beta} \circ \dot{q_i}\circ g_j\circ h_{\alpha} \circ \dot{q_j} = g_i \circ g_j \circ h_{\alpha} \circ \dot{q_j}\circ h_{\beta} \circ \dot{q_i}.
\]
Using lemma \ref{Cfhg} we then get that
\[
s' \circ g_j\circ h_{\alpha} \circ \dot{q_j} \circ g_i\circ h_{\beta} \circ \dot{q_i} = s' \circ g_j \circ g_i \circ h_{\alpha} \circ \dot{q_j}\circ h_{\beta} \circ \dot{q_i}.
\]
Using axiom $1.5.4 (ii)$ in chapter $\rom{5}$ of \cite{Tu} we see that
\[
s' \circ g_j \circ g_i=g_i \circ g_j.
\]
Therefore we see that equation (\ref{*}) holds.

\noindent {\bf (ii)} In the notation of definition \ref{gaxiom} we must prove that the following diagram commutes
\begin{center}
\begin{tikzcd}[swap]
Z(\mathbf{\Sigma}_1) \arrow{r}[swap]{\tilde{g}}
\arrow{d}{Z(\mathbf{f})}
& Z((\mathbf{\Sigma}_1)_c^{p_0,p_1}) \arrow{d}[swap]{Z(\mathbf{f}')}
\\ Z(\mathbf{\Sigma}_2) \arrow{r}[swap]{\tilde{g}} & Z((\mathbf{\Sigma}_2)_{c'}^{q_0,q_1})
\end{tikzcd}
\end{center}

This amounts to proving
\begin{equation}
\label{**}
\mathcal{H}(f') \circ g\circ h \circ \dot{q}= g\circ h \circ \dot{q} \circ \mathcal{H}(f).
\end{equation}
Here we use that $\mathbf{f'}$ is equipped with the same integer as $\mathbf{f}.$ Equation (\ref{**}) follows directly from lemma \ref{h}, lemma \ref{f}, and axiom $1.5.4 (i)$ in chapter $\rom{5}$ of \cite{Tu}. Here we use that even though the naturality condition is only formulated for $m$-homeomorphisms in this axiom, Turaev argues in section $4.6$ of chapter $\rom{5}$ that it is also valid for weak $m$-homeomorphisms.

\noindent {\bf (iii)} In the notation of definition \ref{gaxiom} we must prove that the following diagram commutes
\begin{center}
\begin{tikzcd}[swap]
Z(\mathbf{\Sigma}_1 \sqcup \mathbf{\Sigma}_2) \arrow{r}[swap]{\tilde{g}}
& Z((\mathbf{\Sigma}_1\sqcup \mathbf{\Sigma}_2)_c^{p_0,p_1})
\\ Z(\mathbf{\Sigma}_1) \otimes Z(\mathbf{\Sigma}_2) \arrow{r}[swap]{\tilde{g}\otimes 1} \arrow{u}[swap]{Z_2} & Z((\mathbf{\Sigma}_1)_{c}^{p_0,p_1})\otimes Z(\mathbf{\Sigma}_2) \arrow{u}{Z(\iota)\circ Z_2}
\end{tikzcd}
\end{center}
 As the integer associated with the morphism $\iota$ is zero, this takes the following equational form
\begin{equation}
\label{***}
g\circ h \circ \dot{q} \circ \mathcal{H}_2 = \mathcal{H}(\iota) \circ \mathcal{H}_2 \circ (g\circ h \circ \dot{q}\otimes 1)
\end{equation}
Rewrite the RHS as $\mathcal{H}(\iota) \circ \mathcal{H}_2 \circ (g\otimes 1)\circ (h\circ \dot{q}\otimes 1).$ Now use lemma \ref{f} and lemma \ref{h} to rewrite the LHS as $g\circ H_2 \circ (h\circ \dot{q}\otimes 1).$ Now axiom $1.5.4 (iii)$ entails $g\circ H_2=\mathcal{H}(\iota) \circ \mathcal{H}_2 \circ (g\otimes 1).$ This implies equation (\ref{***}).

\noindent {\bf (iv)} In the notation of definition \ref{gaxiom} we must prove that the following diagram commutes\begin{center}
\begin{tikzcd}[swap]
  Z(\mathbf{\Sigma}) \arrow{r}[swap]{\tilde{g}^{c_1}}
       \arrow{rd}{\tilde{g}^{c_2}}   
  & Z(\mathbf{\Sigma}_{c_1}^{p_0,p_1}) \arrow{d}[swap]{Z(\tilde{f}(c_1,c_2))}
  \\ &  Z (\mathbf{\Sigma}_{c_2}^{p_0,p_1})
\end{tikzcd}
\end{center}
As the integer associated with the morphism $\tilde{f}(c_1,c_2)$ is zero, this takes the form
\begin{equation}
\label{****}
\mathcal{H}(\tilde{f}(c_1,c_2))) \circ  g^{c_1}\circ h \circ \dot{q}= g^{c_2} \circ h \circ \dot{q}.
\end{equation}
\begin{lem}\label{<} The morphisms $\mathcal{H}(\tilde{f}(c_1,c_2))) \circ  g^{c_1}$  and $g^{c_2}$ are operator invariants of extended three manifolds that are naturally $e$-homeomorphic through an $e$-homeomorphism commuting with boundary parametrizations. In particular they coincide.
\end{lem}
We see that lemma \ref{<} implies equation (\ref{****}). The lemma will be proven in section \ref{Proofs of lemmas}.
\end{proof}

\section{Review of the TQFT based on extended cobordisms}
\label{Review of the TQFT}
As observed above, $\mathcal{H}$ is defined as 
\[
\mathcal{H}(\Sigma)=\mathcal{T}^e_{\mathcal{V}}(\overline{\Sigma}),
\]
where $\overline{\Sigma}$ is the associated extended surface, and $\mathcal{T}^e_{\mathcal{V}}$ is the modular functor based on the category of extended surfaces and the modular tensor category $\mathcal{V}.$ 

In this section we will give a quick review of the TQFT $(\mathcal{T}^e_{\mathcal{V}}, \tau^e_{\mathcal{V}})$ based on the cobordism theory of extended cobordisms, as defined in chapter $\rom{4}$ of \cite{Tu}. 

We will assume familiarity with the axioms for a TQFT based on a cobordism theory as defined in chapter $\rom{3}$ of \cite{Tu}. We will assume familiarity with the quantum invariant $\tau(M,\Omega)$ of a closed oriented three manifold $M$ containing a ribbon graph $\Omega$ with colors in $\mathcal{V}.$ This invariant is defined in chapter $\rom{2}$ of \cite{Tu}. We will however provide the formula associated to a surgery presentation below, but we will not explain the Reshetikhin-Turaev functor $F_{\mathcal{V}}$ as defined in chapter $\rom{1}$ of \cite{Tu}.

\subsection{Quantum invariants of $3$-manifolds}

\label{quantuminvariantsof3mfds}

We now recall the construction of $\tau(\tilde{M})\in K$ where $\tilde{M}$ is a closed oriented three manifold with a colored ribbon graph inside. Here $\tau(\tilde{M})$ is called the quantum invariant of $\tilde{M}.$ We may assume $\tilde{M}=\partial W,$ where $W$ is a compact oriented four manifold obtained by performing surgery along a framed link $L=\{L_1,...,L_m\}$ in $S^3 =\partial B^4.$ Let $\sigma(L)$ be the signature of the intersection form on $H_2(W,\R).$ Let $\text{Col}(L)$ be the set of all colorings of $L$ by colors in $(V_i)_{i \in I}.$ For any coloring $\lambda$ we let $\Gamma(L,\lambda)$ be the associated colored ribbon graph in $S^3.$ Then $\tau(\tilde{M},\Omega)$ is given by
\begin{equation}
\tau(\tilde{M},\Omega) = \Delta^{\sigma(L)} D^{-\sigma(L)-m-1}\sum_{\lambda \in \text{Col}(L)}\text{dim}(\lambda) F(\Gamma(L,\lambda)\cup \Omega).
\end{equation}
Here $\text{dim}(\lambda)=\text{dim}(\lambda_1)\cdots \text{dim}(\lambda_m)$ and $D$ is given by
\[
D^2 =\sum_{i \in I} \text{dim}(i)^2.
\]  and $\Delta$ is given by
\[
\Delta:=\sum_{i \in I} k_i^{-1} (\text{dim}(i))^2 \in K,
\]
where the $k_i$ are the standard twists coefficients - Turaev denotes them $v_i$ in \cite{Tu}.

\subsection{The TQFT based on decorated cobordisms}

\subsubsection{Modules of states}
Recall the notion of a $d$-surface and decorated type as defined in section $1.1$ of chapter $\rom{4}$ in \cite{Tu}. Recall the notion of a standard $d$-surface and of a parametrized $d$-surface as in section $1.2$ and section $1.3$ of chapter $\rom{4}$ in \cite{Tu}. 

Assume $\Sigma$ is a connected parametrized $d$-surface of topological type $t$ given by $(g;(W_i,\mu_i),...,(W_m,\mu_m))$. Recall the standard $d$-surface of type $t.$ This is denoted by $\Sigma_t.$ These notions can be found in sections $1.1-1.3$ of chapter $\rom{4}.$

For a decorated type $t$ as above, and for $i \in I^g,$ let 
\[
\Phi(t,i):= W_1^{\mu_1} \otimes \cdots \otimes W_m^{\mu_m} \bigotimes_{s=1}^g (V_{i_s}\otimes V_{i_s}^*). 
\]

Here $W^1=W,$ and $W^{-1}=W^*.$ Recall that elements of $\Phi(t,i)$ can be though of as colorings of the ribbon graph $R_t$ sitting inside $\Sigma_t,$ as defined in section $1.2$ of \cite{Tu}. Moreover we define $\mathcal{T}(\Sigma):=\Psi(t)$ where

\[
\Psi(t):= \bigoplus_{i \in I^g} \text{Hom}(\mathbf{1},\Phi(t,i)).
\]

Finally, if $\Sigma$ is not connected, then we define $\mathcal{T}(\Sigma)$ to be the unordered tensor product of the modules of states of the components of $\Sigma.$

\subsubsection{Operator invariants}
We now describe the construction of $\tau(M)$, where $M$ is a decorated cobordism. That is, $M$ is a triple $(M,\partial_{-} M, \partial_{+}M)$ where $\partial_{\pm} M$ are parametrized $d$-surfaces.

For a general decorated type $t$ let $U_t$ be the standard decorated handlebody bounded by $\Sigma_t$ as in section 1.7 of chapter $\rom{4}$ in \cite{Tu}. Equip it with the RH orientation. For an element $x \in \Phi(t,i)$ consider the three manifold with boundary $H(U_t,R_t,i,x).$

Let $U_t^-$ be the image of $U_t$ under the reflection of $\R^3$ in the plane $\R^2\times \{1/2\}.$ We denote this orientation reversing diffeomorphism by $\text{mir}:\R^3 \rightarrow \R^3.$ Equip $U_t^{-}$ with the RH orientation. We recall that they contain certain ribbon graphs denoted $R_t,R_{-t}$ respectively.

Let $f:\Sigma_{t_0} \rightarrow \partial_{-} M$ be a parametrization of a component of $\partial_{-}M.$ We glue in $U_{t_0}$ by gluing $\partial U_{t_0}$ to $\Sigma_{t_1}\times \{0\}$ through $f.$ We do this for all components of $\partial_{-}M.$

 Similarly, for any parametrized component $g: \Sigma_{t_1}\rightarrow \partial_{+}M$ we glue in $U_{t_1}^{-}$ by gluing according to $\text{-g}\circ \text{mir}: \partial (U_{t_1}^{-}) \rightarrow \Sigma_{t_1}.$

This produces a closed oriented three manifold $\tilde{M}$ with a ribbon graph inside, such that choosing an element $x\in \mathcal{T}(\partial_{-}M)$ and an element $y \in \mathcal{T}(\partial_{+} M)^*$ will produce a colored ribbon graph $\Omega(x,y) \subset \tilde{M}.$ This descends to a $K$-linear map $\mathcal{T}(\partial_{-}M)\otimes_K \mathcal{T}(\partial_{+}M)^* \longrightarrow K$ given by
\[
x\otimes y \longrightarrow \tau(\tilde{M},\Omega(x,y)),
\]
where $\tau$ is the quantum invariant defined in chapter $\rom{2}$ of \cite{Tu}.  This pairing induce a morphism $j:\mathcal{T}(\partial_{-}M) \rightarrow \mathcal{T}(\partial_{+}M).$  Finally, composing this with the map $\eta: \mathcal{T}(\partial_{+}M) \rightarrow \mathcal{T}(\partial_{+}M)$ induced by multiplication by $\mathcal{D}^{1-g}\text{dim}(i)$ on $\text{Hom}(\mathbf{1}, \Phi(t_1;i))$, we get the desired $K$-linear map 
\[
\tau(M):=\eta \circ j: \mathcal{T}(\partial_{-}M) \rightarrow \mathcal{T}(\partial_{+}M).
\]

\subsection{The TQFT based on extended cobordisms}

\subsubsection{Module of states} \label{Modules of e-states} We start by desribing the module of states for an $e$-surface. Start by assuming that $\Sigma$ is a connected $e$-surface. Recall the notion of a parametrization of $\Sigma.$ This is simply a weak e-homeomorphism $\Sigma_t \rightarrow \Sigma$ t

Given two parametrizations $f: \Sigma_{t_0} \rightarrow \Sigma$ and $g:\Sigma_{t_1} \rightarrow \Sigma,$ we wish to define an isomorphism $\varphi(f,g)$ between $\Psi(t_0)$ and $\Psi(t_1).$

We define
\[
\varphi(f,g):=(D\Delta^{-1})^{-\mu((f_0)_*(\lambda(t_0)),\lambda(\Sigma), (g_0)_*(\lambda(t_1)))}\mathcal{E}(g^{-1}f): \Psi(t_0) \rightarrow \Psi(t_1).
\]
Here $\mu$ is the Maslov index for triples of Lagrangian subspaces. $\mathcal{E}$ is the morphism induced by the decorated three manifold $\Sigma_{t_1}\times I$ where the bottom is parametrized by $g^{-1}\circ f: \Sigma_{t_0} \rightarrow \Sigma_{t_1}$ and the top is parametrized by the identity.

Turaev proves in \cite{Tu} that
\[
\varphi(f_1,f_2)\circ \varphi(f_0,f_1)=\varphi(f_0,f_2).
\]

Now $\mathcal{T}^e(\Sigma)$ is defined as the $K$-module of coherent sequences $(x(t,f))_{(t,f)}$ where we index over all parametrizations.

Finally, if $\Sigma$ is not connected, then we define $\mathcal{T}^e(\Sigma)$ to be the unordered tensor product of the modules of states of the components of $\Sigma.$

\subsubsection{Operator invariants}
Consider an extended $3$-manifold $(M,\partial_{-} M, \partial_{+}M).$ Then any two parametrizations  $f:\Sigma_{-} \rightarrow \partial_{-} M$ and $g:\Sigma_{+} \rightarrow \partial_{+} M$ makes $M$ into a decorated cobordism $\tilde{M}.$ We now define $\tau^e(M)$ to be composition
\[
\mathcal{T}^e(\partial_{-} M) \rightarrow \mathcal{T}(\Sigma_{-})\overset{\lambda(M)\tau(\tilde{M})}{ \longrightarrow} \mathcal{T}(\Sigma_{+}) \rightarrow \mathcal{T}^e(\partial_{+}M).
\]
Here $\lambda(M)$ is an invertible element of $K$ defined in section $6.5$ of chapter $\rom{4}$ of \cite{Tu}.

We observe that the description of $\mathcal{T}^e$ given above implies that a once punctured sphere marked with $i$ has a module of states that is isomorphic to $\text{Hom}(1,V_{i}),$ which is $K$ if $i=0$ and $\mathbf{0}$ otherwise.

\section{Proofs of lemmas}
\label{Proofs of lemmas}

\begin{pro}[The cobordism associated to a weak $e$-homomorphism]
\label{cob of e-homeo}
Let $f: \mathbf{\Sigma}_1 \rightarrow \mathbf{\Sigma}_2$ be a weak $e$-homomorphism. There is an invertible scalar $c \in K$ such that the operator invariant of the extended cobordism $\Sigma_1 \times I \cup_f \Sigma_2 \times I$ coincide with  $c\mathcal{T}^e(f):\mathcal{T}^e(\Sigma_1) \rightarrow \mathcal{T}^e(\Sigma_2).$ The scalar $c$ depends only on the underlying continuous map of $f$ and the Lagrangian subspaces $L_i \subset H_1(\Sigma_i).$
\end{pro}

\begin{proof}
This follows from theorem $7.1$ of chapter $\rom{7}$ in \cite{Tu}.
\end{proof}

Here it is understood that the extended three manifold $\Sigma_1 \times I \cup_f \Sigma_2 \times I$ is obtained by gluing the top of $\Sigma_1 \times I$ to the bottom of $\Sigma_2 \times I$ through $f.$

\begin{proof}[Proof of Lemmas \ref{f} and \ref{fd}]
Using \ref{cob of e-homeo} and theorem $7.1$ of chapter $\rom{7}$ we see that both $\mathcal{H}(g)\circ \dot{f}$ and $\dot{f}\circ\mathcal{H}(g)$ are - up to the same scalar- induced by gluing certain extended three manifolds. Let $M_1$ be the extended three manifold with $\tau(M_1)=c\mathcal{H}(g)\circ \dot{f},$ and let $M_2$ be the extended three manifold with $\tau^e(M_2)=c\dot{f}\circ\mathcal{H}(g)$ Clearly there is a homeomorphism of extended three manifolds taking $M_1$ to $M_2,$ commuting with boundary parametrizations. Therefore they induce the same morphism.
\end{proof}

\begin{proof}[Proof of Lemmas \ref{h} and \ref{hd}]
The proof is virtually identical to the proof of \ref{f}.
\end{proof}

\begin{proof}[Proof of Lemma \ref{farc}]
The proof is virtually identical to the proof of \ref{f}.
\end{proof}

\begin{proof}[Proof of Lemma \ref{Cfhg}]
We start by recalling the definition of the gluing homomorphism provided by Turaev in sections $4.4-4.6$ of chapter $\rom{5}$ in \cite{Tu}. Let $M_2$ be the extended three manifold obtained by attaching handles to $\Sigma\times I$ as in section $4.4.$ of chapter $\rom{5}.$ The attachment uses the gluing data $c.$ The operator invariant $\tau^e(M_2)$ now yield a map $ g': \mathcal{T}^e(\Sigma) \rightarrow \mathcal{T}^e(\Sigma_c').$
Here $\Sigma_c'$ is an $e$-surface canonically $e$-homeomorphic to $\Sigma_c.$ Composing with the associated isomorphism of $K$-modules, we get the required gluing homomorphism $g: \mathcal{T}^e(\Sigma) \rightarrow \mathcal{T}^e(\Sigma_c).$ 

Similarly, if we let $\widetilde{\Sigma}$ and $\widetilde{\Sigma}_c'$ be the two same $e$-surfaces with the relevant change of markings, then the gluing $g$ is obtained as the operator invariant of $\widetilde{M_2}$ with the obvious notation. 

Assume now that $r=\dot{f}$ for some homomorphism $f: (V,+1) \rightarrow (W,+1).$ Recalling the naturality property of $\dot{f}$ we see that it is enough to argue that $\dot{f}$ commute with $g'.$

Let $M_1=\Sigma\times I$ be the extended cobordism inducing $\dot{f}.$ Then $g'\circ \dot{f}$ is a multiple of $\tau^e(\widetilde{M_2}\circ M_1),$ where we glue the top of $M_1$ to the bottom of $\widetilde{M_2}$ through the identity. To compute the relevant scalar we use theorem $7.1$ in Chapter $\rom{4}$ of \cite{Tu}. Since the identity is an $e$-homeomorphism here, there is only one Maslow index to compute.

In the notation of theorem $7.1$ we have $\text{id}_{\#}(N_1)_{*}(\lambda_{-} (M_1))=\text{id}_{\#}\lambda_{+}(M_1).$ Thus we see that
\[
0=\mu\left(\text{id}_{\#}(N_1)_{*}(\lambda_{-} (M_1)),\text{id}_{\#}\lambda_{+}(M_1),N_2^*(\lambda_{+}(\widetilde{M_2})\right).
\]
See the proof of lemma $6.7.2$ in chapter $\rom{4}$ of \cite{Tu}. Thus we get that
\[
\tau^e(\widetilde{M_2}\circ M_1)=g'\circ \dot{f}.
\]
Clearly $\widetilde{M_2}\circ M_1$ is $e$-homeomorphic to a cylinder with handles attached on the top, such that the $\beta$-band has  a coupon colored with the $f$-coupon, and all other 'vertical' bands are colored with $\text{id}.$ 

The exact same argument will yield a similar description of $\dot{f}\circ g.$ Consider $Q:= \Sigma_c \times I$ as the extended cobordism inducing $\dot{f}: \mathcal{T}^e(\Sigma_c) \overset{\sim}{\rightarrow} \mathcal{T}^e(\widetilde{\Sigma_c}).$ Arguing as above, wee see that $\dot{f}\circ g'$ is given by the operator invariant $\tau^e(Q\circ M_2).$  But this is $e$-homeomorphic to $\widetilde{M_2}\circ M_1$. Therefore $g\circ \dot{f}=\dot{f}\circ g'.$

Observe that a homomorphism of type $h$ can be dealt with in exactly the same way.
\end{proof}

\begin{proof}[Proof of Lemma \ref{<}]This is a consequence of the description of the gluing homomorphism given above, together with the existence of the proclaimed $e$-morphisms. \end{proof}

\section{Uniqueness up to quasi-isomorphism}

We observe that the construction of the gluing $\tilde{g}$ depended on a choice of isomorphisms $q_i: V_{i^*} \overset{\sim}{\longrightarrow} V_i^*.$ This dependence is not essential. 

\begin{df}[Quasi-isomorphism]
\label{qid}
Let $(Z,g)$ and $(Z',g')$ be two modular functors with the same label set $\Lambda.$ These are said to be quasi-isomorphic if there is a pair $(\Phi,\gamma).$  Here $\Phi$ is an assignment of isomorphisms, which for each labeled marked surface $\mathbf{\Sigma}$ gives an isomorphism
\[
\Phi(\mathbf{\Sigma}): Z(\mathbf{\Sigma}) \overset{\sim}{\longrightarrow} Z'(\mathbf{\Sigma}).
\]
This assignment is required to be natural with respect to morphisms of modules induced by morhisms of labeled marked surfaces. Similarly it is required to preserve the splitting into tensor products induced by disjoint union, as well as the permutation map. Further, $\gamma$ is an assignment $\gamma: I \rightarrow K^*$ such that if $\mathbf{\Sigma}_c$ is obtained from $\mathbf{\Sigma}$ from gluing along an ordered pair $(p,q)$ where $p$ is labeled with $\lambda$, then the following diagram is commutative

\begin{equation}
\label{quasi-isomorphism}
\begin{tikzcd}[swap]
Z(\mathbf{\Sigma}) \arrow{r}[swap]{g}
\arrow{d}{\gamma(\lambda) \Phi(\mathbf{\Sigma})}
& Z(\mathbf{\Sigma}_c) \arrow{d}[swap]{\Phi(\mathbf{\Sigma}_c)}
\\ Z'(\mathbf{\Sigma}) \arrow{r}[swap]{g'} & Z'(\mathbf{\Sigma}_{c}).
\end{tikzcd}
\end{equation}

Moreover we demand that $\gamma(\lambda)\gamma(\lambda^*)=1$ for all $\lambda \in I$.
\end{df}

This is easily seen to define an equivalence relation on modular functors with the same label set.

\begin{thm}[Independence of $(q_i)$ up to quasi-isomorphism]
\label{QIT}
Let $q,q'$ be two choices of isomorphisms $V_{i^*} \overset{\sim}{\longrightarrow} V_i^*.$ Then the two resulting modular functors $Z(q)$ and $Z(q')$ are quasi-isomorphic.
\end{thm}

\begin{proof}
Write $q'_i= f_iq_i.$ Then we have $\tilde{g}_j(q')=f_j\tilde{g}_j(q).$ We want to construct a pair $(\Phi,\gamma).$ Consider a labeled marked surface $\mathbf{\Sigma}$ with labels $i_1,...,i_k.$ We want to construct $\Phi(\mathbf{\Sigma})$ to be of the form $(\prod_l^k \alpha_{i_l}) \text{Id}_{Z(\mathbf{\Sigma})}$ for some function $\alpha: I \rightarrow K^*.$

Assume $\mathbf{\Sigma}_c$ is obtained from $\mathbf{\Sigma}$ by gluing along an ordered pair $(p,q)$ where $p$ is labeled with $i.$ Assume the labels of $\mathbf{\Sigma}$ are $i_1,...,i_k,i,i^*.$ Then equation (\ref{quasi-isomorphism}) becomes
\[
\prod_{l=1}^k \alpha(i_l)= \gamma(i)f_i \alpha(i)\alpha(i^*)\prod_{l=1}^k \alpha(i_l).
\]

Thus we are forced to define
\[
\gamma(i) := \frac{1}{\alpha(i)\alpha(i^*)f_i}.
\]

We still have to ensure $\gamma(i)\gamma(i^*)=1.$ We see that this will follow for any choice of $\alpha$ with
\[
\left(\alpha(i)\alpha(i^*)\right)^2 f_if_{i^*}=1,
\] 
which is easy to solve (by adjoining the needed square roots if needed).
\end{proof}

\section{Universal property}

In this section we will describe how to apply $Z,$ and how to use it in calculations. Let $\mathbf{\Sigma}$ be a connected labeled marked surface. Recall that a parametrization $f: \Sigma_t \rightarrow \overline{\G(\mathbf{\Sigma})}$ is an orientation preserving homeomorphism that preserves all structure of extended surfaces, except possibly the Lagrangian subspaces in homology. Clearly the set of parametrizations is non-empty. Let $f$ be a parametrization. This will induce an isomorphism

\[
\Psi(t) \simeq Z(\mathbf{\Sigma}).
\]
For the definition of $\Psi(t)$ see section \ref{Review of the TQFT}. We now recall the definition of $Z(\mathbf{\Sigma})$ and desribe the isomorphism above. For any pair of parametrizations $f_i : \Sigma_{t_i} \rightarrow \overline{\G(\mathbf{\Sigma})}$ there is an isomoprhism
\[
\varphi(f_1,f_2): \Psi_{t_1} \overset{\sim}{\longrightarrow} \Psi_{t_2}.
\]
See section \ref{Review of the TQFT}. With obvious notation these isomorphisms satisfy
\[
\varphi(f_1,f_3)= \varphi(f_2,f_3) \circ \varphi (f_1,f_2).
\]
The module $Z(\mathbf{\Sigma})$ is the the module of coherent sequences. Hence an element of this module is an equivalence class of pairs $(x,f)$ where $f$ is a parametrization with domain $\Sigma_t$ and $x$ is an element of $\Psi(t).$ We have $(x,f) \sim (y,g)$ if and only if $\varphi(f,g)(x)=y.$ The isomorphism $\Psi(t) \simeq Z_{\mathcal{V}}(\mathbf{\Sigma})$ induced from a parametrization is simply $x \mapsto (x,f).$

If $\mathbf{f}=(f,s): \mathbf{\Sigma}_1 \rightarrow \mathbf{\Sigma}_2$ is a morphism of connected labeled marked surfaces, then any representative $f'$ of the isotopy class $f$ will induce $Z_{\mathcal{V}}(\mathbf{f})$ which is given by
\[
(x,g) \mapsto ((\Delta^{-1}D)^s x,f'\circ g).
\]
Here we also write $f'$ for the induced $e$-homeomorphism $\overline{\G(\mathbf{\Sigma}_1)} \rightarrow \overline{\G(\mathbf{\Sigma}_2)}.$

\section{The duality pairing}

Consider a modular functor $V.$ For a modular functor with duality we would like the operation of orientation reversal to be taken to the operation of taking the dual $K$-module. That is, we would like a perfect pairing $V(\mathbf{\Sigma})\otimes V(-\mathbf{\Sigma}) \rightarrow K$ that is compatible with the structure of $V.$ 

Before we formulate the axioms, consider an arbitrary $\Lambda$-labeled marked surface $\mathbf{\Sigma}'.$ Observe that if $p,q \in \mathbf{\Sigma}'$ are subject to gluing then so are $p,q \in -\mathbf{\Sigma}'.$ Oberve that if $\mathbf{\Sigma}$ is the result of gluing $\mathbf{\Sigma}$ along $p,q$ then $-\mathbf{\Sigma}$ is the result of gluing $-\mathbf{\Sigma}'$ along the same ordered pair of points.

\begin{df}[Duality]
\label{duality def}
Let $(V,g)$ be a modular functor based on $\Lambda$ and $K.$ A duality for $V$ is a perfect pairing
\[
(\emptyarg, \emptyarg)_{\mathbf{\Sigma}}: V(\mathbf{\Sigma})\otimes V(-\mathbf{\Sigma}) \rightarrow K,
\]
subject to the following axioms.

\paragraph{Naturality.} Let $\mathbf{f}=(f,s): \mathbf{\Sigma}_1 \rightarrow \mathbf{\Sigma}_2$ be a morphism between $\Lambda$-labeled marked surfaces. Then
\begin{equation}
(V(\mathbf{f}),V(-\mathbf{f}))_{\mathbf{\Sigma}_2} =(\emptyarg, \emptyarg)_{\mathbf{\Sigma}_1}
\end{equation}
\paragraph{Compatibility with disjoint union.} Consider a disjoint union of $\Lambda$-labeled marked surface $\mathbf{\Sigma}=\mathbf{\Sigma}_1 \sqcup \mathbf{\Sigma}_2.$ The modular functor $V$ provide an isomorphism
\[
\eta:V(\mathbf{\Sigma})\otimes V(-\mathbf{\Sigma}) \overset{\sim}{\longrightarrow} V(\mathbf{\Sigma}_1)\otimes V(-\mathbf{\Sigma}_1) \otimes V(\mathbf{\Sigma}_2) \otimes V(-\mathbf{\Sigma}_2).
\]
We demand that with respect to the natural isomorphism $K \otimes K \simeq K$ we have that
\begin{equation}
( \emptyarg, \emptyarg)_{\mathbf{\Sigma}}=\left((\emptyarg, \emptyarg)_{\mathbf{\Sigma}_1} \otimes (\emptyarg, \emptyarg)_{\mathbf{\Sigma}_2}\right) \circ \eta.
\end{equation}.

\paragraph{Compatibility with gluing.} Let $\mathbf{\Sigma}$ be a $\Lambda$ labeled marked surface obtained from gluing. Consider the gluing isomorphism:
\[
g: \bigoplus_{\lambda \in \Lambda} V(\mathbf{\Sigma}(\lambda)) \overset{\sim}{\longrightarrow} V(\mathbf{\Sigma}),
\] 
as desribed in definition \ref{defofmf}. We have that
\begin{equation}
(g,g)_{\mathbf{\Sigma}}= \sum_{ \lambda \in \Lambda} \mu_{\lambda}(\emptyarg, \emptyarg)_{\mathbf{\Sigma}(\lambda)}.
\end{equation}
where $\mu_{\lambda} \in K$ is invertible and only depends on the isomorphism class of $\mathbf{\Sigma}(\lambda)$ for all $\lambda$.

\paragraph{Compatibility with orientation reversal.} For a $\Lambda$-labeled marked surface $\mathbf{\Sigma}$ we demand that there is an invertible element $\mu \in K^*$ that only depends on $\mathbf{\Sigma}$ such that for all $(v,w) \in V(\mathbf{\Sigma})\times V(-\mathbf{\Sigma)}$ the following equation holds
\begin{equation}
 \mu (w,v)_{-\mathbf{\Sigma}}= (v,w)_{\mathbf{\Sigma}}
\end{equation}

\end{df}

It is worth spelling out how we demand that the duality is compatible with gluing in a litte more detail. Observe $-(\mathbf{\Sigma
}(\lambda))=(-\mathbf{\Sigma})(\lambda^\dagger)$. Thus the gluing isomorphism is a splitting
\[
g': \bigoplus_{\lambda \in \Lambda} V(-\mathbf{\Sigma}(\lambda^\dagger)) \overset{\sim}{\longrightarrow} V(-\mathbf{\Sigma}).
\]
This gives a decomposition
\[
\bigoplus_{\lambda
,\lambda' \in \Lambda} V(\mathbf{\Sigma}(\lambda)) \otimes V(-\mathbf{\Sigma}(\lambda')) \overset{g\otimes g}{\longrightarrow} V(\mathbf{\Sigma})\otimes V(-\mathbf{\Sigma}).
\]

Then the statement is that $g(V(\mathbf{\Sigma}(\lambda))$ and $g(V(-\mathbf{\Sigma}(\lambda'))$ are orthogonal w.r.t. the duality $(\emptyarg, \emptyarg)_{\mathbf{\Sigma}}$ unless $\lambda=\lambda'.$ In this case we have 
\[
(g_{\lambda}, g_{\lambda^{\dagger}})_{\mathbf{\Sigma}}=\mu_\lambda(\emptyarg,\emptyarg)_{\mathbf{\Sigma}(\lambda)}.
\]
We allow $\mu_\lambda$ to depend on the isomorphism class of $\mathbf{\Sigma}(\lambda).$

Let us briefly comment on the self-duality condition. We see that $\mu=1$ is equivalent to
\[
\langle u,w  \rangle_{\mathbf{\Sigma}}=\langle w,u \rangle_{-\mathbf{\Sigma}},
\] 
for all $u \in V(\mathbf{\Sigma})$ and $w \in V(-\mathbf{\Sigma}).$

\subsection{Review of the duality for Turaev's modular functor based on extended surfaces}
Let $\Sigma$ be an $e$-surface. Recall the operation of oriental reversal as desribed in definition \ref{Orientation reversal for e}. We can think of $\Sigma\times I$ as a morphism $\Sigma\sqcup -\Sigma \rightarrow \emptyset.$ This induce a perfect pairing
\[
\langle \emptyarg , \emptyarg \rangle_{\Sigma}:\mathcal{T}^e(\Sigma) \otimes \mathcal{
T}^e(-\Sigma) \rightarrow K. 
\]
See chapter $\rom{3}$ section $2$ in \cite{Tu}.

The pairing is compatible with the action of $e$-homeomorphisms in the sense that for any $e$-homeomorphism $f: \Sigma_1 \rightarrow \Sigma_2$ we have
\[
\langle \mathcal{T}^e(f)(\emptyarg), \mathcal{T}^e(-f)(\emptyarg) \rangle_{\Sigma_2} = \langle \emptyarg, \emptyarg \rangle_{\Sigma_1}.
\]
It is proven in exercise $7.3$ in chapter $\rom{4}$ that the pairing is also natural with respect to weak $e$-homeomorphisms. The pairing is multiplicative with respect to disjoint union. Moreover the pairing is self-dual in the following sense:
\[
\langle \emptyarg, \emptyarg \rangle_{\Sigma} \circ \text{Perm} = \langle \emptyarg, \emptyarg \rangle_{-\Sigma}.
\]
All these properties are stated in axiom $1.2.4$ in section $1.2$ of chapter $\rom{3}$ in \cite{Tu}.

\subsection{Construction of a duality pairing for $Z$}

Consider an $I$-labeled marked surface $\mathbf{\Sigma}=(\Sigma, P,V,(i_p)_{p\in P},L).$ Write 
\[
\overline{\G(\mathbf{\Sigma})}=(\widetilde{\Sigma}, (\alpha_p)_{p \in P}, (V_{i_p},1)_{p \in P}, L)
\]
for the $e$-surface associated to the $m$-surface $\G(\mathbf{\Sigma}).$ We have
\[
\overline{\G(-\mathbf{\Sigma})}=(\widetilde{\Sigma}, (\alpha_p)_{p \in P}, (V_{i_p^*},1)_{p \in P}, L).
\]
This is not quite $-\overline{\G(\mathbf{\Sigma})}.$ However, let $\dot{q}$ be the isomorphism of states that take all markings $(V_{i^*},1)$ to $(V_i^*,1)$ and let $\tilde{h}$ be the isomorphism of modules of states that exchange all markings $({V_i^*},1)$ with $(V_i,-1).$ Define
 \[
*\Sigma:=(\widetilde{\Sigma}, (\alpha_p)_{p \in P}, (V_{i_p},-1)_{p \in P}, L)
\]
Now let $r$ be the orientation preserving diffeomorphism
\[
r:*\Sigma\overset{\sim}{\longrightarrow} -\overline{\G(\mathbf{\Sigma})},
\]
that is given by twisting all arcs with a half-twist. This can of course be done in two different ways, the important thing for now is that it is done the same way for all arcs. We return to this choice in the proof of proposition \ref{rp}.

Then we have an isomorphism
\[
\mathcal{T}^e(r) \circ \tilde{h}\circ \dot{q}: \mathcal{T}^e(\overline{\G(-\mathbf{\Sigma})})\overset{\sim}{\longrightarrow} \mathcal{T}^e(-\overline{\G(\mathbf{\Sigma})})
\]
This will allows us to define a perfect pairing. For notational convenience we will simply write $r'=\mathcal{T}^e(r).$ We will write
\begin{equation*}
\zeta =\mathcal{T}^e(r) \circ \tilde{h}\circ \dot{q}.
\end{equation*} One last notational definition will be convenient. For a decorated type $t$ of the form $(g;(V_{i_l},\nu_l))$  with $i_l \in I$ for all $l,$ let $t*$ be the decorated type $(g;(V_{i_l^*},\nu_l)).$

\begin{df}
Consider an $I$-labeled marked surface $\mathbf{\Sigma}.$ We have a perfect pairing given by the composition
\[
\left( \emptyarg , \emptyarg \right)_{\mathbf{\Sigma}}=\langle \emptyarg , \zeta(\emptyarg) \rangle_{\overline{\G(\mathbf{\Sigma})}}.
\]
\end{df}

We have here resorted to a slight abuse of notation, since technically, $Z(\mathbf{\Sigma})$ is not equal to $\mathcal{T}^e(\overline{\G(\mathbf{\Sigma})}),$ but canonically isomorphic to it.

\begin{thm}[Duality] \label{DT} \label{duality} The pairing $(\emptyarg, \emptyarg)_{\mathbf{\Sigma}}$ is a duality pairing for the modular functor $Z_{\mathcal{V}}.$
\end{thm}

\subsection{Description of the gluing homomorphism}
\label{dofg}

For the proof of theorem (\ref{duality}) we will use an explicit description of the gluing homomorphism in two cases.

\subsubsection{The two points lie on the same component} We will start by assuming that the points subject to gluing lie on the same component. Write $\mathbf{\Sigma}$ for the labeled marked surface resulting from gluing and write $\mathbf{\Sigma}(i)$ for the labeled marked surface with the two points subject to gluing where the preferred point is labeled with $i.$ Due to the multiplicativity of the gluing we will assume that $\mathbf{\Sigma}(i)$ is connected. See equation $(\ref{multofg}).$ We will start by assuming $\overline{\mathcal{G}(\mathbf{\Sigma}(i))}=\Sigma_t,$ where 
$$t = (g; (V_{i_1},1),...,(V_{i_k},1),(V_i,1),(V_{i^*},+1)).$$ 
Hence $\overline{\mathcal{G}(\mathbf{\Sigma})}=\Sigma_{t'}$ where
$t'$ is equal to the topological type $(g+1,(V_{i_1},1),...,(V_{i_k},1)).$ Moreover let 
$$\tilde{t} = (g; (V_{i_1},1),...,(V_{i_k},1),(V_i,1),(V_{i},-1)).$$ In Turaev's setup we see that $\Sigma_{\tilde{t}}$ can be glued along the points labeled with $(V_i,1)$ and $(V_i,-1)$ to obtain $\Sigma_{t'}.$ Now the identity parametrizations induce isomorphisms
\begin{align*}
\\ & \mathcal{T}^e(\Sigma_{\tilde{t}}) \simeq \bigoplus_{l \in I^{g}} \text{Hom}(\mathbf{1},\Phi(\tilde{t},l)),
\\& Z(\mathbf{\Sigma}) \simeq \bigoplus_{l \in I^{g+1}} \text{Hom}(\mathbf{1},\Phi(t',l)),
\\& Z(\mathbf{\Sigma}(i)) \simeq \bigoplus_{l \in I^{g}} \text{Hom}(\mathbf{1},\Phi(t,l)).
\end{align*}

With respect to these isomorphisms, we see that our gluing homomorphism is the composition
\begin{equation}
\label{ggg}
Z(\mathbf{\Sigma}(i)) \overset{h \circ \dot{q}}{\rightarrow} \mathcal{T}^e(\Sigma_{\tilde{t}}) \hookrightarrow Z(\mathbf{\Sigma}).
\end{equation}
Here the last map is the natural summandwise inclusion. This is proven in section $5.9$ of chapter $\rom{5}$ of \cite{Tu}. Thus we only need to describe the first map.

Consider a summand in $Z(\mathbf{\Sigma})$ and an element $f$
\[
f\in \text{Hom}\left(\mathbf{1},(\otimes_{j=1}^k V_{i_j}) \otimes V_i\otimes V_{i^*}\otimes_{r=1}^{g}(V_{l_r}\otimes V_{l_r}^*)\right).
\]
Let $W=(\otimes_{j=1}^k V_{i_j}) \otimes V_i.$ Let $R=\otimes_{r=1}^{g}(V_{l_r}\otimes V_{l_r}^*).$ Let $q_i: V_{i^*} \overset{\sim}{\longrightarrow} V_i^*$ be the isomorpism used to define the gluing. Postcomposing $f$ with $(1_{W}\otimes q_{i} \otimes 1_{R})$ we get an element
\[ (1_{W}\otimes q_{i} \otimes 1_{R})\circ f \in 
\text{Hom}\left(\mathbf{1},(\otimes_{j=1}^k V_{i_j}) \otimes V_i\otimes V_i^* \otimes_{r=1}^{g}(V_{l_r}\otimes V_{l_r}^*)\right).
\]
To see that $h\circ\dot{q}(f)= (1_{W}\otimes q_{i} \otimes 1_{R})\circ f $ 
one can either go throug the construction given in the review in section \ref{Review of the TQFT} above and use that the identity cylinder induce the identity, or one can use the techniques of section $2.3$ in chapter $\rom{4}$ of \cite{Tu}.

Now we desribe the gluing in a slightly more general situation. We observe that $\Sigma_{r}$ is naturally a labelled marked surface, for any type $r$ where all marks are of type $(V_i,1).$ Assume that we have a parametrization $f: \Sigma_t \rightarrow  \overline{\G(\mathbf{\Sigma}(i))}.$ There is a natural homeomorphism $\overline{\G(\mathbf{\Sigma}(i))}\simeq \mathbf{\Sigma}(i).$  With respect to this identification we can think of $f$ as a diffeomorphism betweem labeled marked surface that preserve all the data except possibly the Lagrangian subspaces in homology. There is also a natural homeomorphism of $\Sigma_{t'}$ with the surface obtained from $\Sigma_t$ by gluing the ordered pair corresponding under $f$ to the relevant ordered pair of $\mathbf{\Sigma}(i)$. We have a natural homeomorphism $\overline{\G(\mathbf{\Sigma})}=\mathbf{\Sigma}.$ As in proposition \ref{Fofg} we can choose a parametrization diffeomorphism  $F=z(f): \Sigma_{t'} \rightarrow \overline{\G(\mathbf{\Sigma})}$ that is compatible with $f.$

Now $f,F$ induce a pair of isomorphisms
\begin{align}
\label{paraf} & Z(\mathbf{\Sigma}) \simeq \bigoplus_{l \in I^{g+1}} \text{Hom}(\mathbf{1},\Phi(t',l)),
\\& \label{paraF} Z(\mathbf{\Sigma}(i)) \simeq \bigoplus_{l \in I^{g}} \text{Hom}(\mathbf{1},\Phi(t,l)).
\end{align}
With respect to these isomorphisms we have equation (\ref{ggg}). 

\subsubsection{The two points lie on distinct components}

We assume that $\mathbf{\Sigma}(i)=\mathbf{\Sigma}^{+} \sqcup \mathbf{\Sigma}^{-}$ where $\mathbf{\Sigma}^{+}$ and $\mathbf{\Sigma}^{-}$ are two spheres. Assume that $(p,i) \in \mathbf{\Sigma}^{+}$ and that $(q,i^*) \in \mathbf{\Sigma}^{-}.$ We will assume $\mathbf{\Sigma}^{-} = \Sigma_{t_-}$ and $\mathbf{\Sigma}^{+}=\Sigma_{t_+}$ where $t_+=(0;(V_{i_1},1),...,(V_{i_n},1),(V_i,1))$ and $t_- = (0;(V_{i^*},1),(V_{i_{n+1}},1),...,(V_{i_m},1)).$ Thus we get $\mathbf{\Sigma}=\Sigma_t$ where 
$$t = (0,(V_{i_1},1),...,(V_{i_m},1)).$$ 
We get isomorphisms
\begin{align*}
& Z(\mathbf{\Sigma}(i)) \simeq \text{Hom}(\mathbf{1},\Phi(t_+)) \otimes \text{Hom}(\mathbf{1},\Phi(t_{-})),
\\ & Z(\mathbf{\Sigma}) \simeq \text{Hom}(\mathbf{1},\Phi(t)).
\end{align*}
Let $V_1=V_{i_1} \otimes \cdots \otimes V_{i_n}$ and $V_2=V_{i_{n+1}}\otimes \cdots \otimes V_{i_m}.$ With respect to these isomorphisms the gluing homomorphism is given by
\begin{equation}
\label{gggg}
Z(\mathbf{\Sigma}(i)) \ni x\otimes y \mapsto (1_{V_1} \otimes \overline{d_{V_i}} \otimes 1_{V_2}) \circ (x \otimes q_i \circ y) \in Z(\mathbf{\Sigma}).
\end{equation}
Here $\overline{d_V}$ is given by $F(\cap_V^{-})$ where $F$ is the Reshetikhin-Turaev functor and $\cap_V^{-}$ is defined in Figure $2.6$ in section $2.3$ of chapter $\rom{1}$ in \cite{Tu}. 

This formula can be argued by using the description of $h\circ \dot{q}$ given above and by arguing very similarly to the reasoning in section $5.10$ of chapter $\rom{5}$ in \cite{Tu}.

In the general case, where $\mathbf{\Sigma}^{+},\mathbf{\Sigma}^{-}$ are homeomorphic to spheres we start with a parametrization of each component $\overline{\G(\mathbf{\Sigma}^{\pm})}$ and then we glue these two together to obtain a paramtrization of $\overline{\G(\mathbf{\Sigma})}.$ These will induce isomorphisms with respect to which the gluing is given by (\ref{gggg}). If we start with parametrizations $f,g$ we will write $z'(f\otimes g)$ for the resulting parametrization.

\subsection{Proof of theorem \ref{duality}}

\begin{pro} 
The pairing $(\emptyarg, \emptyarg)_{\mathbf{\Sigma}}$ is functorial and is compatible with disjoint union. 
\end{pro}

\begin{proof}
This easily follows from the properties of $\langle \emptyarg, \emptyarg \rangle$ and the functorial properties of $\dot{q},h$ and $r.$  Observe that even though the axioms in section $1$ of chapter $\rom{2}$ only ensure that $\langle \emptyarg, \emptyarg \rangle$ is natural with respect to $e$-homeomorphisms, it is proven in exercise $7.3$ in chapter $\rom{4}$ that the pairing is also natural with respect to weak $e$-homeomorphisms.
\end{proof}

Thus it remains to prove that it is compatible with gluing, and that it is self-dual.

The proof of the main propositions needed for these results is based on an explicit ribbon graph presentation of $\langle \emptyarg, \emptyarg \rangle_{\Sigma}: \mathcal{T}^e(\Sigma)\otimes \mathcal{T}^e(-\Sigma) \rightarrow K.$ 

All of the proofs in this section are modifications of material appearing in section $10.4.$ in chapter $\rom{4}$ of \cite{Tu}.

\begin{pro}[Surgery presentation of $\langle \emptyarg, \emptyarg \rangle_{\Sigma}$] \label{presentation 1}
Assume $\Sigma$ is a connected $e$-surface. For any parametrization $f:\Sigma_t \rightarrow \Sigma$ there is an induced parametrization $y(f):\Sigma_{-t} \rightarrow -\Sigma$ such that with respect to the two induced isomorphisms 
\begin{gather*}
\Psi(t) \simeq \mathcal{T}^e(\Sigma), \\ \Psi(-t) \simeq \mathcal{T}^e(-\Sigma),
\end{gather*} we have the following surgery presentation of $\langle \emptyarg, \emptyarg \rangle_{\Sigma}.$
\\
\\
\begin{center}
\includegraphics{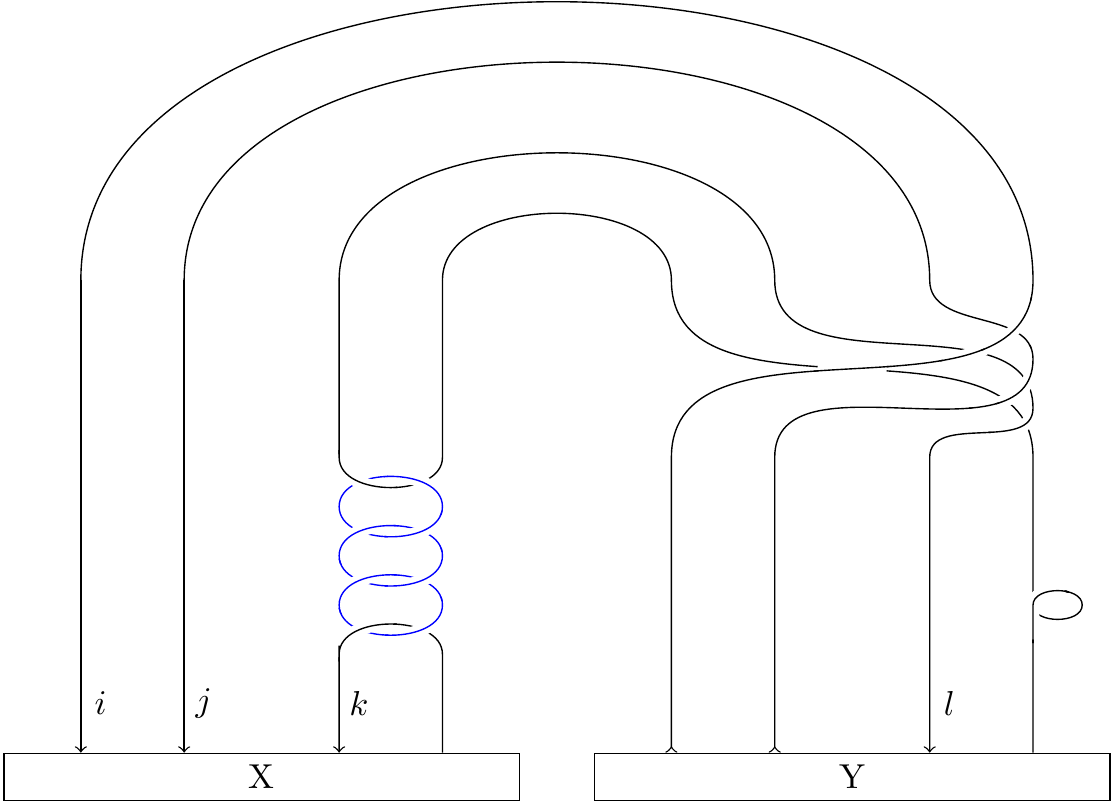}
\end{center}
\end{pro}
Observe that in this proposition, orientation reversal is with respect to extended $e$-surfaces. Here the blue unknot's are the surgery link components. We depict here the genus $1$ case. It is obvious how to generalize to higher genus.

Here the blue unknot's are the surgery link components. We depict only the genus $1$ case, since the generalization to higher genus is obvious.

We stress that the tangle is a presentation of a perfect pairing
\begin{equation}
\label{algpairing}
\Psi(t) \times \Psi(-t) \rightarrow K,
\end{equation}
and that we can only use it as a presentation of the duality pairing with respect to certain pairs of parametrizations $f:\Sigma_t \rightarrow \Sigma$ and $y(f):\Sigma_{-t} \rightarrow -\Sigma.$ This will be explained in the proof.
We will denote the pairing from (\ref{algpairing}) by $\langle \emptyarg, \emptyarg \rangle_t.$

\begin{proof}
This proof is a slight modification of the proof of theorem $10.4.1$ in chapter $\rom{4}$ of \cite{Tu}. Observe that in this proof $-\Sigma$ is the result of using the operation of orientation reversal of extended surfaces on $\Sigma.$

Choose a parametrization $f: \Sigma_t \rightarrow \Sigma.$ This will induce a weak $e$-homeomorphism $-f: -\Sigma_t \rightarrow - \Sigma.$ Consider the $e$-homeomorphism $s:\Sigma_{-t} \rightarrow -\Sigma_t$ given by a reflection in $y=0$ followed by counter clockwise half twists in the $X,Y$-plane at the distinguished arcs - with respect to the usual identification $\R^2\times \R = \R^3.$ This yields a parametrization $y(f):=(-f)\circ s: \Sigma_{-t} \rightarrow - \Sigma.$ These two parametrizations provide isomorphisms
\begin{gather*}
\Psi(t) \simeq \mathcal{T}^e(\Sigma),
\\
\Psi(-t) \simeq \mathcal{T}^e(-\Sigma).
\end{gather*}

Now let 
\[
x \in \text{Hom}(\mathbf{1},\Phi(t,i)) \subset \mathcal{T}^e(\Sigma), \ \ \  y \in \text{Hom}(\mathbf{1},\Phi(-t,j)) \subset \mathcal{T}^e(-\Sigma).
\]

Consider the standard handlebodies denoted by $P(x)=H(U_t,R_t,i,x)$ and $ Q'(y)=H(U_{-t},R_t,j,y).$ We recall that $\langle x,y \rangle$ is given by $\tau(W(x,y)),$ where $W(x,y)$ is the closed three manifold with a colored ribbon graph inside it, that is obtained by gluing $P\sqcup Q'$ to $\Sigma\times I$ through the orientation reversing homeomorphism
\[
 \partial (P(x)\sqcup Q'(y))=\Sigma_t \sqcup \Sigma_{-t} \overset{f \sqcup y(f)}{\longrightarrow} \Sigma\times \{0\}\sqcup \Sigma\times \{1\} = \partial (\Sigma \times I).
\] 

We now observe that the parametrization $s:\Sigma_{-t} \rightarrow -\Sigma$ extends to an $e$-homeomorphism of three manifolds 
\[
Q'(y) \rightarrow Q(y),
\]
where $Q(y)$ is same handlebody with the LH-orientation and the induced colored ribbon graph.

We have a homeomorphism of extended three manifolds with colored ribbon graphs
\[
P(x) \cup_{\text{id}} Q(y) \overset{\sim}{\longrightarrow} W(x,y).
\]
Comparing $P(x) \cup_{\text{id}} Q(y)$ with the three manifold $M \sqcup_{\text{id}} -N$ as considered in the proof of theorem $10.4.1$ in \cite{Tu}, we obtain the desired presentation. 
\end{proof}

Now we want to use this to provide a presentation for the induced duality pairing on $Z_{\mathcal{V}}.$ Again it should be stressed that this presentation is only valid with respect to certain parametrizations.

\begin{pro}[Surgery presentation of $(\emptyarg, \emptyarg)_{\mathbf{\Sigma}}$] \label{rp}
Let $\mathbf{\Sigma}$ be a connected $I$-labeled marked surface. For any parametrization $f:\Sigma_t \rightarrow \overline{\mathcal{G}(\mathbf{\Sigma})}$ there is a parametrization $u(f): \Sigma_{t*} \rightarrow \overline{\G(-\mathbf{\Sigma})}$ such that with respect to the induced isomorphisms
\begin{gather*}
\Psi(t) \simeq Z(\mathbf{\Sigma}),
\\  \Psi(t*) \simeq Z(-\mathbf{\Sigma}),
\end{gather*} we have the following presentation of $(\emptyarg, \emptyarg)_{\mathbf{\Sigma}}.$

\
\begin{center}
\includegraphics{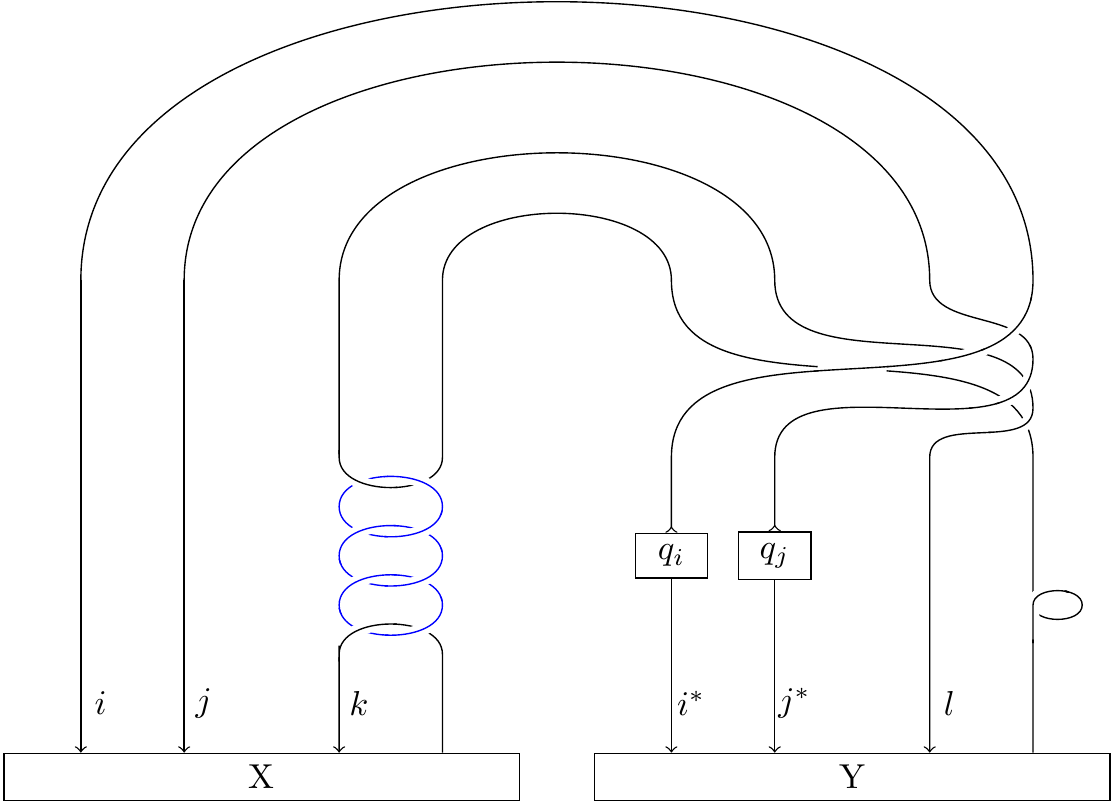}
\end{center}
\end{pro}

\begin{proof}
Choose a parametrization $f: \Sigma_t \rightarrow \overline{\G(\mathbf{\Sigma})}.$ Consider  the induced parametrization $y(f)=(-f)\circ s:\Sigma_{-t} \rightarrow -\overline{\G(\mathbf{\Sigma})}$ as in the previous proof. This provide isomorphisms
\begin{align}
\oplus_{i \in I^g} \text{Hom}(\mathbf{1},\Phi(t,i)) & \simeq \mathcal{T}^e(\overline{\G(\mathbf{\Sigma})}),
\\ \label{uuu}\oplus_{i \in I^g} \text{Hom}(\mathbf{1},\Phi(-t,i)) & \simeq \mathcal{T}^e(-\overline{\G(\mathbf{\Sigma})}).
\end{align}
Recall the isomorphism
\begin{equation}
\label{vvv} \mathcal{T}^e(r) \circ \tilde{h}\circ \dot{q}: \mathcal{T}^e(\overline{\G(-\mathbf{\Sigma})})\overset{\sim}{\longrightarrow} \mathcal{T}^e(-\overline{\G(\mathbf{\Sigma})})
\end{equation}
We can define $r$ such that the half-twists cancel with those of $s.$ More precisely there is a choice of convention such that the following holds.  Let $\tilde{s}$ be the same as $s$ but without the twists at the arcs. That is $\tilde{s}$ is only reflection in the $y$-plane. Observe that $f$ induce a parametrization
\[
(-f)\circ \tilde{s}: \Sigma_{t*} \rightarrow \overline{\G(-\mathbf{\Sigma})}
\] Take $u(f)=(-f)\circ \tilde{s}.$ We see that with respect to the parametrizations $u(f),y(f),$ the isomorphism
\[
\zeta: \Psi(t*) \overset{\sim}{\longrightarrow} \Psi(-t)
\] is given by postcomposing with $q.$  Now combine the description of the pairing $\Psi(t) \otimes \Psi(-t) \rightarrow K$ given in the previous proposition with the description of $\tilde{h}\circ \dot{q}$ as postcomposing with $q$ in each factor to obtain the desired presentation.
\end{proof}

We show next that the pairing is compatible with gluing. That is, we prove that the formula holds and explicitly calculate the $\mu_{\lambda}'$s. This calculation will depend on whether or not the two points subject to gluing are on the same connected component or not. 

\begin{pro}
Let $\mathbf{\Sigma}$ be a connected $I$-labeled marked surface obtained from gluing two points subject to gluing that lie on the same component. Consider the gluing isomorphism
\[
\tilde{g}: \bigoplus_{i \in I} Z(\mathbf{\Sigma}(i)) \overset{\sim}{\longrightarrow} Z(\mathbf{\Sigma}),
\] 
as desribed in definition \ref{defofmf}. We have
\[
(\tilde{g},\tilde{g})_{\mathbf{\Sigma}}= \sum_{i \in I} D^{4}\text{dim}(i)^{-1}(\emptyarg, \emptyarg)_{\mathbf{\Sigma}(i)}.
\]
\end{pro}

\begin{proof}

We will use the desription of the gluing homomorphism given in section \ref{dofg} above. To do this we need to compare two parametrizations of $-\overline{\G(\mathbf{\Sigma})}.$

Recall that to use (\ref{ggg}) in general, we start with a paramtrization $f: \Sigma_t \rightarrow \overline{\mathcal{G}(\mathbf{\Sigma}(i))}$ and provide a parametrization $F:\Sigma_{t'} \rightarrow \overline{\mathcal{G}(\mathbf{\Sigma})}$ that agrees with $f$ away from the points subject to gluing. Then (\ref{ggg}) holds with respect to the isomorphisms (\ref{paraf}),(\ref{paraF}). We will write $F=z(f).$ Recall that (\ref{presentation 1}) is only valied with respect to a pair of parametrizations $(r,y(r)).$ In order to use (\ref{presentation 1}) for the pairing on $\overline{\G(\mathbf{\Sigma})}$ we use the pair of parametrizations $(z(f),y(z(f)).$ For the pairing on $\overline{\G(\mathbf{\Sigma}(i))}$ we use the pair $(f,y(f)).$ Since (\ref{ggg}) only holds with respect to isomorphisms induced by compatible parametrizations, we need to compare $z(y(f))$ with $y(z(f)).$ We see that $z(y(f))$ is $y(z(f))$ followed by a Dehn twists at the attached handle. This implies that 
\begin{equation}
\label{zy=kyz}
(k_l Y',z(y(f)))= (Y', y(z(f))),
\end{equation}
for all $Y' \in \text{Hom}(\mathbf{1},\Phi(-t',l)) \subset \Psi(-t'),$ where $l$ is such that the cap corresponding to the points subject to gluing is colored with $V_l$ or $V_{l^*}.$ To verify (\ref{zy=kyz}) recall the description of how to pass from one parametrization to another given in section \ref{Review of the TQFT}, and use that a Dehn twists followed by a reflection in $y=0$ is the same as the same reflection followed by the reverse Dehn twist.

Using proposition \ref{rp}, equation (\ref{zy=kyz}) and equation (\ref{ggg}) we see that with respect to the isomorpshims $\Psi(t) \simeq Z(\mathbf{\Sigma}(i))$ and $\Psi(-t) \simeq \mathcal{T}^e(-\overline{\G(\mathbf{\Sigma}(i)})$ induced by the pair of parametrizations $(f,y(f)),$ we have the following presentation of $(\tilde{g}_i, k_l\tilde{g}_{l^*})$ 
\begin{center}
\includegraphics{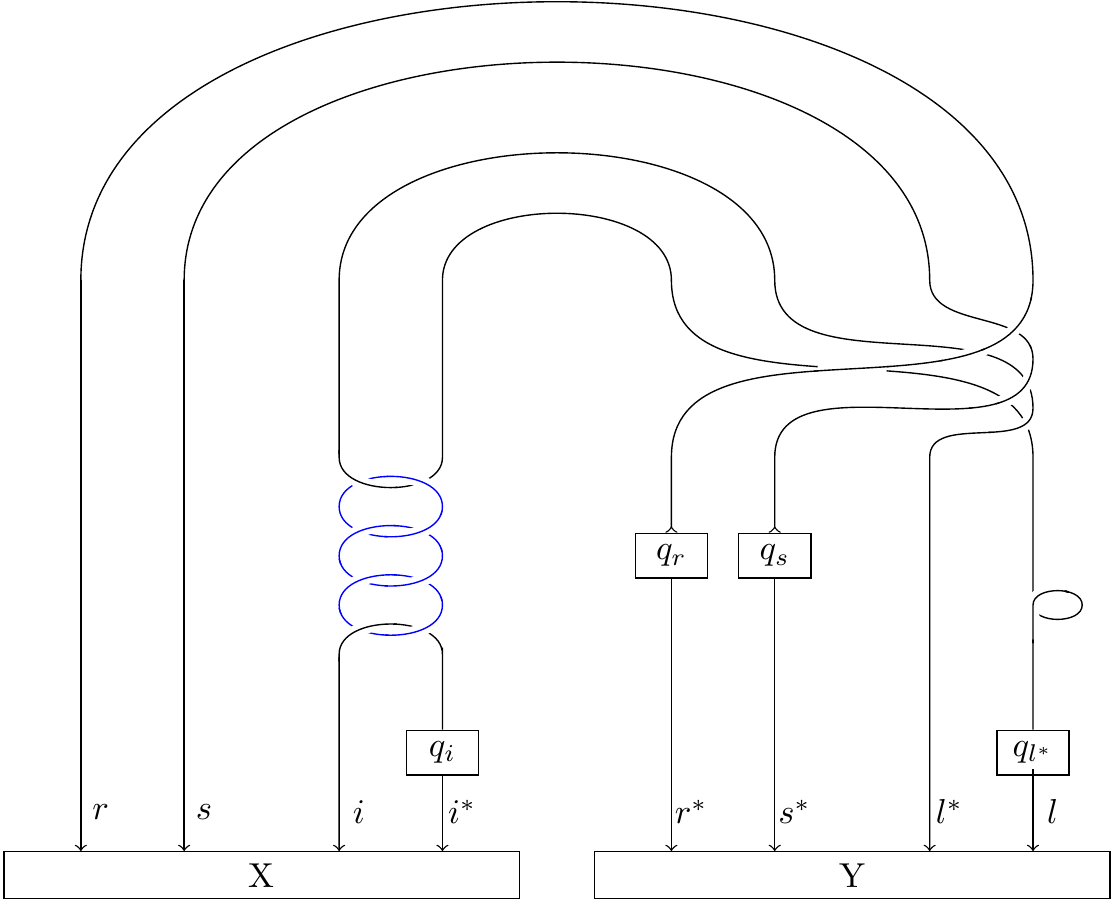}
\end{center} Here the blue unknot's are the surgery link components. For the sake of notational simplicity we have assumed that 
\[
t=(0;(V_{r},1),(V_{s},1),(V_i,1),(V_{i^*},1))
\]
The proof in the general case is easily obtained from this, as it relies on a local argument involving the surgery links.

As in the proof of theorem $10.4.1,$ we get the following local equality
\begin{center}
\includegraphics{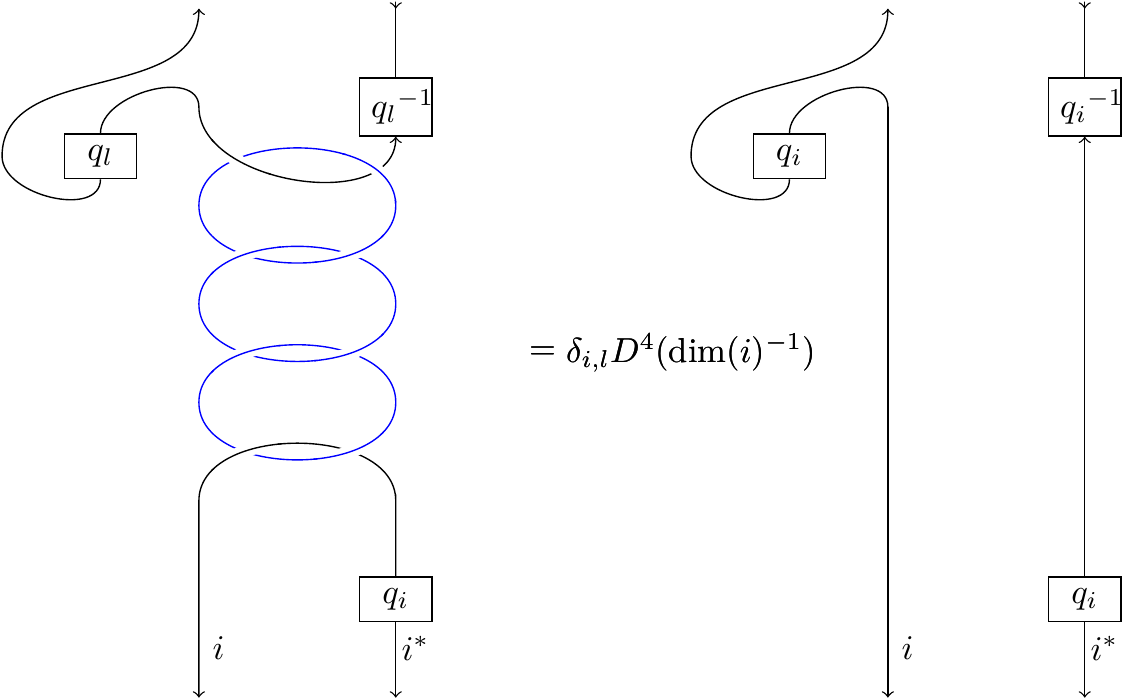}
\end{center}
Thus this pairing is zero unless $l^*=i^*.$ If so, we conclude that the claimed equation hold.
\end{proof}

\begin{pro}\label{distinctcomponentspairing}
Let $\mathbf{\Sigma}$ be a connected $I$-labeled marked surface obtained from gluing two points subject to gluing that lie on two distinct components. Consider the gluing isomorphism
\[
\tilde{g}: \bigoplus_{i \in I} Z(\mathbf{\Sigma}(i)) \overset{\sim}{\longrightarrow} Z(\mathbf{\Sigma}),
\] 
as desribed in definition \ref{defofmf}. We have
\[
(\tilde{g},\tilde{g})_{\mathbf{\Sigma}}= \sum_{i \in I} \text{dim}(i)^{-1}(\emptyarg, \emptyarg)_{\mathbf{\Sigma}(i)}.
\]
\end{pro}

\begin{proof}
Let $\mathbf{\Sigma}_1,$ be the component containing the first point and let $\mathbf{\Sigma}_2$ be the component containing the second point. We may assume that both of these components are homeomorphic to spheres. To see this, let $ x \in Z(\mathbf{\Sigma}_1)$ and let $y \in Z(\mathbf{\Sigma}_2).$ We want to compare $\langle x,y \rangle$ with $\langle \tilde{g}(x), \tilde{g}(y) \rangle.$ We can reduce the genus by $1$ on one of the components by factorization. That is, assume $\mathbf{\Sigma}_1 \sqcup \mathbf{\Sigma}_2$ is obtained from $\tilde{\mathbf{\Sigma}}$ by gluing, where the gluing increase the genus. That is, the points subject to gluing lie on the same component. Then we may assume $x=h(\tilde{x})$ and $y=h(\tilde{y}),$ where $h$ is the gluing homomorphism. Now the task is to identify a scalar $\lambda$ such that
\[
(h(\tilde{x}), h(\tilde{y}))=\lambda (\tilde{g}\circ h(\tilde{x}), \tilde{g}\circ h( \tilde{y})).
\]
But we already know from the previous proposition that there is a $C\in K^*$ such that
\[
(h(\tilde{x}), h(\tilde{y}))=C (\tilde{x},\tilde{y})
\]
and
\[
(\tilde{g}\circ h(\tilde{x}), \tilde{g}\circ h( \tilde{y}))=(h\circ \tilde{g}(\tilde{x}), h \circ \tilde{g}(\tilde{y})) =C (\tilde{g}(\tilde{x}), \tilde{g}(\tilde{y}))
\]
In the last equation we use that the pairing is compatible with morphisms and that gluing is associative. We now see that it will suffice to find $\lambda$ such that 
\[
\lambda(\tilde{g}(\tilde{x}), \tilde{g}(\tilde{y}))=(\tilde{x},\tilde{y}).
\]
We have reduced the genus by $1$ and can proceed inductively. Thus we can assume that we deal with spheres. Thus we can use the description of the gluing given above in section \ref{dofg}.

For the sake of notational simplicity, we illustrate the case where we have two spheres with three marked points.

As in the previous proposition one starts by observing $y(z'(f\otimes g))$ followed by a Dehn twists is $z'(y(f)\otimes y(g)).$ This will allow us to adopt the same strategy.

 We consider $\langle \tilde{g}_i(X\otimes Y), k_l\tilde{g}_{l^*}(X'\otimes Y') \rangle_{\mathbf{\Sigma}}.$ The following presentations shows that the given presentation above naturally factors as a composition $P(X,X')\circ Q(Y,Y')$ where $P(X,X')$ is an element of $\text{Hom}(V_i\otimes V_l^*,\mathbf{1}) $ and $Q(Y,Y') \in \text{Hom}(\mathbf{1}, V_i \otimes V_l^*).$

\begin{center}
\includegraphics{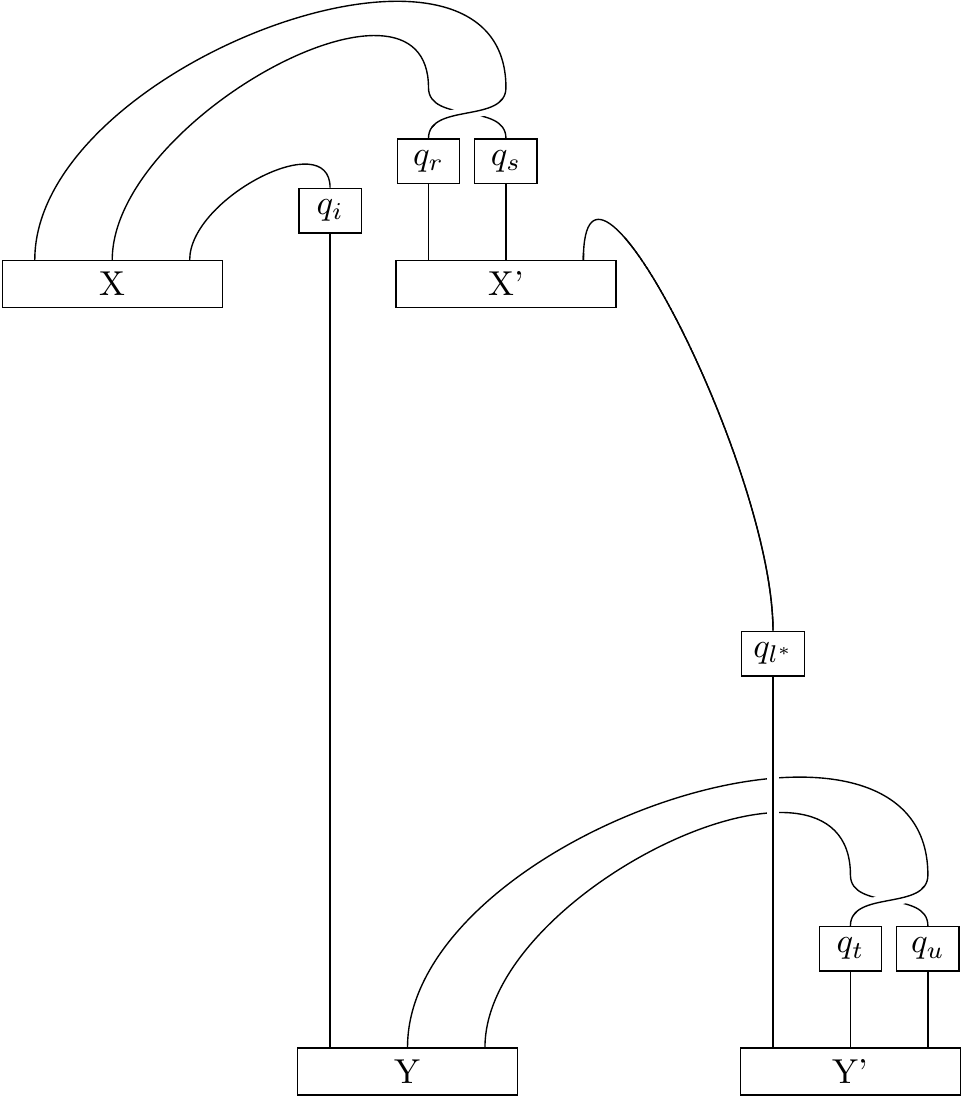}
\end{center}

Now the orthogonality follows from the fact that $\text{Hom}(V_i\otimes V_l^*,\mathbf{1}) $ is $\mathbf{0}$ if $l\not=i$ and isomorphic to $K$ otherwise. For $l=i$ we note the following equation that holds for all $f \in  \text{Hom}(V_i\otimes V_l^*,\mathbf{1})$

\begin{center}
\includegraphics{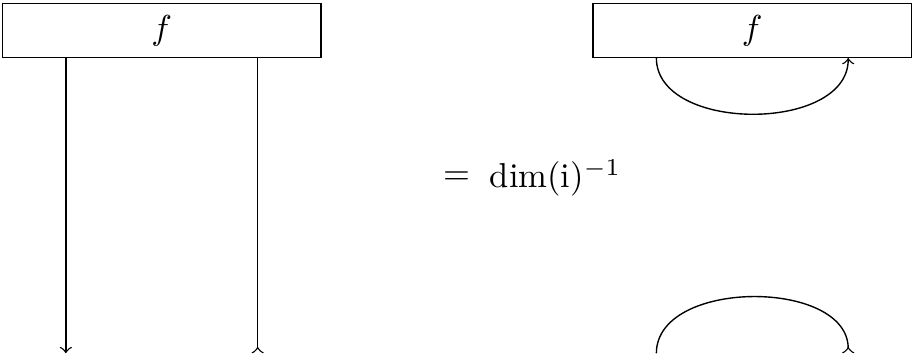}
\end{center}
Applying this to $P(X,X')$ and taking the twist that occurs into account when applying the isotopy to pull $q_i$ to the right of $q_s,$ we see that the claim holds.
\end{proof}

\begin{cor}[Compatibility of the pairing with gluing]
The pairing $(\emptyarg, \emptyarg)_{\mathbf{\Sigma}}$ is compatible with gluing. It can be rescaled to a pairing $\langle \emptyarg \mid \emptyarg \rangle_{\mathbf{\Sigma}}$ according to topological types, such that  
\[
\langle \tilde{g} \mid \tilde{g}\rangle_{\mathbf{\Sigma}}= \sum_{i \in I}\langle \emptyarg \mid  \emptyarg \rangle_{\mathbf{\Sigma}(i)}.
\] 
\end{cor}

It is easily verified that the following normalization has the given property. Since the pairing is multiplicative with respect to disjoint union, it is enough to specify the normalization on connected $\mathbf{\Sigma}.$ Assume therefore that $\mathbf{\Sigma}$ is of genus $g$ with labels $i_1,...,i_k.$ Then the normalization is given by
\begin{equation}
\label{normalization}
\langle\emptyarg \mid \emptyarg\rangle_{\mathbf{\Sigma}}=\left(D^{-4g} \prod_{l=1}^k \sqrt{\text{dim}(i_l)} \right)( \emptyarg , \emptyarg)_{\mathbf{\Sigma}}.
\end{equation}

It only remains to prove that the pairing is compatible with orientation reversal.

\begin{pro}[Compatibility with orientation reversal]
\label{compatibility with orientation reversal}

The two pairings $\langle \emptyarg \mid \emptyarg \rangle$ and $(\emptyarg, \emptyarg)$ are both compatible with orientation reversal. 
\end{pro}

\begin{proof} Since the normalization factor is the same for $\mathbf{\Sigma}$ and $-\mathbf{\Sigma}$ we see that it is enough to consider $(\emptyarg, \emptyarg).$

For $v \in Z(\mathbf{\Sigma})$ and $w \in Z(-\mathbf{\Sigma})$ we want to find a scalar $\mu$ such that $\mu (v,w)_{\mathbf{\Sigma}}= (w,v)_{\mathbf{\Sigma}}.$

For the moment let $\Sigma'$ be an extended surface.  Recall that the to use the presentation of the pairing as given in proposition (\ref{presentation 1}) we choose a parametrization $f$ of $\Sigma'$ and then we constructed a parametrization $y(f):=(-f)\circ s.$  These give isomorphisms $\mathcal{T}^e(\Sigma') \simeq \Psi(t_0)$ and $\mathcal{T}^e(-\Sigma') \simeq \Psi(-t_0).$ With respect to the parametrizations we can use the presentation of (\ref{presentation 1}). For $x\in \Psi(t_0)$ and $y\in \Psi(-t_0)$ it is an easy exercise to verify
\[
\langle x,y \rangle_{t_0} = \langle y,x \rangle_{-t_0}. 
\]

This identity is also necessary for self-duality, because if we take $y(f)$ as the parametrization $\Sigma_{-t_0} \rightarrow \mathbf{\Sigma}'$ then we see that $y(y(f))=f.$ This follows from the fact that $s^2=\text{id}$ which can be seen from the fact that counter clokwise half twists at the arcs followed by a reflection in the $y$-plane is the same as a reflection in the $y$-plane followed by clockwise half twists at the arcs.

Now choose a parametrization $f: \Sigma_t \rightarrow \overline{\mathcal{G}(\mathbf{\Sigma})}.$  This induce a parametrization
$u(f)=(-f)\circ \tilde{s}: \Sigma_{t*} \rightarrow \overline{\mathcal{G}(-\mathbf{\Sigma})}.$ Here $\tilde{s}$ is simply the reflection in the $y$-plane. With respect to these isomorphisms we see that $\langle \emptyarg , \emptyarg \rangle_{\mathbf{\Sigma}}$ is given as $\langle \emptyarg , \dot{q} \rangle_{t},$ where $\dot{q}$ is given by postcomposing suitably in each factor of the tensor product.

Now choose $g=u(f): \Sigma_{t*} \rightarrow \overline{\mathcal{G}(-\mathbf{\Sigma})}.$ Observe $u(g)=f$ Thus for $(v,w) \in \Psi(t) \times \Psi(t*)$ we simply need to compare $\langle v, \dot{q}(w)\rangle_{t}$ with $\langle w, \dot{q}(v) \rangle_{t*}.$

Assume the labeled marked points of $\mathbf{\Sigma}$ are $i_1,...,i_k.$ Let $\mu(i) \in K^*$ be defined by the following equation

\begin{center}
\includegraphics{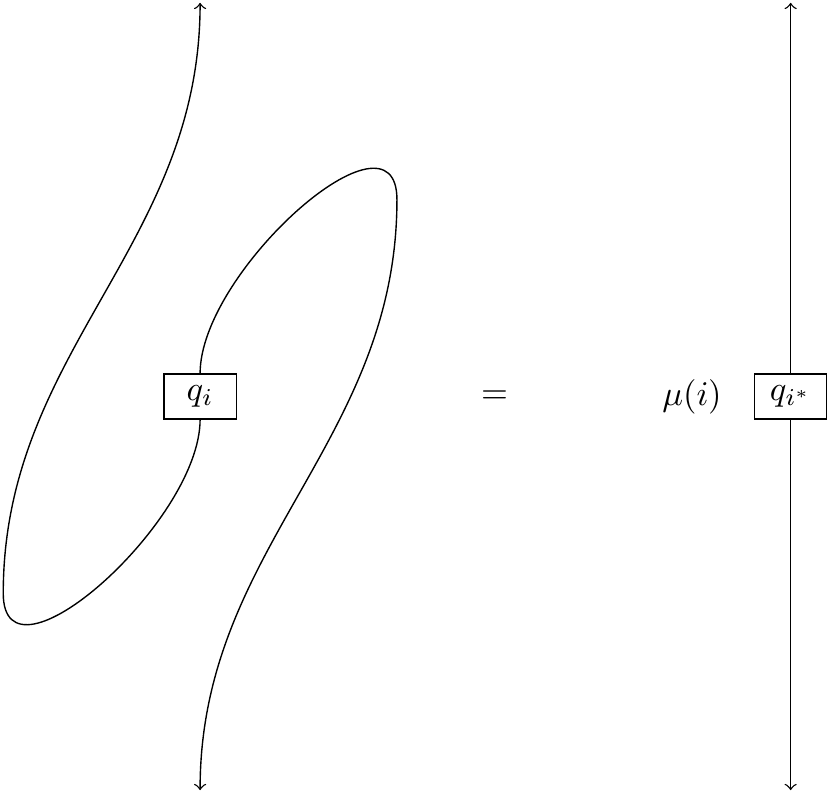}
\end{center}

Let $\mu = \mu_{i_1} \cdots \mu_{i_k}.$ Now use the fact that $\langle v, \dot{q}(w) \rangle_{t}=\langle \dot{q}(w),v\rangle_{-t}.$ Now use the surgery presentation given in proposition (\ref{presentation 1}). In the presentation of $\langle \dot{q}(w),v\rangle_{-t}$ pull over the coupons colored with $q_{i_l}$ from left to right to obtain 
\[
( v,w   )_{\mathbf{\Sigma}} =\mu (w,v)_{-\mathbf{\Sigma}}
\] 
This finishes the proof. \end{proof}

\paragraph*{\textbf{Remark.}} We observe that if $i\not=i^*$ it is possible to consistently choose $q_i$ and $q_{i^*}$ such that $\mu(i)$ and $\mu(i^*)$ takes any values, as long as $\mu(i)\mu(i^*)=1.$ This follows from the fact that turning a coupon upside down, and then turning the resulting morphism upside down will yield the original morphism. Call this operation $F.$ Then the equation above reads $F(q_i)=\mu(i)q_{i^*}.$ Similarly it can be seen that if $i^*=i$ then we must have $\mu(i)^2=1.$ Using the axioms for the unit object of a modular tensor category, it is also easily seen that $\mu(0)=1.$

We note the following result, that allow us to define $\mu$ on the self-dual objects independently of $q.$

\begin{pro}[$\mu$ is weldefined on self-dual objects] Assume that $i\in I$ satisfies $i=i^*.$ Then $\mu(i)$ is independent of $q_i.$
\end{pro}

The fact that $\mu(i)$ might be $-1$ for $i=i^*$ leads us to consider the strict self-duality question in section \ref{FG of a MTC}, where we introduce a new algebraic concept associated to a modular tensor category. As will be clear below, this will in may cases produce a very interesting normalization of the duality pairing, that will be strictly self-dual. 

\section{Unitarity}

Consider a complex vector space $W$ with scalar multiplication $(\lambda,w) \mapsto \lambda.w$ Let $\overline{W}$ be the complex vector space with the same underlying set and scalar multiplication given by $(\lambda,w) \mapsto \overline{\lambda}.w.$ Here $\overline{\lambda}$ is the complex conjugate of $\lambda.$

\begin{df}[Unitarity]
\label{unitarity} Let $(V,g)$ be a modular functor based on $\Lambda$ and $\C.$ A unitary structure on $V$ is a positive definite hermitian form
\[
(\emptyarg, \emptyarg)_{\mathbf{\Sigma}}: V(\mathbf{\Sigma})\otimes \overline{V(\mathbf{\Sigma})} \rightarrow \C,
\]
subject to the following axioms.

\paragraph{Naturality.} Let $\mathbf{f}=(f,s): \mathbf{\Sigma}_1 \rightarrow \mathbf{\Sigma}_2$ be a morphism between $\Lambda$-labeled marked surfaces. Then
\begin{equation}
(V(f),V(f))_{\mathbf{\Sigma}_2} =(\emptyarg, \emptyarg)_{\mathbf{\Sigma}_1}
\end{equation}
\paragraph{Compatibility with disjoint union.} Consider a disjoint union of $\Lambda$-labeled marked surface $\mathbf{\Sigma}=\mathbf{\Sigma}_1 \sqcup \mathbf{\Sigma}_2.$ Composing with a permutation of the factors, the modular functor $V$ provide an isomorphism
\[
\eta:V(\mathbf{\Sigma})\otimes \overline{V(\mathbf{\Sigma})} \overset{\sim}{\longrightarrow} V(\mathbf{\Sigma}_1)\otimes \overline{V(\mathbf{\Sigma}_1)} \otimes V(\mathbf{\Sigma}_2) \otimes \overline{V(\mathbf{\Sigma}_2)}.
\]
We demand that with respect to the natural isomorphism $\C \otimes \C \simeq \C$ we have:
\begin{equation}
( \emptyarg, \emptyarg)_{\mathbf{\Sigma}}=\left((\emptyarg, \emptyarg)_{\mathbf{\Sigma}_1} \otimes (\emptyarg, \emptyarg)_{\mathbf{\Sigma}_2}\right) \circ \eta.
\end{equation}.

\paragraph{Compatibility with gluing.} Let $\mathbf{\Sigma}$ be a $\Lambda$ labeled marked surface obtained from gluing. Consider the gluing isomorphism
\[
g: \bigoplus_{\lambda \in \Lambda} V(\mathbf{\Sigma}(\lambda)) \overset{\sim}{\longrightarrow} V(\mathbf{\Sigma}),
\] 
as desribed in the definition of a modular functor. Clearly $g$ also induce an isomorphism
\[
g: \bigoplus_{\lambda \in \Lambda}\overline{V(\mathbf{\Sigma}(\lambda))} \overset{\sim}{\longrightarrow} \overline{V(\mathbf{\Sigma})}.
\] We have
\begin{equation}
(g,g)_{\mathbf{\Sigma}}= \sum_{ \lambda \in \Lambda} \mu_{\lambda}(\emptyarg, \emptyarg)_{\mathbf{\Sigma}(\lambda)},
\end{equation}
where $\mu_{\lambda} \in \R_{>0}$ for all $\lambda.$

If the modular functor $(V,g)$ also has a duality pairing we demand the unitary structure and the duality is compatible in the following sense.

\paragraph{Compatibility with duality.}

For all labeled marked surfaces $\mathbf{\Sigma}$, we demand that the following diagram is commutative up to a scalar $\rho(\mathbf{\Sigma})$ depending only on the isomorphism class of $\mathbf{\Sigma}.$

\begin{equation}
\label{unidua}
\begin{tikzcd}[swap]
V(\mathbf{\Sigma}) \arrow{r}[swap]{\simeq}
\arrow{d}{\simeq}
& V(-\mathbf{\Sigma})^* \arrow{d}[swap]{\simeq}
\\ \overline{V(\mathbf{\Sigma})^*} \arrow{r}[swap]{\simeq} & \overline{V(-\mathbf{\Sigma})}
\end{tikzcd}
\end{equation}

Here the horizontal isomorphisms are induced by the duality pairing whereas the vertical isomorphisms are induced by the unitary structure.

\end{df}

We allow the $\mu_{\lambda}$ to depend on the isomorphism class of $(\mathbf{\Sigma},(p,q)).$

We now make explicit what the isomorphisms of the diagram $(\ref{unidua})$ are. We start with composition
\[
\omega: V(\mathbf{\Sigma}) \overset{\simeq}{\longrightarrow} V(-\mathbf{\Sigma})^* \overset{\simeq}{\longrightarrow} \overline{V(-\mathbf{\Sigma})}.
\]
Let $\langle \emptyarg, \emptyarg \rangle$ be the duality pairing. Let $(\emptyarg, \emptyarg)$ be the Hermitian form. The first map is given by
\[
V(\mathbf{\Sigma}) \ni f \mapsto \langle \emptyarg , f \rangle_{-\mathbf{\Sigma}} : V(-\mathbf{\Sigma}) \rightarrow \C.
\]

The second map is the inverse of the linear isomorphism $\overline{V(-\mathbf{\Sigma})}\overset{\simeq}{\longrightarrow} V(-\mathbf{\Sigma})^*$ given by
\[
\overline{V(-\mathbf{\Sigma})} \ni u \mapsto ( \emptyarg , u )_{-\mathbf{\Sigma}} : V(-\mathbf{\Sigma}) \rightarrow \C.
\]

Thus $\omega(f)$ is defined by
\begin{equation}\label{omega}
\langle x, f \rangle_{-\mathbf{\Sigma}} = (x, \omega(f))_{-\mathbf{\Sigma}},
\end{equation}
for all $x$ in $V(-\mathbf{\Sigma}).$

We now consider the composition
\[
\phi: V(\mathbf{\Sigma}) \overset{\simeq}{\longrightarrow} \overline{V(\mathbf{\Sigma})^*} \overset{\simeq}{\longrightarrow} \overline{V(-\mathbf{\Sigma})}.
\]

The first is the linear map

\[
V(\mathbf{\Sigma}) \ni f \mapsto ( \emptyarg , f )_{-\mathbf{\Sigma}} : V(\mathbf{\Sigma}) \rightarrow \C.
\]
The second map is the inverse of the linear isomorphism $\overline{V(-\mathbf{\Sigma})}\overset{\simeq}{\longrightarrow} \overline{V(\mathbf{\Sigma})^*}$ given by
\[
\overline{V(-\mathbf{\Sigma})} \ni u \mapsto \langle \emptyarg , u \rangle_{-\mathbf{\Sigma}} : V(\mathbf{\Sigma}) \rightarrow \C.
\]

Thus $\phi(f)$ is defined by
\begin{equation}\label{phi}
\langle y, \phi(f) \rangle_{\mathbf{\Sigma}} = (y, f)_{\mathbf{\Sigma}},
\end{equation}
for all $y$ in $V(\mathbf{\Sigma}).$

Projective commutativity of $(\ref{unidua})$ can be reformulated as the existence of $\rho(\mathbf{\Sigma})$ in $ \C$ with
\begin{equation}\label{projunidua}
\phi= \rho(\mathbf{\Sigma}) \omega.
\end{equation}

\section{Unitary structure from a unitary MTC}

Recall the definition of a unitary modular tensor category $(\mathcal{V},(V_i)_{i\in I})$ with conjugation $f\mapsto \overline{f}$ as defined in section $5.5.$ of chapter $\rom{5}$ in \cite{Tu}. Recall that $K=\C$ in this case.

Assume we are given a unitary modular tensor category. For an $e$-surface $\Sigma$ let $(\emptyarg, \emptyarg)_{\Sigma}$ be the Hermitian form on $\mathcal{T}^e(\Sigma)$ as defined in section $10$ of chapter $\rom{4}$ in \cite{Tu}.

\begin{thm}[Unitarity]
\label{HT}
Let $(\mathcal{V},(V_i)_{i\in I})$ be a unitary modular tensor category. Let $\mathbf{\Sigma}$ be an $I$-labeled marked surface. Consider the positive definit Hermitian form
\[
(\emptyarg, \emptyarg)_{\mathbf{\Sigma}}=(\emptyarg, \emptyarg)_{\overline{\G(\mathbf{\Sigma})}}.
\]
This defines a unitary structure on $Z_{\mathcal{V}}$ compatible with duality.
\end{thm}

\begin{proof}
 It is proven by Turaev, that the induced Hermitian form is natural with respect to weak $e$-morphisms, and that it is multiplicative with respect to disjoint union. As Turaev also proves that $\overline{\Delta^{-1}D}=(\Delta^{-1}D)^{-1}$ these two properties carry over. All of this is proven in section $10$ of chapter $\rom{4}.$

Let us now prove that it is compatible with gluing. We first consider the case where the two points lie on the same component. Since the gluing as well as the Hermitian form is multiplicative with respect to disjoint union, as well as natural with respect to morphisms, we may assume that we are in the situation desribed in section \ref{dofg}. We adopt the the notation form the first subsection of this section.

It follows directly from theorem $10.4.1$ in section $10.4$ of chapter $\rom{4}$ in \cite{Tu} that if $i\not=j$ then $(\tilde{g}_i,\tilde{g}_j)_{\mathbf{\Sigma}}=0.$ Let $x,y \in \text{Hom}(\mathbf{1},\Phi(t,j)) \subset Z(\mathbf{\Sigma}(i)).$  Using theorem $10.4.1,$ linearity of $\text{tr}$ and $\C \simeq \text{End}(V_{i^*}),$ we get that
\begin{align*}
(\tilde{g}_i(x),\tilde{g}_i(y))_{\mathbf{\Sigma}} &=D^{g}\left(\text{dim}(i)\prod_{c=1}^{g} \text{dim}(j_c)\right)^{-1}\text{tr}(\tilde{g}_i(x)\circ \overline{\tilde{g}_i(y)}).
\end{align*}
Here $g$ is the genus of $\mathbf{\Sigma}.$ Unwinding the gluing formula and using properties of the conjugation as well as of the trace we get that 
\begin{align*} & D^{g}(\left(\text{dim}(i)\text{dim}(j)\right)^{-1} \text{tr}((1_{W}\otimes q_{i} \otimes 1_{R})\circ x\circ \overline{y} \circ (1_{W}\otimes \overline{q_{i}} \otimes 1_{R}))
\\ &= D^{g}(\left(\text{dim}(i)\text{dim}(j)\right)^{-1} \text{tr}((1_{W}\otimes \overline{q_i}\circ q_{i} \otimes 1_{R})\circ x \circ \overline{y} )
\\ &= D\text{dim}(i)^{-1} \lambda_i (x,y)_{\mathbf{\Sigma(i)}}.
\end{align*}
Here $\lambda_i \in\C$ is defined by
\begin{equation}
\lambda_i1_{V_{i^*}}= \overline{q_i}\circ q_i.
\end{equation}
Thus we get
\begin{equation}
(\tilde{g},\tilde{g})_{\mathbf{\Sigma}}=\sum_{i \in I} D\text{dim}(i)^{-1} \lambda_i(\emptyarg, \emptyarg)_{\mathbf{\Sigma}(i)}.
\end{equation}

We now consider the case where the two points subject to gluing lie on distinct components. Using the result above, we may assume that these are homeomorphic to spheres. This can be argued as in the proof of Proposition (\ref{distinctcomponentspairing}). Using naturality and multiplicativity of the gluing as well as of the Hermitian form, we may assume that we are in the situation of the second subsection in section \ref{dofg}. An argument based on the surgery presentation of the form given in the proof of theorem $10.4.1$ and based on the ideas of section $10.6$ will show that in this case we have that
\begin{equation}
(\tilde{g},\tilde{g})_{\mathbf{\Sigma}}=\sum_{i \in I} \text{dim}(i)^{-1} \lambda_i(\emptyarg, \emptyarg)_{\mathbf{\Sigma}(i)}.
\end{equation}

Finally we prove that the unitary structure is compatible with duality.

This is done by considering surgery presentations of the equations (\ref{omega}),(\ref{phi}). Let $f \in Z(\mathbf{\Sigma}).$ We may assume $\mathbf{\Sigma}$ is conneced. Equation (\ref{omega}) is presented as an equation involving  $\overline{\omega(f)}$ and equation (\ref{phi}) is presented an equation involving $\phi(f).$ Conjugating the surgery presentation of (\ref{omega}) we see that $\phi(f)=\frac{1}{\sigma(i_1)\cdots \sigma(i_m)}\omega(f)$ where $\sigma(i)$ is defined as $\lambda_{i^*}\mu(i).$ Thus, if $\mathbf{\Sigma}$ is an $I$-labeled marked surface (not necessarily conneced) with labels $i_1,...,i_m$ we see that
\begin{equation}
\rho(\mathbf{\Sigma})= \prod_{l=1}^{m}\sigma(i_l)^{-1}.
\end{equation}\end{proof}

\section{Scaling of the duality, unitarity and gluing}
\label{scaling section}

Fix $q:=(q_i)_{i \in I}$ where $q_i:V_{i^*} \rightarrow {V_i}^*$ is an isomorphism. Let $\tilde{g}$ be the gluing defined using $q.$ Let $\langle \emptyarg, \emptyarg \rangle$ be the duality pairing defined using $q.$ Let $k_i$ be the twist coefficients. For the remainder of this article we fix for all $i$ a choice of $\sqrt{\text{dim}(i)}$ and a choice of $\sqrt{k_i}.$ We make these choices invariant under $i \mapsto i^*.$ If $K=\C$ and $\text{dim}(i)$ is positive, we of choose the positive square root. This will be the case if $\mathcal{V}$ is assumed to be unitary. We recall that if this is the case then $k_i \in S^1.$

Let $K^*$ be the units in $K.$ For $u \in {K^*}^{I}$ let $q(u):=(u_iq_i)_{i\in I}.$ Let $\tilde{g}_u$ be the gluing defined using $q(u).$ Let $\langle \emptyarg, \emptyarg\rangle^{u}$ be the pairing defined using $q(u).$ We start by observing that if $\mathbf{\Sigma}$ has labels $i_1,...,i_m$ then we have 
\begin{equation}
\langle \emptyarg, \emptyarg\rangle^{u}_{\mathbf{\Sigma}}=\left(\prod_{l=1}^m u_{i_l} \right) \langle \emptyarg, \emptyarg \rangle_{\mathbf{\Sigma}}.
\end{equation}

We will write \[
\langle \emptyarg, \emptyarg\rangle^{u}_{\mathbf{\Sigma}}=u(\mathbf{\Sigma}) \langle \emptyarg, \emptyarg \rangle_{\mathbf{\Sigma}}.
\]

Let $u,w \in {K^*}^{I}.$ Below we will consider what happens, if we use $q(u)$ to define the gluing, and $q(w)$.

\begin{df}[Genus normalized pairing]
\label{genus normalized pairing}
Let $w \in {K^*}^{I}.$ For a surface of genus $g$ we consider the following normalization
\[
\langle \emptyarg , \emptyarg \rangle_{*, \mathbf{\Sigma}}^w:= D^{-4g} \langle \emptyarg, \emptyarg\rangle^{w}.
\]
\end{df}

Consider a general modular functor $V$ with label set $I.$ Assume $V$ has a duality pairing $\langle \emptyarg , \emptyarg \rangle.$  consider $S^2$ equipped with the Stokes orientation, where $B^3$ is given the RHS orientation. Let $(S^2,i,j)$ have the northpole colored with $i$ and the southpole colored with $j.$ There is a natural isotopy to the standard decorated surface of type $(0;(V_i,1),(V_j,1)).$ This induce an isomorphism $Z(S^2,i,j) \simeq \text{Hom}(1,V_i\otimes V_j).$ Let $\omega(i)$ be the unique vector in $Z(S^2,i,i^*)$ that solves $\tilde{g}(\omega(i) \otimes \omega(i))=\omega(i).$ Let  $\zeta(i) \in Z(-(S^2,i,i^*))$ be the unique vector that solves the analogous gluing problem. We define
\begin{equation}
V(i):= \langle \omega(i), \zeta(i) \rangle_{(\mathbf{S}^2,i,i^*)}.
\end{equation}

We now return to $Z_{\mathcal{V}}.$
\begin{pro}
 Under the isomorphism $\text{Hom}(\mathbf{1},V_i\otimes V_{i^*}) \simeq Z(S^2,i,i^*)$ induced by the identity parametrization we see that $\omega(i)$ is given as
\begin{center}
\includegraphics{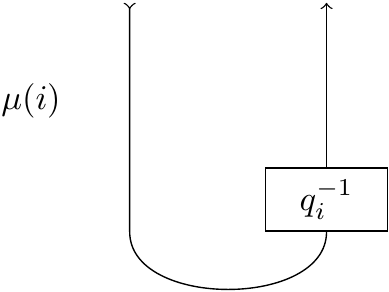}
\end{center}
\end{pro}

\begin{pro}
\label{quantum inv.}
 We have that
\begin{equation}
\langle \omega(i), \zeta(i) \rangle =k_i^{-1} \text{dim}(i).
\end{equation}
\end{pro}

 We start by observing that if $\mathbf{\Sigma}$ has labels $i_1,...,i_m$ then we have 
\begin{equation}
\langle \emptyarg, \emptyarg\rangle^{w}_{\mathbf{\Sigma}}=\left(\prod_{l=1}^m w_{i_l} \right) \langle \emptyarg, \emptyarg \rangle_{\mathbf{\Sigma}}.
\end{equation}

We will write \[
\langle \emptyarg, \emptyarg\rangle^{w}_{\mathbf{\Sigma}}=w(\mathbf{\Sigma}) \langle \emptyarg, \emptyarg \rangle_{\mathbf{\Sigma}}.
\]

\textbf{W: I have correceted the formulas below}

\begin{pro}
Assume $i,i^*$ lie on the same component. We have
\begin{equation}
\label{samecom}
\langle\tilde{g}_u^{i},\tilde{g}_{u}^{i^*}\rangle_{\Sigma}^{w}=\frac{u_iu_{i^*}}{w_iw_{i^*}}D^4 \text{dim}(i)^{-1}\langle \emptyarg, \emptyarg\rangle^w_{\Sigma(i)}.
\end{equation}
Assume $i,i^*$ lie on distinct components. We have
\begin{equation}
\label{difcom}
\langle\tilde{g}_u^{i},\tilde{g}_{u}^{i^*}\rangle_{\Sigma}^{w}=\frac{u_iu_{i^*}}{w_iw_{i^*}} \text{dim}(i)^{-1}\langle \emptyarg, \emptyarg\rangle_{\Sigma(i)}^w.
\end{equation}
\end{pro}

We want to know how scaling affects the self-duality scalar $\mu.$

\begin{pro}
\label{scalingofmu}
We have \begin{equation}
\mu(i,w)= \frac{w_i}{w_{i^*}}\mu(i).
\end{equation}
\end{pro}

Let $\omega(i,u)\in Z(S^2,i,i^*)$ be the unique vector that solves the analogous equation 
$\tilde{g}_u(\omega(i,u)\otimes \omega(i,u))=\omega(i,u).$ Then 
\begin{equation}
\omega(i,u)=u_i^{-1}\omega(i).
\end{equation}

Let $\zeta(i) \in Z(-(S^2,i,i^*))=Z(-S^2,i^*,i)$ be the image of $\omega(i,u)$ under $Z(r).$ Again we have

\begin{equation}
\zeta(i,u)=u_i^{-1}\zeta(i).
\end{equation}

We now investigate how this affect the compatibility of the unitary structure with gluing.

\begin{pro}
Let $\mathbf{\Sigma}$ be a connected $\Lambda$-labeled marked surface obtained from gluing two points subject to gluing that lie on one and the same component. Consider the gluing isomorphism:
\[
\tilde{g}_u: \bigoplus_{i \in I} Z(\mathbf{\Sigma}(i)) \overset{\sim}{\longrightarrow} Z(\mathbf{\Sigma}),
\] 
 We have
\begin{equation}
(\tilde{g}_u,\tilde{g}_u)_{\mathbf{\Sigma}}=\sum_{i \in I} D^4\text{dim}(i)^{-1} \lambda_{i}u_i\overline{u_i}(\emptyarg, \emptyarg)_{\mathbf{\Sigma}(i)}.
\end{equation}
\end{pro}

\begin{pro}
Let $\mathbf{\Sigma}$ be a connected $\Lambda$-labeled marked surface obtained from gluing two points subject to gluing that lie on two distinct components. Consider the gluing isomorphism:
\[
\tilde{g}_u: \bigoplus_{i \in I} Z(\mathbf{\Sigma}(i)) \overset{\sim}{\longrightarrow} Z(\mathbf{\Sigma}),
\]  We have
\begin{equation}
(\tilde{g}_u,\tilde{g}_u)_{\mathbf{\Sigma}}=\sum_{i \in I} \text{dim}(i)^{-1} \lambda_{i}u_i\overline{u_i}(\emptyarg, \emptyarg)_{\mathbf{\Sigma}(i)}.
\end{equation}
\end{pro}

\subsection{Scaled  normalization}
\label{scaled normalization section}
We now consider what happens if we scale the gluing and the pairing as above by using possibly different $q'$s and then normalize the pairing such that they are strictly compatible.

Let $\langle \emptyarg \mid \emptyarg \rangle^{u,w}$ be the normalized pairing given by
\begin{equation}
\label{scalenormalization}
\langle\emptyarg \mid \emptyarg\rangle_{\Sigma_g, i_1,...,i_k}^{u,w}=\left(D^{-4g} \prod_{l=1}^k s_{i_l}^{u,w} \right)\langle \emptyarg , \emptyarg\rangle_{\Sigma_g, i_1,...,i_k}^w,
\end{equation}
where
\begin{equation}
\label{s_i}
s_i^{u,w}:=\frac{\sqrt{w_iw_{i^*}}\sqrt{\text{dim}(i)}}{\sqrt{u_iu_{i^*}}},
\end{equation}

We will write

\[
\langle\emptyarg \mid \emptyarg\rangle_{\Sigma_g, i_1,...,i_k}^{u,w}:=s(\mathbf{\Sigma},u,w) \langle \emptyarg, \emptyarg \rangle.
\]
If we want to stress the choice of square roots chosen for $w_iw_i^*$ and $u_iu_i^*$ we will write 
\[
\langle \emptyarg \mid \emptyarg \rangle^{u,w,S}.
\]

Observe that \[
s(\mathbf{\Sigma},u,w)=\left(D^{-4g} \prod_{l=1}^k s_{i_l}^{u,w}w_{i_l} \right),
\]
when $\mathbf{\Sigma}=(\Sigma_g, i_1,...,i_k).$

We now consider a normalization of the Hermitian form. This normalized Hermitian form will be strictly compatible with the gluing $\tilde{g}_u.$ First we note that

\begin{equation}
\label{uninormalization}
(\emptyarg \mid \emptyarg)_{\mathbf{\Sigma}}^u=\left(D^{-4g} \prod_{l=1}^k r_{i_l}^{u} \right)(\emptyarg , \emptyarg)_{\mathbf{\Sigma}},
\end{equation}
where 
\begin{equation} r_i^u=\frac{\sqrt{\text{dim}(i)}}{\sqrt{\lambda_{i}u_{i}\overline{u_{i}}}}.
\end{equation}
Recall the convention that the square root of any positive number is assumed to be chosen positive. Thus there is no ambiguity in chosing the $r_i.$ Of course we have to choose them positively, if we want the pairing to remain positive definite. We will write
\[
(\emptyarg \mid \emptyarg)_{\mathbf{\Sigma}}^u=r(\mathbf{\Sigma},u)(\emptyarg , \emptyarg)_{\mathbf{\Sigma}}
\]

\begin{pro}
\label{RO}
Assume $\mathbf{\Sigma}$ is an $I$-labeled marked surface (not necessarily conneced) with labels $i_1,...,i_m.$ With respect to the normalized duality $\langle \emptyarg \mid \emptyarg \rangle^{u,w}$ and the normalized Hermitian form $(\emptyarg \mid \emptyarg)^u$ we have the following equation
\begin{equation}
\rho_N^{u,w}(\mathbf{\Sigma})= \left(r(\mathbf{\Sigma},u)\overline{r(-\mathbf{\Sigma},u)}\right)\left(s(\mathbf{\Sigma},u,w)\overline{s(-\mathbf{\Sigma},u,w)}\right)^{-1}\prod_{l=1}^{m}\sigma(i_l)^{-1}.
\end{equation}
With respect to the genus normalized pairing $\langle \emptyarg , \emptyarg \rangle_{*}^w$ we have the following equation
\begin{equation}
\rho_{g,N}^{u,w}(\mathbf{\Sigma})= \prod_{l=1}^{m}\left( r^u_{i_l}r^u_{{i_l}^*}\right)\left(\sigma(i_l)w_{i_l} \overline{w_{{i_l}^*}}\right)^{-1}
\end{equation}
\end{pro}

Observe that since $r$ is always real (and positive) we have $\overline{r(-\mathbf{\Sigma},u)}=r(-\mathbf{\Sigma}).$

\section{The canonical symplectic rescalling}

Assume in the following that $K$ is an integral domain. Fix a set of isomorphisms $q_i: V_{i^*} \overset{\sim}{\longrightarrow} V_i^*$ such that $\mu(i)=1$ for all $i$ with $i\not=i^*.$ In this section, we only consider scalings $\lambda:I \rightarrow K^*$ that satisfies $\lambda_i=\lambda_{i^*}.$ Recall the definition of the normalization factors: $s_i^{u,w}:=\frac{\sqrt{w_iw_{i^*}}\sqrt{\text{dim}(i)}}{\sqrt{u_iu_{i^*}}}.$ If $u,w$ are invariant under $i \mapsto i^*$ we see that we have canonical square roots given by $\sqrt{X^2}=X.$ In the following, these coefficients shall be interpreted according to this. 

 \begin{df}[Symplectic labels]
 \label{symplectic label}
 Let $i \in I$ satisfy $i=i^*.$ We say that $i$ is symplectic if $\mu(i)=-1.$
 \end{df}

\begin{df}[Symplectic multiplicity]
\label{symplectic multiplicity}
Let $\mathbf{\Sigma}$ be a labeled marked surface. Let $\nu(\mathbf{\Sigma})$ denote the number of marked points on $\mathbf{\Sigma}$ labeled with symplectic labels. We call this number the symplectic multiplicity.
\end{df}

\begin{thm}[Canonical symplectic scalling]
\label{Canonical solution}
Choose $u,w \in (K^*)^I$ that solves 
\begin{equation}
\label{canonical}
u_i= s_i^{u,w} w_i
\end{equation}
for all $i.$ Then the modular functor $Z_{\mathcal{V}}$ with gluing $\tilde{g}(u)$ and genus normalized duality $\langle \emptyarg , \emptyarg \rangle^{u}_{*}$ satisfies that gluing and duality are strictly compatible and the duality is self-dual up to a sign which is given by the symplectic multiplicity
\begin{equation}\label{prjs}
\mu = (-1)^{\nu}.
\end{equation}
We have
\begin{equation}
\label{3}
Z(i):=\frac{\text{dim}(i)}{k_i}.
\end{equation}
Moreover, any two solutions $(u,w)$ and $(u',w')$ result in modular functors with duality that are isomorphic  through an isomorphism that preserve the duality pairing.
\end{thm}

\noindent{\bf Remark} {\em We emphasise that equation (\ref{canonical}) means that one uses the same scaling for the glueing isomorphism as one uses in the duality paring.}

Before commencing the proof, we observe that up to a sign there is a preferred solution given by choosing $w_i=1$ for all $i$ and solving 
\[
\frac{\sqrt{\text{dim}(i)}}{u_i}=u_i.
\]
If there is no such $u_i$ we may formally add it.  If $(u,w)$ is a solution, we will write $Z^u$ for the resulting modular functor with duality $\left(Z, \tilde{g}(u), \langle \emptyarg , \emptyarg \rangle_{*}^u\right).$

\begin{proof}
Let $(u,w)$ be a solution. The fact that this is a solution to (\ref{canonical}) implies that
\[
\langle \emptyarg , \emptyarg \rangle_{*}^u= \langle \emptyarg \mid \emptyarg \rangle^{u,w}.
\]
Since the bracket on the left is strictly compatible with $\tilde{g}(u)$ the first claim follows. 

Equation (\ref{prjs}) is an easy consequence of $w_i=w_{i^*},$ proposition \ref{scalingofmu} and the proof of proposition \ref{compatibility with orientation reversal}.

Equation (\ref{3}) follows from the fact that a proof of proposition \ref{quantum inv.} only depends on the fact that the same set of isomorphisms $V_{i^*} \overset{\sim}{\longrightarrow} V_i^*$ is used for the duality as well as for the gluing, and that we always have $\mu(i,w)\mu(i^*,w)=1$ for all $w: I \rightarrow K^*.$ The proof of proposition \ref{quantum inv.} is a straight forward calculation of the surgery presentation given above. 

Finally we prove that if we have two solutions, then they are isomorphic as modular functors through an isomorphism that preserve the duality. 

Consider a function $\alpha: I \rightarrow K^*.$ Let $\mathbf{\Sigma}$ be a labeled marked surface with labels $i_1,...,i_k.$ Then $\alpha$ induce an automorphism $\Phi_{\alpha}=\Phi$
\[
\Phi(\mathbf{\Sigma}) : Z(\mathbf{\Sigma}) \overset{\sim}{\longrightarrow} Z(\mathbf{\Sigma}), 
\]
given by $\Phi(\mathbf{\Sigma})=\left( \prod_{l=1}^k \alpha(i_l) \right) \text{id}_{Z(\mathbf{\Sigma})}.$ Since this is multiplicative on labels, it is easily seen that $\Phi$ is compatible with the action induced by morphisms of labeled marked surfaces, disjoint union and the permutation. 

We will think of $\Phi$ as a morphism of modular functors $Z^{u} \rightarrow Z^{u'}.$

We now identify sufficient conditions for $\Phi$ to be compatible with gluing and with the duality pairings. We start with gluing. Let $\mathbf{\Sigma}(\lambda)$  be obtained by gluing $\mathbf{\Sigma}(\lambda,i,i^*).$ Compatibility with gluing is equivalent with the following equation
\begin{equation*}
\tilde{g}_{u'}^i\circ \Phi(\lambda,i,i^*) = \Phi(\lambda) \circ \tilde{g}_u^i .
\end{equation*}

This is equivalent to
\begin{equation}
\label{nr2}
\alpha(i)\alpha(i^*)=\frac{u_i}{u'_i}.
\end{equation}

For $\Phi$ to be compatible with duality we must have that
\begin{equation*}
\langle \emptyarg, \emptyarg \rangle_{*,u'} = \langle \Phi(\mathbf{\Sigma})(\emptyarg), \Phi(-\mathbf{\Sigma})(\emptyarg) \rangle_{*,u}.
\end{equation*}

For this equation to be satsified we see that eqaution (\ref{nr2}) is sufficient. We see that if this is so, then $\beta=\alpha^{-1}: I \rightarrow K^*$ will satisfy
\begin{equation*}
\beta(i)\beta(i^*)=\frac{u'_i}{u_i}.
\end{equation*}

Therefore $\Phi_{\beta}$ will be an inverse morphism that preserve the duality. Thus we can choose any function $\alpha:I \rightarrow K^*$ that satisfies (\ref{nr2}), and then $\Phi=\Phi_{\alpha}$ will be an isomorphism $Z^{u} \overset{\sim}{\longrightarrow} Z^{u'}$ that preserve the duality pairing. That such a function exists follows from the fact that $u,u'$ are both invariant under $i \mapsto i^*.$ \end{proof}

\subsection{Unitarity}

Assume now that $K=\C$ and that $\mathcal{V}$ comes equipped with a unitary structure $\text{Hom}(V,W) \ni f \mapsto \overline{f} \in \text{Hom}(W,V).$

We may as well assume, that the chosen set of isomorphisms $q_i : V_{i^*} \overset{\sim}{\longrightarrow} V_i^*$ satisfies $\overline{q_i}\circ q_i= \text{id}_{V_i^*}.$ Thus $\lambda_i=1$ for all $i.$ 

\begin{thm}
\label{Canonical solution w. unitarity}
Assume $(u,w)$ is a solution to (\ref{canonical}) with $\lvert w_i \rvert =1$ for all $i.$  Then the following holds. Up to a sign the genus normalized duality pairing $\langle \emptyarg, \emptyarg \rangle_{*,u}$ is compatible with the normalized Hermitian form $( \emptyarg  \mid \emptyarg )^u.$ This sign is given by the parity of the symplectic multiplicity
\[
\rho= (-1)^{\nu}.
\]
Moreover any two solutions $(u,w)$ and $(u',w')$ to (\ref{canonical}) yields modular functors $Z^{u},Z^{u'}$ that are isomorphic through an isomorphism that respects the duality pairing as well as the Hermitian form.
\end{thm}

\begin{proof}
Consider a labeled marked surface $\mathbf{\Sigma}$ with labels $i_1,...,i_k.$ Recall the following formula from proposition \ref{RO} 
\[
\rho_{g,N}^{u,u}(\mathbf{\Sigma})= \prod_{l=1}^{m}\left( r^u_{i_l}r^u_{{i_l}^*}\right) \left( \sigma(i_l)u_{i_l} \overline{u_{{i_l}^*}}\right)^{-1}
\]

Recall that $\sigma(i)=\lambda_{i^*}\mu(i).$ So in our situation we see that the product of the $\sigma_i'$s is equal to $\mu(\mathbf{\Sigma}),$ which we already know is given by $(-1)^{\nu(\mathbf{\Sigma})}.$ Since $u_i=u_i^*$ and $\lambda_i=1$ for all $i$ we get
\[
r^u_{i_l}r^u_{{i_l}^*}= \frac{\text{dim}(i)}{\lvert u_i \rvert^2}
\]
Thus $\rho/\mu$ is seen to be a product of factors of the form
\[
\frac{\text{dim}(i)}{\lvert u_i \rvert^4}.
\]
The equation
\[
u_i^2=w_i^2 \sqrt{\text{dim}(i)},
\]
implies that all of these factors are $1.$ Here we use that $\lvert w_i\rvert=1$ for all $i.$

Assume now that $(u',w')$ is another solution. We recall that the isomorphism $Z^u \overset{\sim}{\longrightarrow}Z^{u'}$ from theorem \ref{Canonical solution} can be constructed by choosing a suitable function $\alpha: I \rightarrow \C$ with $\alpha(i)\alpha(i^*)=u_i/u'_i$ for all $i.$ For any labeled marked surface $\mathbf{\Sigma}$ the isomorphism 
\[
\Phi_{\alpha}:Z^u(\mathbf{\Sigma}) \overset{\sim}{\longrightarrow} Z^{u'}(\mathbf{\Sigma}),
\] will be multiplication by $\alpha(i_1)\cdots \alpha(i_k)$ where $i_1,...,i_k$ are the labels of $\mathbf{\Sigma}.$ However, since $\lvert w_i \rvert =\lvert w'_i \rvert =1$ for all $i$ we see that $\lvert u_i \rvert =\lvert u'_i \rvert$ for all $i.$ This implies the following two things. First $r^u_i=r^{u'}_i$ for all $i.$ Second it implies that we can choose $\alpha(i)=\alpha(i^*)$ to be a square root of $u_i/u'_i$ which lies on the unit circle. Therefore $\Phi_{\alpha}$ will be a Hermitian isomorphism.
 \end{proof}

\section{The dual of the fundamental group of a modular tensor category}
\label{FG of a MTC}

\label{Thedualofthefundamentalgroup}

Recall the definition of the dual of the fundamental group and a fundamental symplectic character of a modular tensor category given in the introduction.

\begin{thm}\label{strict}
Assume that a modular tensor category $({\mathcal V}, I)$ has a fundamental symplectic character. Then there exists $u: I \rightarrow K^*$ such that the genus normalized duality pairing $\langle \emptyarg , \emptyarg \rangle_{*,u}$ is strictly self-dual, strictly compatible with gluing and we have that
\[
Z(u,i)= \frac{\text{dim}(i)}{k_i}.
\]
Moreover if $\mathcal{V}$ is unitary and the image of $\tilde{\mu}$ above is a subset of 
$$S(K)= \{z\in K| z \bar z = 1\},$$ 
then we can choose $u$, such that $\langle \emptyarg, \emptyarg \rangle_{*,u}$ is strictly compatible with the Hermitian form $(\emptyarg \mid \emptyarg )^u.$
\end{thm}

Before commencing the proof, we remark that if $\mu(i)=1$ for all self-dual objects $i,$ then the neutral element $e \in \Pi(\mathcal{V},I)^*$ is such an extension. 

\begin{proof}

Partition $I=I_{SD} \sqcup I_{NSD}$ such that $i \in I_{SD} $ if and only if $i^*=i$. We further have the natural splitting
$$ I_{SD} = I_{SD}^+ \sqcup I_{SD}^-,$$
where $i\in I_{SD}^+$ if and only if $\mu(i) = 1$. Hence we have of course that $i\in _{SD}^-$ if and only if $\mu(i) = -1$.

Let us now pick a splitting
$$ I_{NSD} = I_{NSD}^1 \sqcup I_{NSD}^2,$$
such that $i\in I_{NSD}^1 $ if and only if $i^* \in I_{NSD}^2.$

We start by describing the normalization. We start by choosing $q_i,q_{i^*}$ such that $\mu(i)=\mu(i^*)=1$ whenever $i\not=i^*.$ Let $w:I \rightarrow K^*.$ We can then scale the pairing by using $q(w)$ in the isomorphism $Z(-\mathbf{\Sigma}) \overset{\sim}{\rightarrow} \mathcal{T}(-\overline{\mathcal{G}(\mathbf{\Sigma})}).$ Then we have $\mu(i,w)=\frac{w_{i}}{w_i^*}\mu(i).$ Since $\tilde{\mu}(i)={\tilde{\mu}(i^*)}^{-1}$ this implies that we can consistently choose $w_i,w_i^*$ such that $\mu(i,w)=\tilde{\mu}(i)$ whenever $i$ is not self-dual. Since $\tilde{\mu}$ is assumed to extend the $\mu$ on the self-dual objects we conclude that we can normalize such that $\mu(i,w)=\tilde{\mu}(i).$ Assume that we can choose $w$ such that $\langle \emptyarg \mid  \emptyarg\rangle_{w,*}$ is also strictly compatible with gluing. We want to argue that in this case, we have $\mu(\mathbf{\Sigma})=1$ unless $Z(\mathbf{\Sigma})=\mathbf{0}.$

Let $\mathbf{\Sigma}$ be a labeled marked surface. We recall that to prove strict self-duality is the same as proving that for all $(u,w) \in Z(\mathbf{\Sigma})\times Z(-\mathbf{\Sigma})$ we have
\[
\langle u,w  \rangle_{\mathbf{\Sigma}}=\langle w,u \rangle_{-\mathbf{\Sigma}}.
\] 
Let $C$ be a collection of simple closed curves on $\mathbf{\Sigma}$, whose homology classes are contained in the Lagrangian subspace of $\mathbf{\Sigma}$ and such that factorization along all of these will produce a disjoint union of spheres with one, two or three marked points.   For the existence of such a collection see \cite{W}. Let $\lambda \in I^{C}$ and let $\mathbf{\Sigma}_C(\lambda)$ be the labeled marked surface obtained from factorization in $C$. Thus $\mathbf{\Sigma}_C(\lambda)$ is a disjoint union of labeled marked surfaces of genus zero with one, two or three labels. Write $\mathbf{\Sigma}_C(\lambda)=\sqcup_{l=1}^k \mathbf{S}_l(\lambda).$  Let $P_{\lambda}: Z(\Sigma) \rightarrow Z(\mathbf{\Sigma}_C(\lambda))$ be the projection resulting from the factorization isomorphism.  
Let $u \in Z(\mathbf{\Sigma})$ and let $w \in Z(-\mathbf{\Sigma}).$ We can write $ P_{\lambda}(u)$ as a finite sum $\sum_{\alpha \in (u,\lambda)} u(\alpha,\lambda)^{(1)}\otimes \cdots \otimes u(\alpha,\lambda)^{(l)}$ with $u(\alpha,\lambda)^{(i)} \in Z(\mathbf{S}_i(\lambda)).$ Here $(u,\lambda)$ is a finite index set depending only on $u$ and $\lambda.$ In a similar way we write $ P_{\lambda^*}(w)$ as a finite sum of the form $\sum_{\beta \in (w,\lambda^*)} w(\beta,\lambda^*)^{(1)}\otimes \cdots \otimes w(\beta,\lambda^*)^{(l)}.$ Recall the following identity: $-(\mathbf{\Sigma}(\lambda))=(-\mathbf{\Sigma})(\lambda^*).$ We have that
\begin{align*}
\langle u, w \rangle_{\mathbf{\Sigma}}&= \sum_{\lambda \in I^C} {\langle P_{\lambda}(u), P_{\lambda^*}(w) \rangle_{\mathbf{\Sigma}_C(\lambda)}}
\\  &= \sum_{\lambda \in I^C} \sum_{\alpha \in (u,\lambda),\beta \in (w,\lambda^*)} \prod_{l=1}^k {\langle u(\alpha,\lambda)^{(l)}, w(\beta,\lambda^*)^{(l)} \rangle_{\mathbf{S}_l(\lambda)}}
\end{align*}

 Similarly we see that
\[
\langle w, u \rangle_{-\mathbf{\Sigma}} =  \sum_{\lambda^* \in I^C} \sum_{\alpha \in (u,\lambda), \beta \in (w,\lambda^*)} \prod_{l=1}^k {\langle w(\beta,\lambda^*)^{(l)},u(\alpha,\lambda)^{(l)} \rangle_{(-\mathbf{S}_l)(\lambda^*)}}
\]
Therefore we see that it reduces to the case of spheres marked with one, two or three points. If a sphere is marked with one point its module of states is zero unless the point is labeled with $0,$ but we already saw that $\mu(0,w)=1.$ If it is marked with two points then we use that its associated module of states is zero unless its labels are $i,i^*.$ If this is the case then the desired equality follows from $\tilde{\mu}(i)\tilde{\mu}(i^*)=1.$  Finally, for a sphere with three points labeled by $i,j,k$, we recall that the associated module of states is isomorphic to $\text{Hom}(\mathbf{1},V_i \otimes V_j \otimes V_k)$ which is zero unless $\tilde{\mu}(i)\tilde{\mu}(j)\tilde{\mu}(k)=1.$ 

Therefore it amounts to choosing $(u,w)$ such
\begin{equation}
\label{nummer1}
u_i=\frac{\sqrt{w_iw_i^*}\sqrt{\text{dim}(i)}}{\sqrt{u_iu_{i^*}}}w_i,
\end{equation}
and

\begin{equation}
\label{nummer2}
\mu(i)\frac{u_i}{u_{i^*}}=\tilde{\mu}(i).
\end{equation}

Choose a square root of $\tilde{\mu}(i)$ and a square root of $\mu(i)$ for each $i.$ This can be done consistently such that $\sqrt{\tilde{\mu}(i^*)}=1/\sqrt{\tilde{\mu}(i)},$ for $i\neq i^*$. Now define $\eta: I \rightarrow K^*$ by
\[
\eta_i:= \sqrt{\tilde{\mu}(i)} \sqrt{\mu(i)}.
\]
With such a choice we have for all $i\in I$ 
\[
\eta_i \eta_{i^*}=1.
\]
This implies the important equation

\begin{equation}
\label{eta}
\frac{\eta_i}{\eta_{i^*}}=\tilde{\mu}(i)\mu(i).
\end{equation}

Take $w_i=\eta_i$ for all $i.$ Consider $i \in I_{NSD}^1.$ Fix a choice $\sqrt{\tilde{\mu}(i^*)}$ and then solve 
\[
u_i^2 = \frac{\sqrt{\text{dim}(i)}}{\sqrt{\tilde{\mu}(i^*)}}\eta_i.
\]

Therefore, if we define $u_{i^*}=u_i \tilde{\mu}(i^*)$ then equation (\ref{nummer1}) is true for $i,$ since we may choose $\sqrt{u_iu_{i^*}}=u_i\sqrt{\tilde{\mu}(i^*)}$ in this case. We now need to check that equation (\ref{nummer1}) is also true for $i^*.$ We can choose $\sqrt{u_iu_{i^*}}=u_{i^*}\sqrt{\tilde{\mu}(i)},$ with  $\sqrt{\tilde{\mu}(i)}=1/\sqrt{\tilde{\mu}(i^*)}$ Then we must check

\[
u_{i^*}^2=\frac{\sqrt{\text{dim}(i)}}{\sqrt{\tilde{\mu}(i)}}\eta_{i^*}
\]
We have that
\begin{align*}
u_{i^*}^2 &= u_i^2 \tilde{\mu}(i^*)^2
\\&=  \tilde{\mu}(i^*)u_i^2\tilde{\mu}(i)^{-1}
\\&= \frac{\tilde{\mu}(i^*)}{\sqrt{\tilde{\mu}(i^*)}} \sqrt{\text{dim}(i)}\eta_i \tilde{\mu}(i)^{-1}
\\&=\sqrt{\tilde{\mu}(i^*)}\sqrt{\text{dim}(i)}\eta_{i^*}
\\&=\frac{1}{\sqrt{\tilde{\mu}(i)}}\sqrt{\text{dim}(i)}\eta_{i^*}.
\end{align*}
Thus (\ref{nummer1}) holds for all $j \in I_{NSD}.$ For $i\in I_{SD}$ we have $\eta_i=1$ and it is easy to choose $u_i$ satisfying (\ref{nummer1}). That (\ref{nummer2}) holds is an easy consequence of equation (\ref{eta}).

That $Z(i)=\frac{\text{dim}(i)}{k_i}$ follow as in the proof of theorem \ref{Canonical solution}.

Finally, assume that $(\mathcal{V},I)$ is unitary. As above we may assume $\lambda_i=1$ for all $i.$ Observe that $\eta_i \in S^1$ for all $i.$ Therefore $\lvert u_i \rvert=\lvert u_{i^*} \rvert =\frac{1}{\text{dim}(i)^4}.$ According to proposition \ref{RO} we know $\rho$ is given by
\[
\rho_{g,N}^{u,u}(\mathbf{\Sigma})= \prod_{l=1}^{m}\left( r^u_{i_l}r^u_{{i_l}^*}\right) \left( \sigma(i_l)u_{i_l} \overline{u_{{i_l}^*}}\right)^{-1}.
\]
We have
\[
r^u_ir^u_{i^*}= \frac{\text{dim}(i)}{\lvert u_i \rvert \lvert u_{i^*}\rvert}.
\]
Using $\sigma(i)=\mu(i)$ we see that
\[
\sigma(i)u_i \overline{u_{i^*}}= \mu(i)\frac{u_i}{u_{i^*}}u_{i^*}\overline{u_{i^*}}=\tilde{\mu}(i)\lvert u_{i^*}\rvert ^2.
\]
Thus we get
\[
\rho_{g,N}^{u,u}(\mathbf{\Sigma})=\prod_{l=1}^m \tilde{\mu}(i_l^*).
\]
The argument given above proves that this is $1$ unless $Z(\mathbf{\Sigma})=0.$
 \end{proof}

\begin{cor}
Assume $\tilde{\mu}$ is as above. Assume a labeled marked $\mathbf{\Sigma}$ surface has labels $i_1,...,i_k.$ We see that $\prod_{l=1}^l \tilde{\mu}(i_l)\not=1$ implies $Z(\mathbf{\Sigma})=\mathbf{0}.$
\end{cor}

\section{The quantum $\mbox{SU}(N)$ modular tensor categories}
\label{SU(N)}

We refer to the \cite{TW1}, \cite{TW2} and \cite{B} (which uses the skein theory model for the $\mbox{SU}(2)$ case build in \cite{BHMV1}, \cite{BHMV2}) for the complete construction of the quantum $\mbox{SU}(N)$ modular tensor category $H^{\rm{SU}(N)}_k$ at the root of unity $q= e^{2\pi i /(k+N)}$. For a short review see also \cite{AU4}. The simple objects of this category are indexed by the following set of young diagrams 
$$ \Gamma_{N,k} =\{ (\lambda_1,\ldots,\lambda_p)\mid \lambda_1 \leq k, p < N\}.$$
The involution $\dagger : \Gamma_{N,k} \rightarrow \Gamma_{N,K}$ is defined as follows.
For a Young diagram $\lambda$
in $\Gamma_{N,k}$ we define
  $\lambda^\dagger \in\Gamma_{N,k}$ to be the Young diagram
  obtained from the skew-diagram $(\lambda_1^N)/\lambda$ by rotation
  as indicated in the figure below. 
  \begin{center}
    \includegraphics[scale=0.8]{andersen15aa.mps}
  \end{center}
Let $\mu = e^{2\pi i/N}$ and $\zeta_N$ be the set of $N$'th roots of $1$ in ${\mathbb C}$. We then consider the following map
$$ \tilde \mu : \Gamma_{N,k} \rightarrow \zeta_N$$
given by
$$\tilde \mu(\lambda) = \mu^{|\lambda|},$$
where $|\lambda| = \lambda_1 + \ldots + \lambda_p$.

\begin{pro}
We have that
$$\tilde \mu \in \Pi(H^{\rm{SU}(N)}_k, \Gamma_{N,k})^*.$$
\end{pro}

\begin{proof}
We have 
$$\tilde\mu(\lambda)\tilde \mu(\lambda^\dagger) = 1,$$
since $|\lambda| + |\lambda^\dagger| = N \lambda_1$ by construction. Now consider $\lambda, \mu, \nu \in  \Gamma_{N,k}$. By the very definition $H^{\rm{SU}(N)}_k(0,\lambda\otimes\mu\otimes\nu)=0$ if 
$$|\lambda| + |\mu| + |\nu| \not\in N {\mathbb Z},$$
since the $|\lambda| + |\mu| + |\nu|$ ingoing strands at the top of the cylinder over the disc can only disappear into coupons $N$ at the time inside the cylinder, since we have the empty diagram at be bottom determined by the label $0$.

\end{proof}
Using the notation in \cite{AU4}, we will now fix isomorphism
$$ q_\lambda \in H^{\rm{SU}(N)}(\lambda^\dagger,\lambda^*)$$
as indicated in figure below (illustrated for some particular element $\lambda\in \Gamma_{6,k}$ for $k\geq 5$)
\vskip.3cm
\begin{center}
    \includegraphics[scale=1.7]{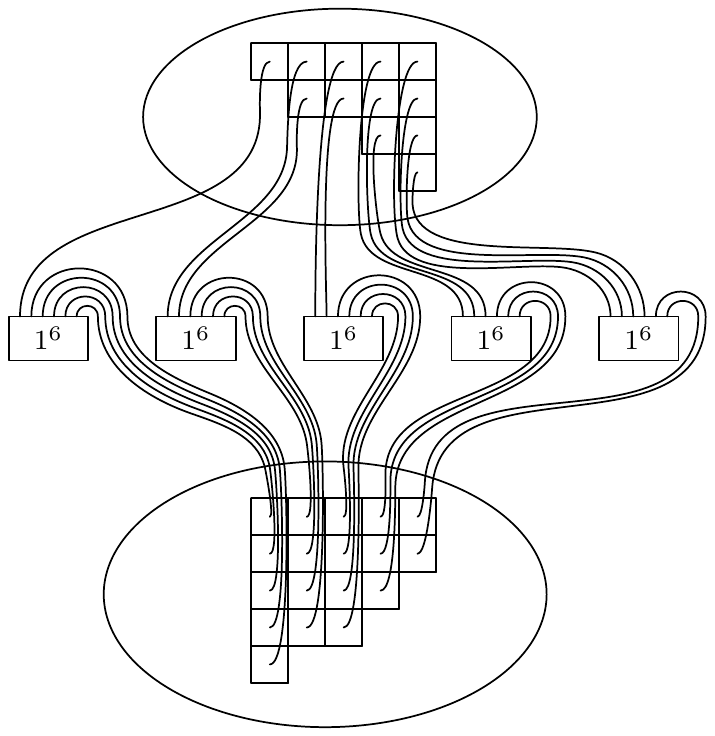}
  \end{center}

This gives us $\mu : \Gamma_{N,k} \rightarrow {\mathbb C}^*$
such that
$$ F(q_\lambda) = \mu(q) q_{\lambda^\dagger}.$$

\begin{pro} For $N$ odd and any $\lambda \in \Gamma_{N,k}$, we have that
$$ \mu(\lambda) =1.$$
For $N$ even and any $\lambda \in \Gamma_{N,k}$ we have that
$$\mu(\lambda) = (-1)^{|\lambda|}.$$
\end{pro}

\begin{proof}
We observe that if we apply $F$ to $q_\lambda$, top and bottom can by a half rotation be brought into the right position for comparison with $q_{\lambda^\dagger}$ and the relevant computation for each coupons in between is the following
\begin{center}
\includegraphics[scale=1]{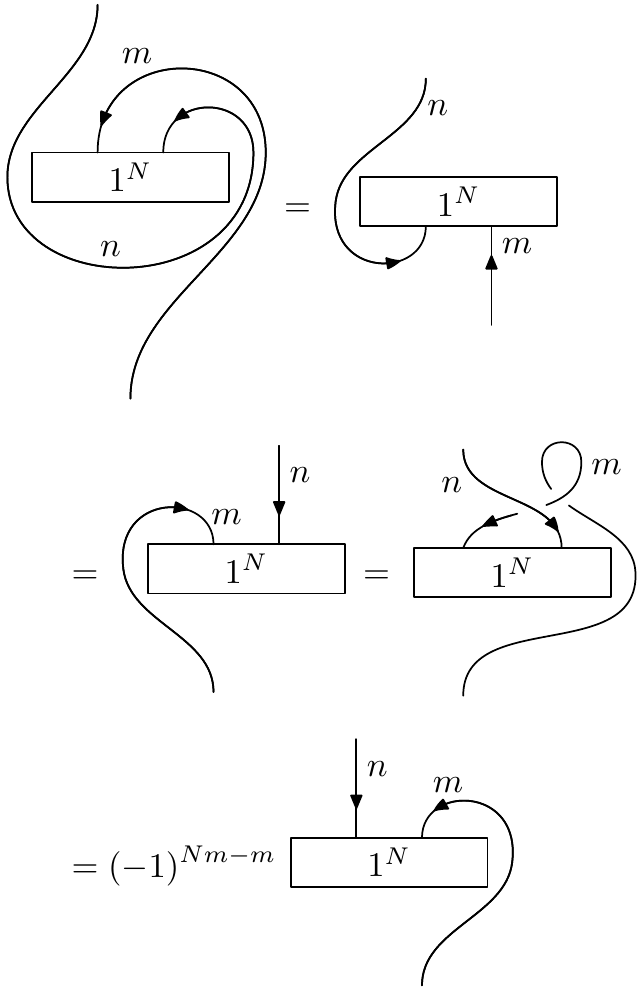} 
\end{center}
Here the last sign is a result of the following calculation in the notation of \cite{B}, using that
$$ a = q^{-\frac{1}{2N}}, v = q^{-\frac{N}{2}}, s=q^{\frac12},$$
namely, the braiding and the twist on top of the coupon contributes
$$ (-a^{-1}s)^{nm + m(m-1)} (a^{-1}v)^m = (-1)^{Nm -m}$$
times the coupon with the strands in the original position again.

\end{proof}

From this proposition, we observe that if $N$ is odd or $N$ is divisible by $4$, then there are no self-dual symplectic objects in $H^{\rm{SU}(N)}_k$. If however, $N$ is even, but $N/2$ is odd, then all self-dual objects are symplectic, since they have $N/2$ boxes. Moreover, we observe that on these self-dual objects
$$ \tilde \mu (\lambda) = -1.$$ In all cases, we see that $\tilde \mu$ is a fundamental symplectic character. 

\section{The general quantum group modular tensor categories}
\label{g}

We will now fix a simple Lie algebra $\Lie{g}$ and we will consider the correspond quantum group at the root  of unity $q = e^{2\pi i /(k+h)}$. The associated modular tensor category will be denoted $(H^{\Lie{g}}_k, \Lambda^{\Lie{g}}_k)$.
Let ${\mathcal W}$ be the weight lattice and ${\mathcal R}$ the root lattice for $\Lie{g}$. We recall that the fundamental group of $\Lie{g}$ is $\Pi(\Lie{g}) = {\mathcal W}/{\mathcal R}$. We have three general facts about $\Pi(\Lie{g})$. The first one is that
$$\lambda + \lambda^\dagger =0 \mbox{ } \rm{ mod } \mbox{ }{\mathcal R}$$
for all dominant weights $\lambda\in {\mathcal W}_+$ and $\lambda^\dagger = - w_0(\lambda)$, where $w_0$ is the longest element of the Weyl group, e.g. $\lambda^\dagger$ is the highest weight vector of the dual of the irreducible representation $V_\lambda$, corresponding to $\lambda$.
We further observe that if $V_\lambda$ is self-dual, then $2\lambda$ will be in ${\mathcal R}$.

The second fact is that if we know that for $\lambda,\mu,\nu\in {\mathcal W}_+$
$$\rm{Hom}_G(0, V_\lambda\otimes\V_\mu\otimes V_\nu) \neq 0,$$
then 
$$\lambda+\mu+\nu \neq 0 \mbox{ } \rm{ mod } \mbox{ }{\mathcal R}.$$

We recall that 
$$\rm{Hom}_G(0, V_\lambda\otimes\V_\mu\otimes V_\nu) = 0$$
implies that
$$H^{\Lie{g}}_k(0, V_\lambda\otimes\V_\mu\otimes V_\nu) = 0.$$
The corresponding property for the modular functor coming from Conformal Field Theory (see \cite{AU2}) is clear by construction. 

We now recall that $\Pi(\Lie{g})$ is cyclic unless $\Lie{g} = D_n$, where $\Pi(\Lie{g}) = {\mathbb Z}_2\times {\mathbb Z}_2$. In the last case, one knows that $(1,0)$ and $(0,1)$ are symplectic, but $(1,1)$ is not.
From this we conclude that in all cases, we see that there exist some even $N$ and a homeomorphism
$$ \tilde \mu' : \Pi(\Lie{g}) \rightarrow \zeta_N,$$
such that $\tilde \mu'(\lambda) = -1$ if and only if $\lambda$ is symplectic.
But then we define
$$\tilde \mu : \Lambda^{\Lie{g}}_k \rightarrow \zeta_N$$
to be the composite of the projection from $ \Lambda^{\Lie{g}}_k$ to $\Pi(\Lie{g})$ followed by $\tilde \mu'$.

We then see that $\tilde \mu \in \Pi(H^{\Lie{g}}_k)^*$ and indeed it is a fundamental symplectic character.

We remark that the result of this section applied to $\Lie{g} = \Lie{sl}(N)$ gives a second proof for the exists of a fundamental symplectic character in that case.


\begin{thebibliography}{99}

\bibitem[AU1]{AU1} J.E. Andersen \& K. Ueno, "Abelian Conformal Field Theories and
Determinant Bundles", International Journal of Mathematics. {\bf
18}, (2007) 919--993.

\bibitem[AU2]{AU2} J.E. Andersen \& K. Ueno, "Geometric Construction of Modular Functors
from Conformal Field Theory", Journal of Knot theory and its
Ramifications. {\bf 16} 2 (2007), 127--202.

\bibitem[AU3]{AU3} J.E. Andersen \& K. Ueno, "Modular functors are determined by
their genus zero data", Quantum Topology {\bf 3} (2012), 255--291.

\bibitem[AU4]{AU4} J.E. Andersen \& K. Ueno, "Construction of the Witten-Reshetikhin-Turaev TQFT from conformal field theory". Invent. Math. {\bf 201 (2)} (2015), 519--559. 

\bibitem[B]{B} C. Blanchet, Hecke algebras, modular categories and $3$-manifolds quantum
   invariants, {\em Topology} 39 no. 1, 193--223, 2000.


\bibitem[BHMV1]{BHMV1} C. Blanchet; N. Habegger; G. Masbaum; P. Vogel,
Three-manifold invariants derived
   from the Kauffman bracket, {\em Topology} 31  no. 4, 685--699, 1992.

\bibitem[BHMV2]{BHMV2} C. Blanchet; N. Habegger; G. Masbaum; P. Vogel,
Topological quantum field theories
   derived from the Kauffman bracket, {\em Topology} 34  no. 4, 883--927, 1995.


\bibitem[G]{G} J. Grove, "Constructing TQFTs from modular functors", J. Knot Theory Ramifications {\bf 10} (2001), no. 8, 1085--1131.


\bibitem[M]{M} S. Mac Lane, Categories for the Working Mathematician. Springer, Second Edition 1997.

\bibitem[W]{W} K. Walker, On Witten's 3-manifold invariants. Preprint 1991. Preliminary version $\#$2 tqft.net/other-papers/KevinWalkerTQFTNotes.pdf

\bibitem[RT1]{RT1}
N.~Reshetikhin and V.~G. Turaev.
\newblock Invariants of {$3$}-manifolds via link polynomials and quantum
  groups.
\newblock {\em Invent. Math.}, 103(3):547--597, 1991.

\bibitem[RT2]{RT2}
N.~Y. Reshetikhin and V.~G. Turaev.
\newblock Ribbon graphs and their invariants derived from quantum groups.
\newblock {\em Comm. Math. Phys.}, 127(1):1--26, 1990.

\bibitem[TW1]{TW1} V. Turaev \& H. Wenzl, Quantum invariants of 3-manifolds associated with classical simple Lie algebras,
{\em International Journal of Mathematics} 4, 323 -- 358, 1993.

\bibitem[Tu]{Tu} V.G. Turaev, Quantum invariants of knots and $3$-manifolds. De Gruyter Studies in Mathematics $18.$ Walter de Gruyter $\&$ Co., Berlin, 1994. 



\bibitem[TW2]{TW2} V. Turaev \& H. Wenzl, Semisimple and modular categories from link invariants,
{\em Masthematische Annalen} 309, 411 -- 461, 1997.


\end{thebibliography}
\end{document}